\documentclass[a4paper,10pt,reqno]{amsart}
\usepackage{mathtools}
\usepackage{amsfonts}
\usepackage{amssymb}
\usepackage{amsrefs} 
\usepackage{mathrsfs}
\usepackage{stmaryrd}
\usepackage{amsmath}
\usepackage{amsthm}
\usepackage[shortlabels]{enumitem}
\usepackage{nomencl}
\usepackage[bookmarks=true]{hyperref} 
\usepackage{esint}
\usepackage{dsfont}
\usepackage{overpic}
\usepackage{upgreek}
\usepackage[dvipsnames]{xcolor}
\usepackage{slashed}
\usepackage{scalerel}
\usepackage{comment}
\usepackage[utf8]{inputenc}
\usepackage{vmargin}
\usepackage[textwidth=20mm, textsize=tiny]{todonotes}
\usepackage{faktor}

\hypersetup{
colorlinks,
linkcolor={red!50!black},
citecolor={blue!80!black},
urlcolor={blue!80!black}
}

\author{Christian Bär}
\author{Lashi Bandara}
\title[First-order elliptic noncompact boundary value problems]
{First-order elliptic boundary value problems on manifolds with noncompact boundary}
\date{\today}
\address{Christian Bär, 
Institut für Mathematik,
Universität Potsdam, 
D-14476, Potsdam, Germany
}
\urladdr{\href{https://www.math.uni-potsdam.de/baer}{https://www.math.uni-potsdam.de/baer}}
\email{\href{mailto:cbaer@uni-potsdam.de}{cbaer@uni-potsdam.de}}

\address{Lashi Bandara,
Deakin University, 
Melbourne Burwood Campus, 
221 Burwood Highway, 
Burwood, 
Victoria, 
Australia, 
3125
}
\urladdr{\href{http://www.lashi.org/}{http://www.lashi.org/}}
\email{\href{mailto:lashi.bandara@deakin.edu.au}{lashi.bandara@deakin.edu.au}}


\def\colour{\colour}

\newcommand{\sym}{\upsigma}
\newcommand{\spec}{\mathrm{spec}}
\newcommand{\res}{\mathrm{res}}
\newcommand{\rest}[1]{{{\lvert_{}}_{}}_{#1}}
\newcommand{\R}{\mathbb{R}}
\newcommand{\C}{\mathbb{C}}
\newcommand{\N}{\mathbb{N}}
\newcommand{\cbrac}[1]{\left(#1\right)}

\newcommand{\dbrac}[1]{\left\{#1\right\}}
\newcommand{\ad}{\ast} 

\newcommand{\interior}[1]{\mathring{#1}}

\newcommand{\id}{\mathrm{id}}

\newcommand{\dtilde}[1]{\tilde{\tilde{#1}}}

\newcommand{\modulus}[1]{|#1|}

\newcommand{\norm}[1]{\| #1 \|}

\newcommand{\set}[1]{\dbrac{#1}}
 
\newcommand{\Lp}[2][{}]{{\rm L}^{#2}_{\rm #1}}
\newcommand{\Ck}[2][{}]{{\rm C}^{#2}_{\rm #1}}
\newcommand{\Dk}[2][{}]{{\rm D}^{#2}_{\rm #1}}
\newcommand{\Hard}[2][{}]{{\rm H}^{#2}_{\rm #1}}
\newcommand{\SobH}[2][{}]{\Hard[#1]{\rm #2}}
\newcommand{\OpSobH}[2][{}]{{\mathcal{H}}^{#2}_{\rm #1}}
\newcommand{\checkH}{\check{\mathrm{H}}}
\newcommand{\hatH}{\hat{\mathrm{H}}}
\newcommand{\Hil}{\mathscr{H}}
\newcommand{\Spa}{\mathscr{X}}
\newcommand{\core}{\mathscr{C}}
\newcommand{\rc}{\mathrm{c}}

\newcommand{\embed}{\hookrightarrow}
\newcommand{\ceil}[1]{\lceil #1 \rceil}

\newcommand{\nul}{\mathrm{ker}}
\newcommand{\ran}{\mathrm{ran}}
\newcommand{\dom}{\mathrm{dom}}
\newcommand{\Ric}{\mathrm{Ric}}
\newcommand{\KK}{\mathscr{K}}

\newcommand{\intersect}{\cap} 
\newcommand{\close}[1]{\overline{#1}} 
\newcommand{\union}{\cup}

\DeclareMathOperator{\spt}{spt}
\DeclareMathOperator{\sgn}{sgn}

\DeclareMathOperator{\End}{End} 

\newcommand{\inprod}[1]{\left\langle #1 \right\rangle}
\newcommand{\ext}{\mathcal{E}}

\newcommand{\embeds}{\hookrightarrow}

\newcommand{\DDD}{\mathscr{D}}
\newcommand{\AAA}{\mathscr{A}}
\newcommand{\RRR}{\mathscr{R}}

\newcommand{\APS}{\mathrm{APS}}
\newcommand{\Match}{\mathrm{Match}}

\renewcommand{\epsilon}{\varepsilon}

\newcommand{\Hinfty}{$\mathrm{H}^\infty$}
\newcommand{\dM}{\partial M}
\renewcommand{\div}{\mathrm{div}}

\newtagform{simple}{}{}{}

 \newtheorem{theorem}{Theorem}[section]
 \newtheorem{lemma}[theorem]{Lemma}
 \newtheorem{corollary}[theorem]{Corollary}
 \newtheorem{proposition}[theorem]{Proposition}

 \theoremstyle{definition}
 \newtheorem{definition}[theorem]{Definition}
 
 \newtheorem{remark}[theorem]{Remark}
 \newtheorem{assumptions}[theorem]{Assumptions}

\begin{document}

\begin{abstract}
We consider first-order elliptic differential operators acting on vector bundles over smooth manifolds with smooth boundary, which is permitted to be noncompact.
Under very mild assumptions, we obtain a regularity theory for sections in the maximal domain.
Under additional geometric assumptions, and assumptions on an adapted boundary operator, we obtain a trace theorem on the maximal domain.
This allows us to systematically study both local and nonlocal boundary conditions.
In particular, the Atiyah-Patodi-Singer boundary condition occurs as a special case.
Furthermore, we study contexts which induce semi-Fredholm and Fredholm extensions.
\end{abstract}
\maketitle

\tableofcontents

\parindent0cm
\setlength{\parskip}{\baselineskip}

\section{Introduction}

Boundary value problems are fundamental to engineering and physics and have been studied extensively in that context.
Historically, the study of boundary value problems has been carried out on Euclidean domains (or more generally in the setting of smooth manifolds with smooth compact boundary) with a focus on \emph{local} boundary conditions.
There have been many significant generalisations - nonsmooth boundary and nonsmooth coefficients for second-order operators in divergence form as well as in the nondivergence form setting.

In the last fifty years, boundary value problems for first-order elliptic operators have become important in geometry, exemplified through the study of scalar curvature via the use of the Dirac operator on spin manifolds.
A description of the maximal domain was obtained by Seeley \cite{Seeley} in the 1960s via the use of Calderón projectors.
While this was useful in the study of PDEs in geometric settings, it was the identification of the APS boundary condition by Atiyah, Patodi and Singer which made a profound impact on index theory.
This is highlighted by their celebrated index theorem  for Dirac-type operators on manifolds with smooth, compact boundary in in \cites{APS-Ann,APS1, APS2, APS3}.

The identification of the APS boundary condition was not only useful from an index theory perspective, but  it also accounted for ``half'' of all possible boundary conditions.
This was understood much later, whereby on taking a sum of the APS boundary condition with its ``anti-APS'' counterpart, a trace theorem on the entire maximal domain of the operator was obtained.
Such a description allows for the study of \emph{all} boundary conditions and is useful in index theory as it permits for the study of continuous deformations  of boundary conditions.

The ability to describe all boundary conditions requires the study of spectral projectors associated with an \emph{adapted boundary operator} on the boundary.
This is a perspective that matured and  culminated in the work of Bär-Ballmann \cite{BB12} and later Bär-Bandara \cite{BBan}.
The former paper assumes that the bounded adapted operator can be chosen selfadjoint, essentially restricting the setting to Dirac-type operators.
In the latter paper, this restriction is removed so that there are  no additional assumptions on the operator other than requiring it to be  first-order elliptic.
An alternative approach is taken in \cite{BGS} allowing for general-order elliptic differential operators on compact manifolds. 

The analysis in \cite{BB12} and earlier in \cite{APS1} utilised Fourier circle methods that are possible due to selfadjointness of the adapted boundary operator and the presence of discrete spectrum.
In contrast, the analysis carried out in \cite{BBan}  differs significantly as adapted boundary operators may no longer be selfadjoint.
The Fourier circle perspective here was replaced by obtaining an \Hinfty-functional calculus for the adapted boundary operator.

In this paper, we dispense with the requirement that the boundary is compact, only assuming that it is smooth.
In this direction, there have been a number of developments, which either assume strong conditions on the underlying geometry or modify the operator through adding a potential to gain a handle on its spectral theory.
In \cite{GN} Große and Nakad consider boundary value problems for noncompact boundary, including applications to index theory.
These results are obtained for local boundary conditions under the assumption of bounded geometry.
In \cites{BS0, BS1, BS2, Shi}, Braverman and Shi investigate both local and nonlocal boundary value problems for Callias-type operators for noncompact boundary.

In our work, under very minimal assumptions which we call the standard setup (cf.\ Subsection~\ref{StdSetup}), we obtain regularity theory for sections in the maximal domain (Theorems~\ref{Thm:Trace1}~and~\ref{Thm:Reg}).
We then move on to describe the maximal domain of the operator.
To account for the lack of compactness of the boundary, we impose geometric conditions which are natural and are much weaker than bounded geometry.
However, we are forced to consider selfadjoint adapted boundary operators, mirroring the assumption made in \cite{BB12}.
This still allows for a very large class of operators including all Dirac-type operators.
These assumptions are what we call the geometric setup and are listed in Subsection~\ref{subsec:GeometricSetup}.

Due to the fact that the boundary is allowed to be noncompact, this operator may have the entire real line as spectrum.
Therefore, Fourier circle methods cannot be used in the analysis.
Instead, we perform the analysis in the same vein as \cite{BBan}, alluding to the \Hinfty-functional calculus of the adapted boundary operator.
This posits the methods and ideas utilised in this paper closer to those arising from real-variable harmonic analysis.

Mirroring the trace results for the maximal domain in \cite{BB12} and \cite{BBan}, Theorem~\ref{Thm:MaxDom} establishes a trace theorem on the maximal domain.
As in the compact boundary case, this allows for an understanding of \emph{all} boundary conditions, including those that are nonlocal.
Moreover, this allows us to define the APS boundary condition in the noncompact boundary setting.

In Section~\ref{Sec:Freddy}, we study  semi-Fredholm and Fredholm extensions, despite the noncompactness of the boundary.
We identify an appropriate  notion of coercivity  (cf.\ Definition~\ref{Def:Coercive}) with respect to a boundary condition.
When such a boundary condition is sufficiently regular, we show that operators satisfying this notion of coercivity are semi-Fredholm.
In particular, the APS boundary condition is a Fredholm one when the operator is coercive.

Finally, in Section~\ref{Sec:DiracType}, we study applications to Dirac-type operators.
In particular, we show how the minimal and geometric setup can be satisfied under natural curvature assumptions. 
Together with the results obtained in Section~\ref{Sec:Freddy}, we show Fredholm extensions for Dirac-type operators for the APS boundary condition as well as chiral boundary conditions.
In particular, this includes the setup considered by Große and Nakad in \cite{GN}.
We also consider Callias-type perturbations of Dirac-type operators in the spirit of Braverman and Shi in \cites{BS0, BS1, BS2}, extending some of their results.

\textbf{Acknowledgements.}
This work has been financially supported by the priority programme SPP~2026 ``Geometry at Infinity'' funded by Deutsche Forschungsgemeinschaft.
L.B.\ thanks the University of Potsdam for its hospitality.
Both authors thank the anonymous referee whose suggestions helped improve the paper.

\section{Setup and statement of main results}

\subsection{Notation}
Throughout, we use the analysts' inequality $a \lesssim b$ to mean that there exists a constant $C < \infty$ such that $a \leq C b$.
The objects $a$ and $b$ will be quantified and the dependency of $C$ will often be clear from context.
Otherwise, it will be spelt out.
By $a \simeq b$, we mean that $a \lesssim b$ and $b \lesssim a$.

On Banach spaces $\Spa_1, \Spa_2$, we consider $T: \Spa_1 \to \Spa_2$ to typically be an unbounded operator.
The domain on which it acts will be written as $\dom(T) \subset \Spa_1$, with its range denoted by $\ran(T) \subset \Spa_2$.
Its kernel is given by $\nul(T)$.
When $\ran(T) = \Spa_2$, we say that the operator $T$ is \emph{invertible} if it is injective and has a bounded inverse.
An operator $T$ is \emph{densely-defined} if $\dom(T)$ is dense in $\Spa_1$ and it is \emph{closed} if its graph is closed in $\Spa_1 \times \Spa_2$.
When $\Spa = \Spa_1 = \Spa_2$, let $\res(T)$ be the \emph{resolvent set} consisting of  $\zeta \in \C$ such that $(\zeta - T)$ is invertible.
The \emph{spectrum} is $\spec(T) = \C \setminus \res(T)$.

For $M$ a smooth manifold which is not necessarily compact, and $E \to M$ a vector bundle over $M$, we do not obtain canonical Sobolev spaces as Banach spaces.
Nevertheless, we obtain local versions as locally convex linear spaces which are independent of measure and metric.
For $\alpha \in \R$, let $\SobH[loc]{\alpha}(M;E)$ be the set of all distributional sections $u$ of $E$ such that $u\rest{\Omega} \in \SobH{\alpha}(\Omega;E)$ for every precompact $\Omega \subset M$.

Typically, we will be in the setting of $(M,\mu)$, where $M$ is a smooth manifold with smooth measure $\mu$.
The spaces $\SobH[loc]{\alpha}(M;E)$ are locally convex linear spaces with the topology induced by the family of semi-norms $ \set{ \rho_{\Omega}(u) := \norm{u\rest{\Omega}}_{\SobH{\alpha}(\Omega)}: \Omega \text{ precompact} }.$
If $(E,h^E) \to M$ is a Hermitian vector bundle, then for $p \in [1,\infty)$, we obtain  Banach spaces $\Lp{p}(M;E)$ as the space of equivalence classes of measurable sections of $E$ with
$$ \int_{M} h^E(u,u)^{\frac{p}{2}}\ d\mu < \infty.$$
This is, by definition,  the norm $\norm{\cdot}_{\Lp{p}}^p$,  in $\Lp{p}(M;E)$ raised to the power $p$.
Two sections are considered equivalent if they coincide outside a set of measure zero.
For $p = \infty$, the space $\Lp{\infty}(M;E)$ is the space (of equivalence classes) of essentially bounded measurable sections of $E$.
The special case $p = 2$ is a Hilbert space, with inner product
$$\inprod{u,v}_{\Lp{2}} = \int_{M} h^E(u,v)\ d\mu.$$

For a first-order differential operator $D: \Ck{\infty}(M;E) \to \Ck{\infty}(M;F)$, where $(E,h^E), (F,h^F) \to M$ are two Hermitian bundles, there exists a unique formal adjoint $D^\dagger: \Ck{\infty}(M;F) \to \Ck{\infty}(M;E)$.
Then, the  maximal and minimal extensions  of $D$ are given by
$$ D_{\max} = (D^\dagger)^\ad,\quad D_{\min} = \close{D_{cc}},$$
where $D_{cc} = D$ with $\dom(D_{cc}) = \Ck[cc]{\infty}(M;E)$, the space of smooth sections with compact support contained in the interior of $M$.
Similarly, we define $D^\dagger_{\max}$ an $D^\dagger_{\min}$ by interchanging the roles of $D$ and $D^\dagger$.
The principal symbol of $D$ is denoted by $\sym_D(x,\xi)$, which in the first-order case is given by the commutator $[D, f_\xi]$, where $f_{\xi} \in \Ck[c]{\infty}(M)$ with $d f_{\xi}(x) = \xi$.

\subsection{Minimal setup and regularity}

The most general setup, the background setup,  under which we obtain results, is the following.
\begin{enumerate}[label=(M\arabic*), labelwidth=0pt, labelindent=2pt, leftmargin=26pt]
\label{StdSetup}
\item \label{Hyp:StdFirst}
      $M$ is a smooth manifold with smooth boundary $\dM$, equipped with a smooth measure $\mu$;
\item  \label{Hyp:InteriorVec}
      $\vec{T}$ is an interior pointing vector field along $\dM$ and $\tau$ the associated covector field;
\item $(E,h^E),\ (F,h^F) \to M$ are Hermitian vector bundles over $M$;
\item $D$ is a first-order elliptic differential operator from $E$ to $F$;
\item \label{Hyp:StdLast}
      \label{Hyp:Complete}
      $D$ and $D^\dagger$ are complete: i.e., compactly supported sections in $\dom(D_{\max})$ and $\dom(D^\dagger_{\max})$  are dense in the respective graph norms.
\end{enumerate}

From \ref{Hyp:InteriorVec}, combining with \ref{Hyp:StdFirst}, the induced measure on $\dM$ is given by
$$\mu_{\tau}(\xi_1, \dots, \xi_{n-1}) := \mu(\tau, \xi_1, \dots, \xi_{n-1})$$
for $\xi_i \in T^\ast \dM$ and by viewing $\mu$ as a density.

The assumption \ref{Hyp:Complete} is readily verified.
Theorem~\ref{Thm:cherwolf} provides a general criterion.

Under these assumptions, we first note we can make sense of the boundary trace map as follows.

\begin{theorem}
\label{Thm:Trace1}
Under the assumptions~\ref{Hyp:StdFirst}--\ref{Hyp:StdLast}, the space $\Ck[c]{\infty}(M;E)$ is dense in $\dom(D_{\max})$  with respect to the corresponding graph norm.
Moreover, the restriction map to the boundary
$$u \mapsto u \rest{\partial M}: \Ck[c]{\infty}(M;E) \to \Ck[c]{\infty}(\partial M;E)$$
has a unique bounded extension
$$ u \mapsto u\rest{\partial M}: \dom(D_{\max}) \to \SobH[loc]{-\frac12}(\partial M;E).$$
\end{theorem}

Recall the Green's formula
\begin{equation}
\label{A:IntParts}
\inprod{D u, v}_{\Lp{2}(M; F)} - \inprod{u, D^\dagger v}_{\Lp{2}(M; E)} = -\inprod{\sym_0 u\rest{\dM}, v\rest{\dM}}_{\Lp{2}(\dM;F)},
\end{equation}
for $u \in \Ck[c]{\infty}(M;E)$ and $v \in \Ck[c]{\infty}(M;F)$.
Here $\sym_0=\sym_D(\tau)$.
We will generalise this formula to $u$ and $v$ in the respective maximal domains in Theorem~\ref{Thm:MaxDom}.

By making sense of the trace, we can then proceed to prove the following regularity result under the minimal setup.

\begin{theorem}
\label{Thm:Reg}
Under the minimal setup~\ref{Hyp:StdFirst}--\ref{Hyp:StdLast}, we have that:
\begin{align}
\label{Eq:MaxDom}
\dom(D_{\max}) \intersect & \SobH[loc]{k}(M;E) \notag \\
                          & =
\big\{ u \in \dom(D_{\max}): Du \in \SobH[loc]{k-1}(M;E)\text{ and } u\rest{\dM} \in \SobH[loc]{k-\frac{1}{2}}(\dM;E)\big\}.
\end{align}
In particular, by the Sobolev embedding theorem,
\begin{align*}
\dom(D_{\max}) \intersect \Ck{\infty}(M;E)
=
\set{ u \in \dom(D_{\max}): Du \in \Ck{\infty}(M;E)\text{ and } u\rest{\dM} \in \Ck{\infty}(\dM;E)}.
\end{align*}

\end{theorem}

Obviously, on interchanging the roles of $D$ and $D^\dagger$, we obtain the same conclusions of Theorem~\ref{Thm:Trace1} and  Theorem~\ref{Thm:Reg} for $D^\dagger$.

\subsection{Geometric setup and existence}
\label{subsec:GeometricSetup}

To understand the maximal domain of the operator and study boundary value problems, we impose an additional set of assumptions which we call the geometric setup.
In what is to follow, we will  use the notation $Y_{[0,r)} := [0,r) \times \dM$ for $r > 0$ to denote the $r$-cylinder over $\dM$.

The assumptions below are automatically satisfied when the boundary $\dM$ is compact, and they are identified here as a sufficiently general and geometrically verifiable set of assumptions beyond compact boundary.

\begin{enumerate}[label=(G\arabic*), labelwidth=0pt, labelindent=2pt, leftmargin=26pt]
\label{ExtSetup}

\item \label{Hyp:ExtFirst}
      Let $\tau$ be as in \ref{Hyp:InteriorVec} and let $A$ and $\tilde{A} = -(\sym_{D}(\cdot,\tau(\cdot))^{-1})^\ast A \sym_{D}(\cdot,\tau(\cdot))^\ast$ be essentially selfadjoint first-order differential operators on  $E\rest{\partial M}$ and $F\rest{\partial M}$  respectively, adapted to the operators $D$ and $D^\dagger$ in the sense that for all $\xi \in T^*\dM$, 
      \begin{align*}
      \sym_{A}(x,\xi)         & = \sym_{D}(x,\tau(x))^{-1} \circ \sym_{D}(x,\xi) \text{ and}      \\
      \sym_{\tilde{A}}(x,\xi) & = \sym_{D^\dagger}(x,\tau(x))^{-1} \circ \sym_{D^\dagger}(x,\xi).
      \end{align*}

\item  \label{Hyp:Metric}
      Let $U \supset \dM$ in $M$ be an open neighbourhood and let $\Phi = (t, \phi)\colon U \to Y_{[0,T_0)}$ be a diffeomorphism with $T_0 > 0$ such that
      \begin{enumerate}[(i)]
      \item $\dM = t^{-1}(0)$
      \item $\phi\rest{\dM} = \id_{\dM}$
      \item $\tau = dt$ along $\dM$,
      \item $d\Phi(\vec{T}) = \partial/\partial t$ on $\dM$ where $\vec{T}$ is the associated vector field to $\tau$, and
      \item $\Phi_{\ast} \mu = \modulus{dt} \otimes \mu_{\tau}$.
      \end{enumerate}

\item \label{Hyp:RemControl}
      \label{Hyp:ExtLast}
      There exists a $T \in (0, T_{0})$ such that on $\Phi^{-1} Y_{[0,T)}$:
      $$D = \sym_t(\partial_t + A + R_t)\quad\text{and}\quad D^\dagger = - \sym_t^\ast(\partial_t + \tilde{A} + \tilde{R}_t),$$
      where $\sym_t(x) := \sym_{D}(t,x,dt)$ and for which there exists $C \geq 1$ such that for all $(t,x) \in \Phi^{-1} Y_{[0,T)}$ and $u\in\Ck[c]{\infty}(\dM;E)$,
      \begin{gather*}
      C^{-1}\modulus{u(x)}_{h^E} \leq \modulus{\sym_t(x)u(x)}_{h^F} \leq C \modulus{u(x)}_{h^E}, \\
      \norm{R_t u}_{\Lp{2}(\dM)} \leq C t  \norm{A u}_{\Lp{2}(\dM)} + C\norm{ u}_{\Lp{2}(\dM)}, \\
      \norm{\tilde{R}_t u}_{\Lp{2}(\dM)} \leq  C  t\norm{\tilde{A} u}_{\Lp{2}(\dM)} + C \norm{ u}_{\Lp{2}(\dM)}.
      \end{gather*}
\end{enumerate}

As a consequence of \ref{Hyp:Metric}, for a given $\rho \leq T$, we write $Z_{[0,\rho)} := \Phi^{-1} Y_{[0,\rho)}$.
If $\rho < T$, then define $Z_{[0, \rho]} := \Phi^{-1} Y_{[0, \rho]}$.

\begin{remark}
Note that the $\tau$ in \ref{Hyp:ExtFirst} is a particular choice of $\tau$, which is related to the operators $A$ and $\tilde{A}$ through the alluded principal symbol condition.
Contrast this to the $\tau$ appearing in \ref{Hyp:InteriorVec}, which was arbitrary.
\end{remark}

By the fact that $A$ are selfadjoint operators, we obtain that $\chi^{+}(A) = \chi_{[0,\infty)}(A)$ and $\chi^-(A) = \chi_{(-\infty,0)}(A)$ are bounded selfadjoint projections on $\dom(\modulus{A}^\alpha)$ as well as its dual space $\dom(\modulus{A}^\alpha)^\ast$  for all $\alpha \geq 0$.
Therefore, define
$$
\checkH(A) := \chi^-(A) \dom(\modulus{A}^{\frac12}) \oplus \chi^+(A) \dom(\modulus{A}^{\frac12})^\ast,
$$
with norm
$$
\norm{u}_{\checkH(A)} := \norm{\chi^-(A)u}_{\dom(\modulus{A}^\frac12)}  + \norm{\chi^+(A)u}_{ \dom(\modulus{A}^\frac12)^\ast}.
$$
For $u\in\Lp{2}(\dM;E)$ the norm $\norm{u}_{\dom(\modulus{A}^\frac12)^\ast}$ is equivalent to $\norm{(1+|A|)^{-\frac12}u}_{\Lp{2}}$ and, if $A$ is invertible, to $\norm{|A|^{-\frac12}u}_{\Lp{2}}$.
After defining $\hatH(A) := \checkH(-A)$, the $\Lp{2}$-inner product extends to a perfect pairing $\inprod{\cdot,\cdot}_{\checkH(A) \times \hatH(A)} \to \C$.

\begin{theorem}
\label{Thm:MaxDom}
Under the  setup~\ref{Hyp:StdFirst}--\ref{Hyp:StdLast} and  \ref{Hyp:ExtFirst}--\ref{Hyp:ExtLast}, we obtain the following:
\begin{enumerate}[label=(\roman*), labelwidth=0pt, labelindent=2pt, leftmargin=21pt]
\item \label{Thm:MaxDom:1}
      the trace maps $\Ck[c]{\infty}(M;E) \to \Ck[c]{\infty}(\dM;E)$ and $\Ck[c]{\infty}(M;F) \to \Ck[c]{\infty}(\dM;F)$
      given by $u \mapsto u\rest{\dM}$ extend uniquely to surjective bounded linear maps
      $\dom( D_{\max}) \to \checkH(A)$ and $\dom( (D^\dagger)_{\max}) \to \checkH(\tilde{A})$;
\item \label{Thm:MaxDom:1.5}
      the kernel of this extension $\dom( D_{\max}) \to \checkH(A)$ coincides with $\dom( D_{\min})$ and similarly for $D^\dagger$;
\item \label{Thm:MaxDom:1.7}
      $\Ck[c]{\infty}(\dM;E)$ is dense in $\checkH(A)$ and in $\hatH(A)$;
\item \label{Thm:MaxDom:2}
      for all $u \in \dom(D_{\max})$ and $v \in \dom((D^\dagger)_{\max})$,
      \begin{equation}
      \label{Eq:MaxDInt}
      \inprod{D_{\max} u, v}_{\Lp{2}(M;F)} - \inprod{u, (D^\dagger)_{\max}v}_{\Lp{2}(M;E)} = -\inprod{ u\rest{\dM}, \sym_0^\ast v\rest{\dM}}_{\checkH(A) \times \hatH(A)},
      \end{equation}
      where $\sym_0: E_{\dM} \to F_{\dM}$ is given by $\sym_0(x) = \sym_D(x,\tau(x))$. 
This induces an isomorphism $\checkH(A)\to \hatH(\tilde{A})$.
\end{enumerate}
\end{theorem}

\begin{remark}
\label{rem:maxmin}
Assertions~\ref{Thm:MaxDom:1} and \ref{Thm:MaxDom:1.5} in Theorem~\ref{Thm:MaxDom}, together with the open mapping theorem, imply that the trace map induces an isomorphism (in the sense of Banach spaces)
\begin{equation*}
\faktor{\dom(D_{\max})}{\dom(D_{\min})} \to \checkH(A).
\end{equation*}
\end{remark}

Theorem~\ref{Thm:MaxDom} and Remark~\ref{rem:maxmin} motivate the following definition.
\begin{definition}[Boundary condition, semi-regularity, regularity]
A \emph{boundary condition} for $D$ is a closed subspace $B \subset \checkH(A)$.
A \emph{semi-regular} boundary condition is one which satisfies $B \subset \SobH[loc]{\frac12}(\dM;E)$ and it is said to be \emph{regular} if in addition, $B^\ast \subset \SobH[loc]{\frac12}(M;F)$, where $B^\ast$ is the adjoint boundary condition.
\end{definition}

The associated operator $D_B$ is the extension with domain $\dom(D_B) = \set{u \in\dom(D_{\max}): u\rest\dM \in B}$, which is the pullback of $B$ under the trace map.
See  Section~\ref{Sec:BVPs} for more details.
In Section~\ref{Sec:Freddy}, a notion of coercivity of $D$ with respect to a boundary condition $B$ is introduced.
Loosely speaking, this is to say that the operator $D$ is invertible away from a compact set $K$ on $\dom(D_B)$.
This allows us to obtain classes of Fredholm boundary conditions despite the lack of compactness of both the manifold and boundary.

\begin{theorem}
\label{Thm:Coercive}
Suppose $D$ satisfies \ref{Hyp:StdFirst}--\ref{Hyp:StdLast} and \ref{Hyp:ExtFirst}--\ref{Hyp:ExtLast} and $B$ is a semi-regular boundary condition.
If $D$ is $B$-coercive with respect to a compact $K$ (cf.\ Definition~\ref{Def:Coercive}), then $\ker(D_B)$ is finite dimensional and $\ran(D_B)$ is closed.
\end{theorem}

\begin{corollary}
\label{Cor:Coercive}
If further $B$ is regular and $D^\dagger$ is $B^\ast$-coercive (where $B^\ast$ the adjoint boundary condition of $B$) with respect to a compact $K'$, then $D_B$ is Fredholm.
\hfill\qed
\end{corollary}

An important special case is when the adapted operator has discrete spectrum.
In this case, we obtain the following important consequence where we are able to obtain semi-Fredholmness where we only make a demand on the minimal extension for $A$-semi-regular boundary conditions, but with slightly more stringent geometric requirements than those in the geometric setup. 
Fredholmness follows when the boundary condition in question is $A$-regular.

\begin{theorem}
\label{Thm:Freddy:Discrete}
Suppose $D$ satisfies \ref{Hyp:StdFirst}--\ref{Hyp:StdLast} and \ref{Hyp:ExtFirst}--\ref{Hyp:ExtLast}. 
Assume the following.
  \begin{enumerate}[label=(\roman*), labelwidth=0pt, labelindent=2pt, leftmargin=21pt]
\item \label{Itm:Freddy:Discrete1} 
$M$ carries a complete Riemannian metric $g$  and constant $C < \infty$ such that $\modulus{ \sym_{D}(\xi)}_{h^E \to h^E} \leq C \modulus{\xi}_{g}$.\footnote{By Theorem~\ref{Thm:cherwolf}, this assumption implies \ref{Hyp:Complete}.}
\item \label{Itm:Freddy:Discrete3} 
The spectrum of $A$ is discrete.
\item \label{Itm:Freddy:Discrete4} 
$B$ is an $A$-semi-regular boundary condition. 
\end{enumerate} 
If $D$ is $0$-coercive with respect to a compact $K$ (cf.\ Definition~\ref{Def:Coercive}), then $\ker(D_B)$ is finite dimensional and $\ran(D_B)$ is closed.

If further $B$ is $A$-elliptically-regular, $D^\dagger$ is $0$-coercive with respect to a compact $K'$, then $D_B$ is Fredholm.
\end{theorem}

\begin{remark}
Note that for the boundary condition $B = 0$, which corresponds to the minimal extension $D_{\min}$, if $\eta \in \Ck[c]{\infty}(M)$, then $\eta \dom(D_{\min}) \subset \dom(D_{\min})$.
So Definition~\ref{Def:Coercive}~\ref{Def:Coercive:1} is automatically satisfied. 
Therefore, the $0$-coercivity condition on $D$ with respect to a compact set $K$ in Theorem~\ref{Thm:Freddy:Discrete} reduces to verifying that there is a constant $C < \infty$ such that and for all $u \in \Ck[cc]{\infty}(M;E)$ with $\spt u \cap K = \emptyset$, we have $\norm{Du} \geq C \norm{u}$. 
\end{remark}

\section{Sufficient criterion for completeness}

By definition, completeness of $D$ and $D^\dagger$ holds if $M$ is compact.
In the noncompact case, completeness can often be checked using the following theorem:

\begin{theorem}\label{Thm:cherwolf}
Let $D\colon \Ck{\infty}(M;E)\to\Ck{\infty}(M;F)$ be a first-order differential operator.
Suppose that $M$ carries a complete Riemannian metric with respect to which the principal symbol of $D$ satisfies an estimate
\begin{equation}
|\sym_D(\xi)| \le C(\mathrm{dist}(x,p))\cdot|\xi|
\label{eq:symest}
\end{equation}
for all $x\in M$ and $\xi\in T_x^*M$, where $p\in M$ is a fixed point and $C:[0,\infty)\to\R$ is a positive monotonically increasing continuous function with
\begin{equation}
\int_0^\infty \frac{dr}{C(r)} = \infty .
\label{eq:chernoff}
\end{equation}
Then $D$ and $D^\dagger$ are complete.
\end{theorem}

This theorem applies, for instance, if $C$ is a constant.
This is often the case when  the operator arises geometrically.
Note that we do not assume that $\mu$ is the volume element induced by the Riemannian metric.
The proof is essentially identical with the discussion in Section~3 of \cite{BB12}.
We give it for the sake of completeness.

\begin{proof}[Proof of Theorem~\ref{Thm:cherwolf}]
We first prove the theorem under the assumption that $C>0$ is constant.
Let $r:M\to\R$ be the distance function from the point $p$, $r(x)=\mathrm{dist}(x,\dM)$.
Then $r$ is a Lipschitz function with Lipschitz constant $1$.
Choose $\rho\in\Ck{\infty}(\R,\R)$ so that $0\le\rho\le 1$, $\rho(t)=0$ for $t\ge2$,
$\rho(t)=1$ for $t\le1$, and $|\rho'|\le 2$.
For $m\in\N$ set
\[
\chi_m(x) := \rho\left(\frac{r(x)}{m}\right).
\]
Then $\chi_m$ is a Lipschitz function and we have almost everywhere
\[
|d\chi_m(x)| \le \frac{2}{m} .
\]
Moreover, $\{\chi_m\}_m$ is a uniformly bounded sequence of functions converging pointwise to $1$.

Now let $ u\in\dom(D_{\max})$.
Then $\| \chi_m u -  u \|_{\Lp{2}(M)} \to 0$ as $m\to\infty$ by Lebesgue's theorem.
Furthermore, $\chi_m u$ has compact support and $\chi_m u\in\dom(D_{\max})$, see Lemma~3.1 in \cite{BB12}.
Since
\begin{align*}
\| D_{\max}(\chi_m u) & - D_{\max} u \|_{\Lp{2}(M)}                                                       \\
                      & \le \| (1-\chi_m)D_{\max} u\|_{\Lp{2}(M)} + \| \sigma_D(d\chi_m) u \|_{\Lp{2}(M)} \\
                      & \le \| (1-\chi_m)D_{\max} u\|_{\Lp{2}(M)} + \frac{2C}{m}\,\|  u \|_{\Lp{2}(M)}
\to 0
\end{align*}
as $m\to\infty$, we conclude that $\chi_m u \to  u$ in the graph norm of $D_{\max}$.

The same discussion applies to $D^\dagger$ because $|\sigma_{D^\dagger}(\xi)|=|-\sigma_{D}(\xi)^*|=|\sigma_{D}(\xi)|$.
This proves the theorem in case $C$ is constant.

Now we only assume that $C$ satisfies \eqref{eq:chernoff}.
Choose a smooth function $f:M\to\R$ with
\[
C(\mathrm{dist}(x,p)) \leq f(x) \le 2\, C(\mathrm{dist}(x,p))
\]
for all $x\in M$.
Let $g$ be the complete Riemannian metric for which \eqref{eq:symest} holds.
Then $h:=f^{-2}g$ is also complete because the $h$-length of a curve $c:[0,\infty)\to M$, starting at $p$ and parametrised by arc-length with respect to $g$, is estimated by
\begin{align*}
L_h(c)
 & =
\int_0^\infty \frac{|c'(t)|_g}{f(c(t))}\, dt
\ge
\frac12 \int_0^\infty \frac{|c'(t)|_g}{C(\mathrm{dist}(c(t),p))}\, dt
\ge
\frac12 \int_0^\infty \frac{1}{C(t)}\, dt
=
\infty .
\end{align*}
With respect to $h$, the principal symbol $\sigma_D$ is uniformly bounded and we have reduced the discussion to the case that $C$ is constant.
\end{proof}

\section{The first trace theorem and regularity}

In the compact boundary case, the regularity assertions for when the operator has higher regularity were given in terms of the $\SobH{k}$ spaces on the boundary.
Here, the compactness of these spaces meant that $\SobH{k}$ spaces on the boundary were canonically determined.
In our case, these assertions have to be rephrased to be local, and therefore, we require a localisation argument.
One issue with localisation at the boundary is that this introduces boundary to a subset that is already on the boundary.
Moreover, before we can begin to consider the questions of local regularity, we need to make sense of the boundary trace as a map from $\dom(D_{\max})$. For that, we introduce some important geometric constructions that will allow us to prove the trace theorem along with the regularity results.

\begin{lemma}
\label{Lem:ExtOp}
For every $x \in \partial M$, there exist compact neighbourhoods $V_x, U_x$ of $x$ in $\partial M$, a compact neighbourhood $Z_x$ of $x$ in $M$ which is contained in a chart, a $\delta > 0$, and a diffeomorphism $\Phi_x: [0, \delta] \times U_x \to Z_x$ such that:
\begin{enumerate}[label=(\roman*), labelwidth=0pt, labelindent=2pt, leftmargin=21pt]
\item \label{Lem:ExtOp:1} $V_x\subset\interior{U}_x$ and $\partial V_x$ and $\partial U_x$ are smooth, compact, boundaryless manifolds,
\item \label{Lem:ExtOp:2} $\Phi_x(0,y) = y$ for all $y\in U_x$,
\item \label{Lem:ExtOp:3} $E$ and $F$ are trivialised over a chart containing $Z_x$,
\item \label{Lem:ExtOp:4} there exists an elliptic first-order differential operator
      $$\tilde{D}: \Ck{\infty}(Z_x; E) \to \Ck{\infty}(Z_x; F)$$
      with
      $$ \tilde{D} \rest{\Phi_x([0,\frac12 \delta] \times V_x)} = D \rest{\Phi_x([0,\frac12 \delta] \times V_x)}$$
      and also such that, w.r.t.\ the chart and the chosen local trivialisations, $\tilde{D}$ has constant coefficients on a neighbourhood of $\Phi_x([0,\delta]\times\partial U_x \cup \{\delta\}\times U_x)$ in $Z_x$.
\end{enumerate}
\end{lemma}
\begin{proof}
Given $x \in \dM$, using a pre-compact chart $(W_x,\psi_x)$ allows us to pull the problem into a precompact open set  $W_x' := \psi_x(W_x) \subset  [0, \infty)\times \R^{n-1}$.
Shrinking the chart if necessary, there exist trivialisations for $E$ and $F$ over it.
With a slight abuse of notation, we identify $x\in\dM$ with $\psi_x(x)\in\{0\}\times\R^{n-1}$.

Let $t_3 \in (0,1)$  such that $B_{t_3}(x) \subset W_x'$.
Let $\delta \in (0,1)$ and $t_2 \in (0, t_3)$ such that  $[0, \delta]\times (B_{t_2}(x) \cap \R^{n-1})  \subset B_{t_3}(x)$.
Let $U_x' := B_{t_2}(x) \cap \R^{n-1}$.
Now, we can choose $0 < t_1 < \frac{3}{4}t_2$ such that $[0, \frac12 \delta]\times (B_{t_1}(x) \cap \R^{n-1})  \subset B_{\frac34 t_2}(x)$, see Figure~\ref{fig:neighbourhoods}.
Let $V_x' := B_{t_1}(x) \cap \R^{n-1}$.
We write the restriction to the boundary of the chart as $\psi_x(y)=(0,\phi_x(y))$.
    Put $U_x := \phi_x^{-1}( U_x')$ and $V_x := \phi_x^{-1}(V_x')$.
Define $\Phi_x$ on $[0,\delta]\times U_x$ by $\Phi_x(t,y)  := \psi_x^{-1}(t, \phi_x(y))$ and denote its image by $Z_x$.
With these choices, assertions \ref{Lem:ExtOp:1}--\ref{Lem:ExtOp:3} hold.

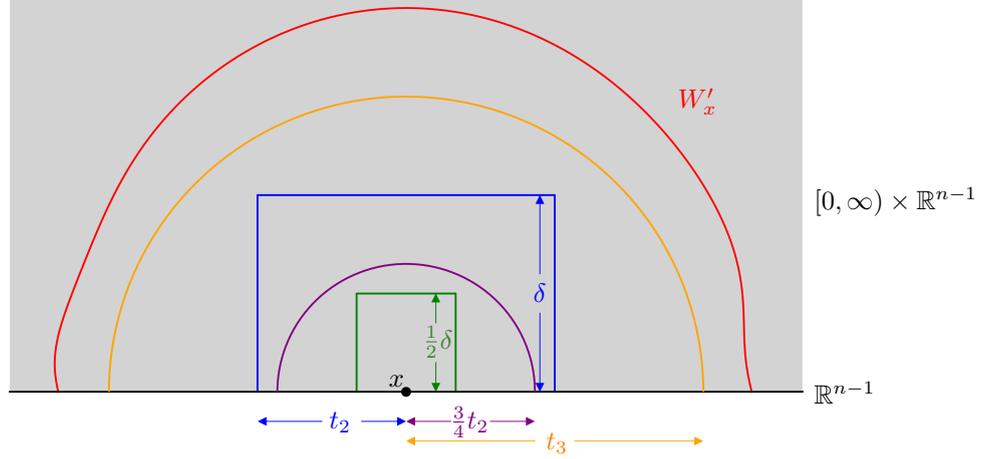
\begin{figure}[ht]
\begin{overpic}[scale=0.7]{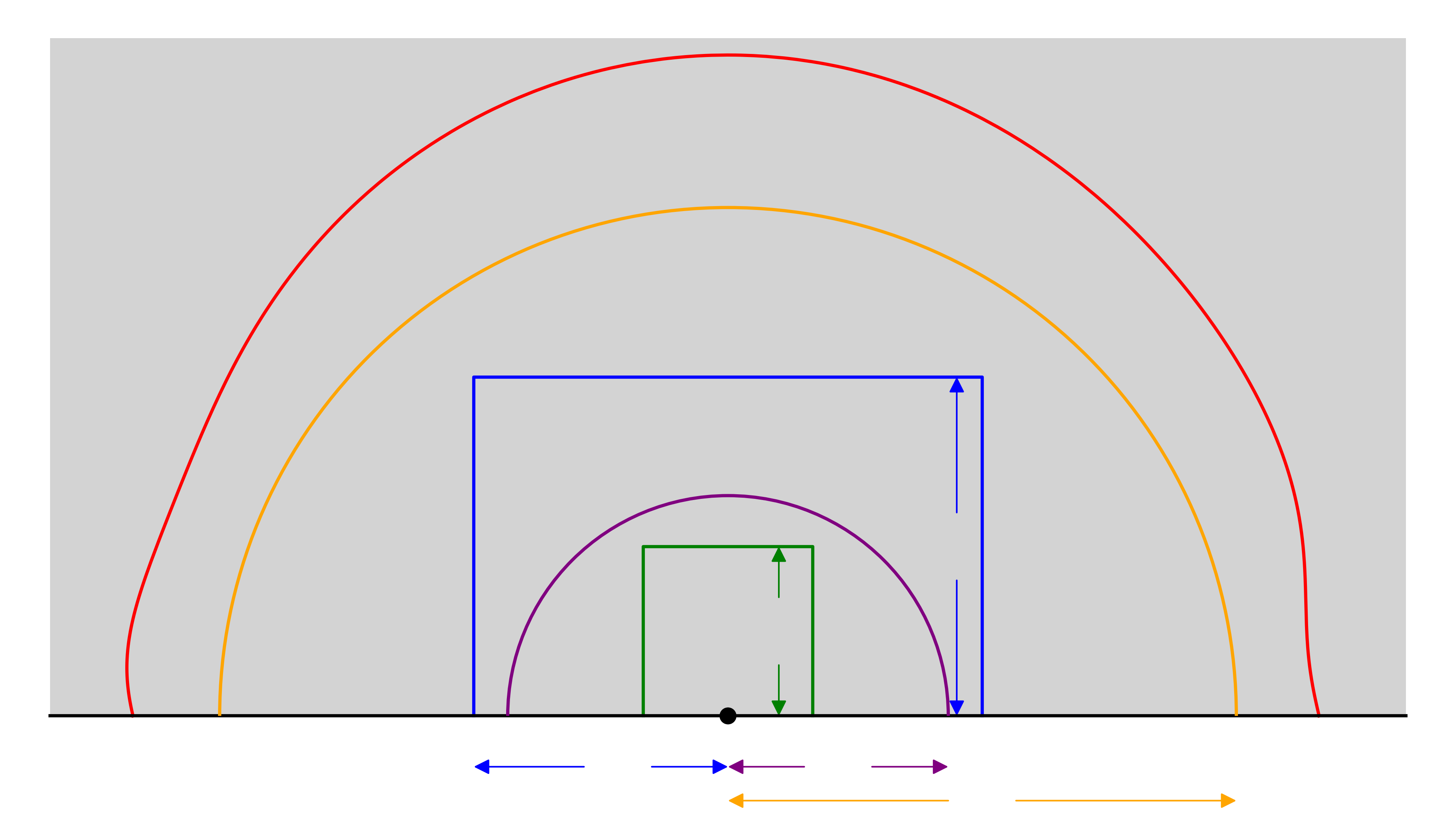}
\put(48,9){$x$}
\put(98,7){$\R^{n-1}$}
\put(98,30){$[0,\infty)\times\R^{n-1}$}
\put(41,4){\textcolor{blue}{$t_2$}}
\put(55.2,4){\textcolor{violet}{$\frac34 t_2$}}
\put(66.5,1.5){\textcolor{orange}{$t_3$}}
\put(52,13.5){\textcolor{OliveGreen}{$\frac12\delta$}}
\put(65,19){\textcolor{blue}{$\delta$}}
\put(82,42){\textcolor{red}{$W_x'$}}
\end{overpic}

\caption{Choice of neighbourhoods}
\label{fig:neighbourhoods}
\end{figure}

Now, we show \ref{Lem:ExtOp:4}.
Let $\eta \in \Ck{\infty}(\R, [0,1])$ such that $\eta(t) = 0$ for $t\le t_1$ and $\eta(t) = 1$ for $t\ge \frac34 t_2$.
Define $\Xi: [0,\delta] \times U_x' \to [0, \delta] \times U_x'$ by
$$
\Xi(y) := y + \eta \cbrac{|y-x|}(x - y).
$$
Clearly $\Xi$ is a smooth transformation and $\Xi\rest{B_{t_1}(x)} = \id$ and $\Xi\rest{[0,\delta] \times U_x' \setminus B_{t_2}(x)} = x$.

Since $D$ is a first-order operator, we can write $D = \sum_{j=1}^n A_j\frac{\partial}{\partial x^j} + B$ with respect to the local coordinates and chosen trivialisations, where $A_1,\dots,A_n,B$ are matrix-valued functions.
Therefore,  defining
$$
\tilde{D} = \sum_{j=1}^n (A_j\circ\Xi)\frac{\partial}{\partial x^j} + B\circ\Xi
$$
yields the desired operator.
\end{proof}

With this, we obtain the following important proposition.

\begin{corollary}
\label{Cor:EmbMf}
The manifold $N_x = Z_x\cup_{[0,\delta]\times U_x} Z_x'$ obtained by gluing a copy $Z_x'$ of $Z_x$ to $Z_x$ along $\Phi_x([0,\delta]\times \partial U_x)$ carries a smooth measure, Hermitian vector bundles $\tilde{E}$, $\tilde{F}$ and an elliptic first-order differential operator
$$
\dtilde{D}: \Ck{\infty}(N_x, \tilde{E}) \to \Ck{\infty}(N_x, \tilde{F})
$$
such that
\begin{enumerate}[label=(\roman*), labelwidth=0pt, labelindent=2pt, leftmargin=21pt]
\item\label{Cor:EmbMf:1}
the measure on $N_x$ extends the given measure on $Z_x$ smoothly,
\item \label{Cor:EmbMf:2}
      $\tilde{E}\rest{Z_x}=E\rest{Z_x}$ and $\tilde{F}\rest{Z_x}=F\rest{Z_x}$,
      \item\label{Cor:EmbMf:3}
      $\dtilde{D}\rest{\Phi_x([0, \frac12\delta] \times V_x)} = D\rest{\Phi_x([0,\frac12\delta] \times V_x])}$.
\end{enumerate}
\end{corollary}
\begin{proof}
The bundles double across along the respective identification.
Since the operator $\tilde{D}$ constructed in the proposition has constant coefficients near $\Phi_x([0, \delta] \times \partial U_x \union \set{\delta} \times U_x)$, the operator also doubles smoothly.
Together, this yields \ref{Cor:EmbMf:2} and \ref{Cor:EmbMf:3}.
The existence of the extension of the measure as stated in \ref{Cor:EmbMf:1} is obvious.
\end{proof}

With this, let us first obtain the following preliminary trace result.

\begin{corollary}
\label{Cor:Trace0}
Fix $x \in \partial M$ and let $V_x$ and $\delta > 0$ be as in Corollary~\ref{Cor:EmbMf}.
Then, there exists $C > 0$ such that for each $u \in \Ck[c]{\infty}(M;E)$ with $\spt u \subset \Phi_x([0, \frac12 \delta ) \times V_x)$ we have

$$
\norm{u\rest{\partial M}}_{\SobH{-\frac12}(V_x)} \leq C \norm{u}_{D}.$$
\end{corollary}
\begin{proof}
Under the stated assumptions, by locality, we see that $\dtilde{D}u = Du$.
Invoking Theorem~2.3 in \cite{BBan}, we obtain that
$$ \norm{u\rest{\partial M}}_{\checkH(A)}  \lesssim \norm{u}_{\dtilde{D}} = \norm{u}_{D},$$
where
$$ \checkH(A) = \chi^-(A)\SobH{\frac12}(\partial N_x,E) \oplus \chi^+(A)\SobH{-\frac12}(\partial N_x, E).$$
Since $\checkH(A) \embeds \SobH{-\frac12}(\partial N_x,E)$ is a continuous embedding, the assertion follows.
\end{proof}

First, applying this proposition, we prove the following important density result.
The proposition allows for the reduction of the density claim to the well-known assertion for compactly boundary considered in \cite{BB12}.

\begin{lemma}
\label{Lem:Density}
Assume \ref{Hyp:StdFirst}--\ref{Hyp:StdLast}.
Then the space $\Ck[c]{\infty}(M;E)$ is dense in $\dom(D_{\max})$ in $\norm{\cdot}_D$.
\end{lemma}
\begin{proof}
By \ref{Hyp:Complete}, we can assume that $\spt u$ is compact.
Let $\Omega \subset M$ be a smooth domain with $\spt u \subset \Omega$.
Using Corollary~\ref{Cor:EmbMf}, cover $\Omega \cap  M$ by sets $\interior{V}_i$, where $V_i$ and $U_i$ are the sets guaranteed by Lemma~\ref{Lem:ExtOp}.
Choose $\delta = \min\set{\delta_i}$ and let $\chi \in \Ck[c]{\infty}(M)$ such that $\chi = 1$ on $\Phi_i([0, \frac12 \delta) \times \interior{V}_i)$ and $0$ outside of $\Phi_i[0,\delta) \times \interior{U}_i$.
Then, $u = \chi u + (1 - \chi)u$ and $\spt (1 - \chi u) \subset \interior{M}$.
Therefore, by interior regularity, there exists $u^0_n \to (1  - \chi u)$ with $u^0_n \in \Ck[cc]{\infty}(M,E)$.

We now consider approximating $\chi u$ by smooth sections.
For this, let $\set{\eta_i}$ be a smooth partition of unity subordinate to $\Phi_i ([0,\delta)\times \interior{V}_i)$ extended to $0$ to the whole of $N_i$.
Now, since $D$ and $\dtilde{D}_i$ are equal on $\Phi_i ([0,\delta) \times \interior{V}_i)$, we obtain that $\eta_i \chi u \in \dom(\dtilde{D}_{i,\max})$.
By Theorem 6.7 in \cite{BB12}, we obtain $\tilde{u}^i_n \to \eta_i \chi u$ in the $\dtilde{D}_i$-norm with $\tilde{u}^i_n \in \Ck[c]{\infty}(N_i;E)$.
However, by the support properties of $\eta_i \chi u$, this is equal to convergence in the $D$-norm.
This sequence can be chosen such that $\spt \tilde{u}^i_n \subset \Phi_i ([0,  \frac12 \delta) \times V_i))$ and by extension to $0$  in $M$, we obtain a sequence $u^i_n$.
Then,
$$ u_n := u^0_n + \dots + u^k_n \in \Ck[c]{\infty}(M,E)$$
and by construction, $u_n \to u$ the $D$-norm.
\end{proof}

Recall that for a bounded domain $\Omega \subset M$ (or in $\dM$), the negative order Sobolev space $\SobH{-\alpha}(\Omega;E) := \SobH[0]{\alpha}(\Omega;E)^\ast$.
In particular,
$$\norm{u}_{\SobH{-\alpha}(\Omega)} \simeq \sup_{0 \neq v \in \Ck[c]{\infty}(\Omega;E)} \frac{\inprod{u,v}}{\norm{v}_{\SobH{\alpha}(\Omega)}}.$$

\begin{lemma}
\label{Lem:NegSobIneq}
Let $\Omega \subset \interior{\tilde{\Omega}} \subset\tilde{\Omega}\subset  N$, where $N$ is a manifold and $\tilde{\Omega},\ \Omega$ are connected compact subsets with nonempty interior.
Then, for $\alpha \geq 0$ and  all $u \in \SobH{-\alpha}(\tilde{\Omega})$
$$ \norm{ u\rest{\Omega}}_{\SobH{-\alpha}(\Omega)} \lesssim  \norm{ u }_{\SobH{-\alpha}(\tilde{\Omega})}.$$
\end{lemma}
\begin{proof}
Taking the extension $\tilde{v} \in \Ck[c]{\infty}(\tilde{\Omega})$, we note
$$
\modulus{u\rest{\Omega}[v]}= \modulus{u[\tilde{v}]} \leq \norm{u}_{\SobH{-\ceil{\alpha}}(\Omega)} \norm{\tilde{v}}_{\SobH{\ceil{\alpha}}(\tilde{\Omega})}.$$
But, since $\spt \tilde{v} \subset \Omega$, we have that
$$\norm{\tilde{v}}_{\SobH{\ceil{\alpha}}(\tilde{\Omega})} =\norm{v}_{\SobH{\ceil{\alpha}}({\Omega})}.$$
The assertion follows since $\SobH{-s}(\Omega;E)$ and $\SobH{s}(\Omega;E)$ form interpolation scales.
\end{proof}

Next, we prove the following lemma which is essential to the proof of Theorem~\ref{Thm:Trace1}.

\begin{lemma}
\label{Lem:OmegaEst}
Let $\Omega_0 \subset \interior{\Omega}_1$ and  $\Omega_1 \subset \interior{\Omega}_2 \subset N$, where $N$ is a manifold and $\Omega_i$ are connected compact sets with nonempty interior.
For $\alpha \geq 0$ and all $u \in \Ck[cc]{\infty}(\Omega_2)$ with $\spt u \subset \Omega_0$,
$$ \norm{u}_{\SobH{-\alpha}(\Omega_2)} \lesssim \norm{u}_{\SobH{-\alpha}(\Omega_1)}.$$
\end{lemma}
\begin{proof}
Let $\chi \in \Ck[cc]{\infty}(\Omega_2)$ such that $\chi = 1$ on $\Omega_0$ and $\chi = 0$ in $\Omega_2 \setminus \Omega_1$.
Then,
\begin{align}
\label{Eq:OmegaEst1}
\norm{u}_{\SobH{-\alpha}(\Omega_2)}
\simeq
\sup_{v \in \Ck[cc]{\infty}(\Omega_2)} \frac{\modulus{\inprod{u, v}}}{\norm{v}_{\SobH[0]{\alpha}(\Omega_2)}}
=
\sup_{v \in \Ck[cc]{\infty}(\Omega_2)} \frac{\modulus{\inprod{\chi u, v}}}{\norm{v}_{\SobH[0]{\alpha}(\Omega_2)}}
=
\sup_{v \in \Ck[cc]{\infty}(\Omega_2)} \frac{\modulus{\inprod{ u, \chi v}}}{\norm{v}_{\SobH[0]{\alpha}(\Omega_2)}}
\end{align}
Now, note that
$$ \norm{\chi v}_{\SobH[0]{\ceil{\alpha}}(\Omega_1)}
=
\norm{\chi v}_{\SobH[0]{\ceil{\alpha}}(\Omega_2)}
\lesssim
\norm{v}_{\SobH[0]{\ceil{\alpha}}(\Omega_2)},
$$
and this inequality holds also for the case $\alpha = 0$, i.e., for $\Lp{2}$.
Therefore, on noting that $\Ck[cc]{\infty}(\Omega_2)$ is dense in $\SobH[0]{\beta}(\Omega_2)$ for $\beta \in [0,\ceil{\alpha}]$, by interpolation, we obtain $\norm{\chi v}_{\SobH[0]{\ceil{\alpha}}(\Omega_1)} \lesssim \norm{v}_{\SobH[0]{\ceil{\alpha}}(\Omega_2)}$.
Hence, substituting this into \eqref{Eq:OmegaEst1},
$$
\norm{u}_{\SobH{-\alpha}(\Omega_2)}
\lesssim
\sup_{v \in \Ck[cc]{\infty}(\Omega_2)} \frac{\modulus{\inprod{ u, \chi v}}}{\norm{\chi v}_{\SobH[0]{\alpha}(\Omega_1)}}
\leq
\sup_{w \in \Ck[cc]{\infty}(\Omega_1)} \frac{\modulus{\inprod{ \chi u, w}}}{\norm{w}_{\SobH[0]{\alpha}(\Omega_1)}}
= \norm{\chi u}_{\SobH{-\alpha}(\Omega_1)},$$
where the last inequality follows from the fact that $\set{ (\chi v)\rest{\Omega_1}: v \in \Ck[cc]{\infty}(\Omega_2)} \subset \Ck[cc]{\infty}(\Omega_1)$, and $\chi u = u$ on $\Omega_1$.
\end{proof}

With this, we are now able to prove Theorem~\ref{Thm:Trace1}.

\begin{proof}[Proof of Theorem~\ref{Thm:Trace1}]
In Lemma~\ref{Lem:Density}, we have already proved that $\Ck[c]{\infty}(M;E)$ is dense in $\dom(D_{\max})$.
Therefore, to prove that $u \mapsto u\rest{\partial M}$ extends uniquely to a bounded map $\dom(D_{\max}) \to \SobH[loc]{-\frac12}(\partial M;E)$, it suffices to prove that, for any $\Omega \subset \partial M$ compact with nonempty interior, we obtain that for all $u \in \Ck[c]{\infty}(M;E)$,
$$ \norm{u\rest{\partial M}}_{\SobH{-\frac12}(\Omega)} \lesssim \norm{u}_{D}.$$

For that, fix $u \in \Ck[c]{\infty}(M;E)$ and fix $\Omega \subset \partial M$ compact with nonempty interior.
Let $\set{V_i}_{i=1}^k$ and $\set{U_i}_{i=1}^k$ be compact subsets such that $\Omega \subset \union_{i=1}^k \interior{V}_i$ and $V_i \subset \interior{U}_i$ satisfying the conclusion of Lemma~\ref{Lem:ExtOp} along with manifolds $N_i$, $\delta_i$ and diffeomorphisms $\Phi_i:[0, \delta] \times U_i \to Z_i \subset N_i$ from Corollary~\ref{Cor:EmbMf}.
Choose $\delta := \min\set{\delta_i: 1 \leq i \leq k}$, and let $\set{\eta_i}$ be a smooth partition of unity subordinate to $\set{ \Phi_i ([0,\frac12 \delta) \times \interior{V}_i)}$, extended by zero to the entirety of $N_i$.

We first pass to $u' = \chi u$ for $\chi \in \Ck[c]{\infty}(M;[0,1])$ where $\chi = 1$ on $\union_i [0, \frac13 \delta) \times V_i$ and $0$ outside of $[0,\frac12) \times V_i$.
This is possible because $Du = D(u' + (1 - \chi)u) = Du' + \sym_{D}(\dot, d\chi)u$ and
\begin{equation}
\label{Eq:Trace1:1}
\norm{u'}_{D} \simeq \norm{Du'} + \norm{u'} \leq \norm{Du} + \sup_{x \in \spt \chi} \modulus{\sym_{D}(x,d\chi)} \norm{u} + \sup_{x \in \spt \chi} \modulus{\chi(x)} \norm{u} \lesssim \norm{u}_{D}.
\end{equation}

Note that $u'\rest{\partial M} = u\rest{\partial M}$ and  we obtain that
\begin{equation}
\label{Eq:Trace1:2}
\norm{u\rest{\partial M}}_{\SobH{-\frac12}(\Omega)} = \norm{u'\rest{\partial M}}_{\SobH{-\frac12}(\Omega)} \lesssim \sum_{i=1}^k \norm{(\eta_i u)\rest{\partial M}}_{\SobH{-\frac12}(\Omega_i)}
\lesssim
\sum_{i=1}^k \norm{(\eta_i u)\rest{\partial M}}_{\SobH{-\frac12}(U_i)},
\end{equation}
where the ultimate inequality follows from invoking Lemma~\ref{Lem:OmegaEst} by choosing $\Omega_0 = V_i$, $\Omega_1 = U_i$ and $\Omega_2 = \Omega$ for each $i$.

Now, since $\Phi_i([0,\delta] \times U_i) = Z_i \subset N_i$ and $\spt (\eta_i u') \subset \Phi_i([0, \frac12 \delta) \times V_i) \subset Z_i$, from Corollary~\ref{Cor:Trace0}, we obtain
\begin{equation}
\label{Eq:Trace1:3}
\norm{(\eta_i u')\rest{\partial M}} \lesssim \norm{\eta_i u'}_{\dtilde{D}_i} = \norm{ (\eta_i u')}_{D} \leq \norm{u'}_{D},
\end{equation}
by the support properties of $(\eta_i u')$, that $\dtilde{D}_i$ is a local operator along with the fact that it is equal to $D$ on $\spt (\eta_i u')$.
On noting that $(\eta_i u')\rest{\partial M} = (\eta_i u)\rest{\partial M}$, combining \eqref{Eq:Trace1:1}, \eqref{Eq:Trace1:2}, and \eqref{Eq:Trace1:3}  yields
$$
\norm{(\eta_i u)\rest{\partial M}}_{\SobH{-\frac12}(V_i)} = \norm{ (\eta_i u') \rest{\partial M}}_{\SobH{-\frac12}(V_i)}  \lesssim \norm{u'}_{D} \lesssim \norm{u}_{D}$$
for each $i$.
Combining these estimates, we obtain
\[
\norm{u\rest{\partial M}}_{\SobH{-\frac12}(\Omega)} \lesssim \sum_{i=1}^k \norm{(\eta_i u)\rest{\partial M}}_{\SobH{-\frac12}(V_i)}  \leq k \norm{u}_{D} \lesssim \norm{u}_{D}.
\qedhere
\]
\end{proof}

Next, we note the following important assertion regarding local Sobolev scales.

\begin{lemma}
\label{Lem:RestReg}
For $k > \frac12$,  $u \mapsto u\rest{\partial M}: \SobH[loc]{k}(M) \to \SobH[loc]{k-\frac{1}{2}}(\partial M)$ continuously.
\end{lemma}
\begin{proof}
Fix $\Omega \subset \partial M$ that is compact with nonempty interior.
Let $\set{V_i}$ compact with smooth boundary with $\delta_i$ as guaranteed by Corollary~\ref{Cor:EmbMf} such that $\set{\interior{V}_i}$ cover $\Omega$.
Set $\delta = \min \set{\delta_i}$.
Letting $\eta_i$ be a smooth partition of unity subordinate to $[0,\frac12 \delta] \times V_i$, and for $u \in \SobH[loc]{k}(M;E)$,
\begin{align*}
\norm{u\rest{\partial M}}_{\SobH{k-\frac12}(\Omega)}
\lesssim \sum_{i=1}^\ell \norm{(\eta_i u)\rest{\partial M}}_{\SobH{k - \frac12}(V_i)}
\lesssim \sum_{i=1}^\ell \norm{(\eta_i u)}_{\SobH{k}([0,\delta] \times V_i)}
\lesssim \sum_{i=1}^\ell \norm{u}_{\SobH{k}([0,\delta] \times V_i)}.
\end{align*}
where the penultimate inequality is obtained from extending $\eta_i u$ to the whole of the compact manifold with boundary $N_i$.
Since $\Omega$ was arbitrary, these bounds are precisely the notion of continuity for locally convex linear spaces.
\end{proof}

We now have all the necessary ingredients to prove the main regularity theorem, Theorem~\ref{Thm:Reg}.

\begin{proof}[Proof of Theorem~\ref{Thm:Reg}]
For $k > 0$, if $u \in \dom(D_{\max}) \cap \SobH[loc]{k}(M;E)$, then clearly, $Du \in \SobH[loc]{k-1}(M;E)$ and from Lemma~\ref{Lem:RestReg}, we obtain that $u \in \SobH[loc]{k-\frac12}(\partial M;E)$. 
In the case of $k = 0$, Theorem~\ref{Thm:Trace1} yields $u \in \SobH[loc]{-\frac12}(\partial M;E)$.
This shows ``$\subset$''.

Now, we prove ``$\supset$''.
That is, suppose that $u \in \dom(D_{\max})$, $Du \in \SobH[loc]{k-1}(M;E)$ and $u\rest{\partial M} \in \SobH[loc]{k-\frac12}(\partial M;E)$.
To prove that $u \in \SobH[loc]{k}(M;E)$, we need to show that for each $\Omega \subset M$ which is compact with nonempty interior, we have that
$\norm{u}_{\SobH{k}(\Omega)}< \infty.$
Therefore, fix $\Omega \subset M$ compact with nonempty interior, and let $\Omega'$ be another compact subset with nonempty interior and with smooth boundary with $\Omega \subsetneqq \Omega'$.
Let $\chi \in \Ck[c]{\infty}(M)$ be such that $\chi = 1$ on $\Omega$ and $\chi = 0$ outside of $\Omega'$.
Set $u' = \chi u$.
It is easy to see that $u' \in \dom(D_{\max})$, $Du' \in \SobH[loc]{k-1}(M;E)$  and that $u' \rest{\partial M} \in \SobH[loc]{k-\frac12}(\partial M;E)$.
Let $\set{V_i}$ compact subsets of $\partial M$ with nonempty interior such that $\set{\interior{V}_i}$ cover $\Omega' \cap \partial M$ as guaranteed from Lemma~\ref{Lem:ExtOp} and let $\delta = \min\set{\delta_i}$.
Now, let $\xi = 1$ on $\union_{i=1}^\ell \Phi_i ([0,\frac12 \delta] \times V_i)$ and $0$ outside of $\union_{i=1}^\ell \Phi_i([0,\delta] \times V_i)$.
Let $u'' := (1 - \xi)u'$ and $u'' \in \dom(D_{\max})$ with $Du'' \in \SobH[loc]{k-1}(M;E)$.
Moreover, $\spt u'' \cap \partial M = \emptyset$ and therefore, from interior regularity theory, we have that $u'' \in \SobH[loc]{k}(M;E)$.
In particular, $u'' \in \SobH{k}(\Omega')$.

Now consider $v := \xi u'$.
We have that $u' = u'' + v$ and we show that $v \in \SobH{k}(\Omega')$.
This would then yield that $u' \in \SobH{k}(\Omega')$.
First let $\set{\eta_i}$ be a smooth partition of unity, subordinate to $\Phi_i([0, \frac34 \delta) \times \interior{V}_i)$.
Since $\spt v \subset \union_{i=1}^\ell \Phi_i( [0,\frac12] \times \interior{V}_i)$, we have that $v = \sum_{i=1}^\ell \eta_i v$.
Let $v_i \in \Lp{2}(N_i;E)$ be defined by
$$ v_i(x) = \begin{cases} \eta_i(x) v(x) & x \in \Phi_i( [0,\frac34 \delta ] \times V_i)     \\
              0              & x \not\in \Phi_i([0, \frac34 \delta] \times  V_i) \\
\end{cases}.$$
Note first that $\dtilde{D}_{i, \max}v_i = Dv_i$ and so  $v_i \in \dom(\dtilde{D}_{i, \max})$.
Furthermore, $Dv \in \SobH{k-1}(\Omega')$ yields that $\dtilde{D}_{i,\max}v_i \in \SobH{k-1}(N_i;E)$.
Moreover, since $v \in \SobH{k-\frac12}(\union_{i=1}^\ell \set{0} \times \partial V_i;E)$, we have that $v_i \in \SobH{k-\frac12 }(\partial N_i;E)$.
Therefore, by applying Theorem 2.4 in \cite{BBan},  we obtain that $v_i \in \SobH{k}(N_i;E)$.
However, since $v_i = 0$ outside of $[0,\frac34 \delta] \times V_i$, we obtain $\norm{\eta_i v}_{\SobH{k}(\Omega')} < \infty$.
Now $v \in \SobH{k}(\Omega')$ since
$$\norm{v}_{\SobH{k}(\Omega')} \lesssim \sum_{i=1}^\ell \norm{\eta_i v}_{\SobH{k}(\Omega')}.$$
This yields $u' \in \SobH{k}(\Omega';E)$ and since $u' = u$ on $\Omega$, we conclude $u \in \SobH{k}(\Omega;E)$.
Since $\Omega$ was an arbitrary compact set with nonempty interior, this proves $u \in\SobH[loc]{k}(M;E)$.
\end{proof}

\section{The model operator}

We first consider the ``model operator'' associated to our problem, which is denoted by $D_0 = \sym_0(\partial_t + A)$.

Throughout this section, due to the assumptions we make in \ref{ExtSetup} (in particular \ref{Hyp:ExtFirst}), we identify the unique closed extension $\close{A}$ with $A$, and similarly for $\tilde{A}$.
In \cite{BBan} the boundary was compact and therefore, the boundary operators exhibited pure point spectrum.
However, they were merely  bi-sectorial and not, in general, selfadjoint.
In particular, this meant that their generalised eigenspaces were in general nonorthogonal and possibly of higher algebraic multiplicity.
Bi-sectorial operators are well suited to analysis via the \Hinfty-functional calculus, which was afforded to us through pseudo-differential methods.

In our situation, despite the fact that the boundary defining operators are selfadjoint, we may have continuous spectrum, and therefore, we are again forced to abandon resorting to reasoning through eigenspaces.
However, motivated by the ``global'' approach we develop in \cite{BBan}, we treat this case operator more in the spirit of the bi-sectorial situation via functional calculus.
Unfortunately, due to noncompactness, it is unlikely we can resort to pseudo-differential methods to obtain access to the \Hinfty-functional calculus as we did in the compact case.
However, since we are in the selfadjoint context, we can instead allude to the Borel functional calculus as a suitable replacement.

Recall that by the selfadjointness of $A$, we obtain a Borel functional calculus via the spectral theorem.
In particular, this means we are able to construct bounded projectors $\chi_{I}(A)$, where $I \subset \R$ is an interval and $\chi_I$ its characteristic function.
Let us write $\chi^+(A) = \chi_{[0,\infty)}(A)$ and $\chi^-(A) = \chi_{(-\infty, 0)}(A)$ and define $\modulus{A} := A \sgn(A)$, where $\sgn(A) = \chi^+(A) - \chi^{-}(A)$.

In \cite{BB12} and \cite{BBan}, we were able to use the fact that $\chi^{\pm}(A)$ are pseudo-differential operators of order $0$ in order to see their boundedness properties on $\SobH{\alpha}(\dM;E)$, which coincided with $\dom(\modulus{A}^\alpha)$ as a consequence of the assumed compactness of $\dM$.
Since this assumption does not hold in our setting, we cannot reason in this manner and instead, we are forced to directly work with $\dom(\modulus{A}^\alpha)$ instead.

For $\epsilon>0$ put $|A|_\epsilon := (|A|+\epsilon I)$.
In light of Lemma~\ref{Lem:Frac} and Remark~\ref{Rem:FracPowerBd}, we see that $\modulus{A}_\epsilon^{-\alpha}$ is extended boundedly to $\dom(\modulus{A}^\alpha)^\ast$, the dual space of $\dom(\modulus{A}^\alpha)$ and, in fact,  $\norm{u}_{\dom(\modulus{A}^\alpha)^\ast} \simeq \norm{ \modulus{A}_\epsilon^{-\alpha} u}$.
By the same lemma, we obtain that $\norm{u}_{\dom(\modulus{A})^\alpha} \simeq \norm{\modulus{A}_\epsilon^\alpha u}$.
Also, from Lemma~\ref{Lem:ProjDualBdd}, we obtain that $\chi_{I}(A): \dom(\modulus{A}^\alpha)^\ast \to \dom(\modulus{A}^\alpha)^\ast$ extend as bounded projections.

Define
\begin{equation}
\label{Eq:Check}
\checkH(A) = \chi^{-}(A) \dom(\modulus{A}^{\frac12}) \oplus\chi^{+}(A) \dom(\modulus{A}^{\frac12})^\ast
\end{equation}
with norm
\begin{equation}
\norm{u}_{\checkH(A)}^2 := \norm{\chi^{-}(A)u}^2_{\dom(\modulus{A}^{\frac12})} + \norm{\chi^+(A)u}_{\dom(\modulus{A}^{\frac12})^\ast}.
\end{equation}
Also, define $\hatH(A) := \checkH(-A)$.
Proposition~\ref{Prop:CheckIso} tells us that $\checkH(A-rI) = \checkH(A)$ and $\hatH(A-rI) = \hatH(A)$ for any $r\in\R$.
In applications, we will be precise about the norm which is used in computations.

Fix $T > 0$ from \ref{Hyp:RemControl} and for $\rho \in (0,T]$ let $\eta_\rho \in \Ck[c]{\infty}([0,T))$ such that $\eta_\rho(t) = 1$ whenever $x \in [0, 1/2\rho]$ and  $\eta_\rho(t) = 0$ for $[3/4\rho, T)$.
Using this, and fixing some $\epsilon > 0$, define the extension operator
\begin{equation}
(\ext_\rho u)(t,x) = \eta_\rho(t) (\exp(-t\modulus{A}_\epsilon) u)(x)
\end{equation}
for $u \in \Lp{2}(Y_{[0,T)})$ on the cylinder $Y_{[0,T)}$.
Using $\Phi$ from \ref{Hyp:Metric}, we regard $\ext_\rho$ as an operator on $Z_{[0,T)}\subset M$.
Recall the notation $\Dk{\infty}(A) = \cap_{k =1}^\infty \dom(\modulus{A}^k)$, which is guaranteed to be a dense subset in $\dom(\modulus{A}^k)$ by Corollary~\ref{Cor:DinfDensity}.
Moreover, it is a dense subset in $\checkH(A)$ and $\hatH(A)$ by Corollary~\ref{Cor:Dinf}.
By the Sobolev embedding theorem, $\Dk{\infty}(A) \subset \Ck{\infty}(\dM;E)$ and if $\dM$ is compact we even have equality.
In general, however, sections in $\Dk{\infty}(A)$ enjoy a certain decay that is crucial for analysis in our noncompact setting.

\begin{lemma}
\label{Lem:ExpReg}
The operator $\ext_\rho$ extends to a linear map $\dom(A)^\ast \to \Ck{\infty}(Z_{(0,T)};E) \cap \Lp{2}(Z_{[0,T)};E)$.
Furthermore,
$$ \ran(\ext_\rho\rest{\Dk{\infty}(A)}) \subset \Ck{\infty}([0,T); \Dk{\infty}(A)) \subset \Ck{\infty}(Z_{[0,T)};E).$$
\end{lemma}
\begin{proof}
This follows from Corollary~\ref{Cor:DinfDensity}, along with the fact that $\Dk{\infty}(A) \subset \Ck{\infty}(\dM;E)$.
\end{proof}

This lemma ensures in particular that $\ext_\rho$ is defined on $\checkH(A)$ and $\hatH(A)$ which are densely embedded in $\dom(A)^\ast$.

\begin{proposition}
\label{Prop:ModelBnd}
For $u \in \checkH(A)$, we have that $\ext_\rho u \in \dom( D_{0,\max})$ and
$$ \norm{\ext_\rho u}_{D_0} \lesssim \norm{u}_{\checkH(A)},$$
where the implicit constants depend on $T$, $\eta_\rho$ and $\epsilon > 0$.
Similarly, for $v \in \hatH(A)$, we have that $\ext_\rho v \in \dom( (\sigma_0^{-1} D)^\dagger_{0,\max})$ and
$$ \norm{\ext_\rho u}_{(\sigma_0^{-1} D_0)^\dagger} \lesssim \norm{v}_{\hatH(A)}.$$
\end{proposition}
\begin{proof}
We restrict our considerations here to $u \in \Dk{\infty}(A)$ because, by Corollary~\ref{Cor:Dinf}, $\Dk{\infty}(A)$ is dense in $\checkH(A)$.
That $\ext_\rho$ maps $\checkH(A)$ to $\Lp{2}(Z_{[0,T)};E)$ follows directly from Lemma~\ref{Lem:ExpReg}.

First, since $\modulus{\sym_0(x)}_{\mathrm{op}} \leq C$ by \ref{Hyp:RemControl}, we note that
$$ \norm{\ext_\rho u}_{D_0} \leq C (\norm{ (\partial_t + A) \ext_\rho u} + \norm{ \ext_\rho u}),$$
so to establish the claim, it suffices to estimate the right hand side.

Let $u = u^- + u^+$, where $u^{\pm} = \chi^{\pm}(A)$.
Then, $ \ext_\rho u = \ext_\rho u^- + \ext_\rho u^+$, and a calculation mimicking the proof of Proposition 5.4 in \cite{BBan} then yields
\begin{equation}
\label{Eqn:Soln}
\begin{aligned}
(\partial_t + A) u^- & = (\partial_t \eta_\rho(t) + \epsilon \eta_\rho(t)) \exp(-t \modulus{A}_\epsilon)u^- - 2 \eta_\rho(t) \modulus{A}_\epsilon \exp(-t\modulus{A}_\epsilon)u^- \\
(\partial_t + A) u^+ & = (\partial_t \eta_\rho(t) + \epsilon \eta_\rho(t)) \exp(-t \modulus{A}_\epsilon)u^+.
\end{aligned}
\end{equation}
Also,
\begin{align*}
\norm{ \ext_\rho u}^2 & = \int_{0}^\infty \eta_\rho(t)^2 \cbrac{\norm{\exp(-t\modulus{A}_\epsilon) u^-}^2 + \norm{\exp(-t\modulus{A}_\epsilon) u^+}^2}\ dt        \\
                      & \lesssim \int_{0}^T \eta_\rho(t)^2 \cbrac{\norm{\exp(-t\modulus{A}_\epsilon) u^-}^2\ dt + \norm{\exp(-t\modulus{A}_\epsilon) u^+}^2\ dt}.
\end{align*}
From the selfadjointness and positivity of  $\modulus{A}_\epsilon$, we have that $\norm{\exp(-t\modulus{A}_\epsilon)}_{\Lp{2}\to\Lp{2}} \leq 1$, and therefore,
$$
\int_{0}^T \eta_\rho(t)^2 \norm{\exp(-t\modulus{A}_\epsilon) u^-}^2\ dt  \leq T \norm{u^-} \lesssim \norm{u^-}_{\dom(\modulus{A}^{\frac12})}.$$
For the second term,
\begin{align*}
\int_{0}^\infty \norm{\exp(-t\modulus{A}_\epsilon) u^+}^2\ dt & \leq \int_0^\infty \norm{t^{\frac{1}{2}} \modulus{A}_\epsilon^{\frac{1}{2}} \exp(-t\modulus{A}_\epsilon) \modulus{A}_\epsilon^{-\frac{1}{2}} u^+}^2\ \frac{dt}{t} \\
                                                              & \lesssim \norm{\modulus{A}_\epsilon^{-\frac{1}{2}}u^+} \simeq \norm{u^+}_{\dom(\modulus{A}^{\frac12})^\ast},
\end{align*}
where the first inequality follows from Lemma~\ref{Lem:Frac} since $\modulus{A}_\epsilon$ has an \Hinfty-functional calculus.
To finish the proof,
\begin{align*}
\int_{0}^\infty \norm{ 2 \eta_\rho(t) \modulus{A}_\epsilon \exp(-t\modulus{A}_\epsilon)u^-}^2\ dt & \lesssim \int_{0}^\infty \norm{t^{\frac{1}{2}}\modulus{A}_\epsilon^{\frac{1}{2}} \exp(-t\modulus{A}_\epsilon) \modulus{A}_\epsilon^{\frac{1}{2}} u^-}^2\ dt \\
                                                                                                  & \lesssim \norm{\modulus{A}_\epsilon^{\frac{1}{2}}u^-} \simeq \norm{u^-}_{ \dom(\modulus{A}^{\frac12})},
\end{align*}
where again, we have alluded to the \Hinfty-functional calculus of $\modulus{A}_\epsilon$.
Combining these estimates proves the first claim.

For the second claim, note that $(\sigma_0^{-1} D_0)^\dagger = -(\partial_t - A)$.
Therefore, we repeat this argument with $-A$ in place of $A$.
Note that $\modulus{-A}  = \sqrt{ (-A)^2} = \sqrt{A^2} = \modulus{A}$, so $\ext_\rho^\ast = \ext_\rho$.
A repetition of the above estimates with $(\sigma_0^{-1} D_0)$ in place of $D_0$ then yields that, for $v \in \dom(A)$,
$$ \norm{\ext_\rho v} \lesssim \norm{v}_{\checkH(-A)}.$$
By definition $\hatH(A) = \checkH(-A)$ and again, by density, we obtain this estimate for all $v \in \hatH(A)$.
\end{proof}

The following is then argued similarly to Lemma 6.5 in \cite{BBan}, but with some minor modifications to account for the fact that $\dM$ is noncompact.
\begin{corollary}
\label{Cor:TraceBd}
For all $u \in \Ck[c]{\infty}(Z_{[0,T)};E)$
\begin{equation*}
\norm{u\rest{\dM}}_{\checkH(A)} \lesssim \norm{u}_{D_0}\quad\text{and}\quad\norm{u\rest{\dM}}_{\hatH(A)} \lesssim \norm{u}_{(\sym_0^{-1}D_0)^\dagger}.
\end{equation*}
\end{corollary}
\begin{proof}
Since $\spt u$ is compact, Green's formula yields
$$
\inprod{(\sigma_0^{-1} D_0)u, v}_{\Lp{2}(M)} - \inprod{u, (\sigma_0^{-1} D_0)^\dagger v}_{\Lp{2}(M)} 
= 
-\inprod{u\rest{\dM},v\rest{\dM}}_{\Lp{2}(\dM)}
$$
for all $v \in \Ck{\infty}(Z_{[0,T)};E)$.
Take $v_0 \in \Dk{\infty}(A)$, and let $v := \ext_\rho v_0$.
Clearly, $v \in \Ck{\infty}(Z_{[0,T)};E)$  by Lemma~\ref{Lem:ExpReg}.
Since $v_0 = v\rest{\dM}$,
\begin{align*}
\left|\inprod{u\rest{\dM},v_0}_{\Lp{2}(\dM)}\right|
 & \leq \left|\inprod{(\sigma_0^{-1} D_0)u, \ext_\rho v}_{\Lp{2}(M)}\right| + \left| \inprod{u, (\sigma_0^{-1} D_0)^\dagger \ext_\rho v}_{\Lp{2}(M)}\right| \\
 & \leq \norm{(\sigma_0^{-1} D_0)u}_{\Lp{2}(M)} \norm{\ext_\rho v}_{\Lp{2}(M)} +  \norm{u}_{\Lp{2}(M)} \norm{(\sigma_0^{-1} D_0)^{\dagger} \ext_\rho v}_{\Lp{2}(M)}          \\
 & \lesssim \norm{u}_{D_0} \norm{\ext_\rho v}_{(\sigma_0^{-1}D_0)^\dagger}                                                    \\
 & \lesssim \norm{u}_{D_0} \norm{v_0}_{\hatH(A)},
\end{align*}
where the ultimate inequality follows from Proposition~\ref{Prop:ModelBnd}.
Therefore,
$$ \norm{u\rest{\dM}}_{\checkH(A)}
\simeq \sup_{0 \neq v_0 \in \hatH(A)} \frac{\modulus{ \inprod{u,v_0}_{\Lp{2}(\dM)}}}{\norm{v_0}_{\hatH(A)}}
\lesssim \norm{u}_{D_0},$$
where we have used that $\inprod{\cdot,\cdot}_{\Lp{2}(\dM)} = \inprod{\cdot,\cdot}_{\checkH(A) \times \hatH(A)}$ on $\dom(A)$ by Lemma~\ref{Lem:Duality} and the fact that $\dom(A)$ is dense in  $\hatH(A)$.
\end{proof}

Also, the following holds as in \cite{BB12}.
\begin{lemma}
\label{Lem:Ell}
For all $u \in \Ck[c]{\infty}(Z_{[0,T)};E)$, the following equation holds.
$$\norm{(\sym_0^{-1} D_0)u}^2_{\Lp{2}(Z_{[0,T)})}
= \norm{u'}^2_{\Lp{2}(Z_{[0,T)})} + \norm{Au}^2 _{\Lp{2}(Z_{[0,T)})}- \inprod{\modulus{A}\sgn(A) u_0, u_0}_{\Lp{2}(\dM)}.$$
\end{lemma}
\begin{proof}
Writing $\norm{(\sym^{-1}_0 D_0)u}_{\Lp{2}(Z_{[0,T)})}^2 = \inprod{ (\partial_t + A) u, (\partial_t + A)u}_{\Lp{2}(Z_{[0,T)})}$, the conclusion follows from the selfadjoint\-ness of $A$.
\end{proof}

Lastly we note the following.
\begin{lemma}
\label{Lem:BdyHom}
Over the boundary $\dM$, the homomorphism field $(\sym_0^{-1})^\ast: E_{\dM} \to F_{\dM}$ induces an isomorphism $\hatH(A) \to \checkH(\tilde{A})$, where $\tilde{A}$ is the adapted boundary operator for $D^\dagger$.
Also, we have that $(\sym_0^{-1})^\ast: \Dk{\infty}(A) \to \Dk{\infty}(\tilde{A})$.
Moreover, the pairing  $\beta: \hatH(A) \times \checkH(\tilde{A}) \to \C$, obtained by extending  $\beta(u,v) = -\inprod{\sym_0 u,v}_{\Lp{2}(\dM)}$ for $u \in \Lp{2}(\dM;E) \cap \hatH(A)$ and $v \in \Lp{2}(\dM;F) \cap \checkH(\tilde{A})$, is perfect.
\end{lemma}
\begin{proof}
Throughout, note by Assumption~\ref{Hyp:RemControl}, the map $\sigma_0$ is uniformly bounded.
By Assumption~\ref{Hyp:ExtFirst}, we note that $\tilde{A}$ is also the induced adapted operator for $D_0^\dagger$ from $A$, given by $\tilde{A} = -(\sigma_0^{-1})^\ast A \sigma_0^\ast$.
Let $u \in \Dk{\infty}(A)$ and since
$$\tilde{A} (\sigma_0^{-1})u  = -(\sigma_0^{-1})^\ast A \sigma_0^{-1}(\sigma_0^{-1})u = -(\sigma_0^{-1})^\ast A u,$$
we have that $(\sigma_0^{-1})u \in \dom(\tilde{A})$.
By repeated application of this calculation, we conclude that $(\sigma_0^{-1})u \in \Dk{\infty}(\tilde{A})$.

Now, note that $u = (\ext_{\rho} u)\rest{\dM}$ and therefore,
\begin{align*}
\norm{(\sigma_0^{-1})^\ast u}_{\checkH(\tilde{A})}
\lesssim \norm{(\sigma_0^{-1})^\ast \ext_{\rho} u}_{D_0^\dagger}
\simeq \norm{\ext_{\rho} u}_{(\sigma_0^{-1} D_0)^\dagger}
\lesssim \norm{u}_{\hatH(A)},
\end{align*}
where the first inequality follows from applying Corollary~\ref{Cor:TraceBd} to $D_0^\dagger$ and the last inequality from applying Proposition~\ref{Prop:ModelBnd}.
The space $\dom(\tilde{A})$ is dense in $\checkH(\tilde{A})$ and $\Dk{\infty}(A)$ is dense in $\hatH(A)$ by Lemma~\ref{Lem:ExpReg}.
Therefore, this inequality holds for all $u \in \hatH(A)$.

Using a similar argument, interchanging $\ext_\rho$ to $\tilde{\ext}_\rho$, where the latter is the extension map with respect to $\tilde{A}$, we obtain by a similar calculation as before that
$$ \norm{\sigma_0^\ast v}_{\hatH(A)} \lesssim \norm{\sigma_0^\ast \tilde{\ext}_{\rho} v}_{D_0^\dagger} \lesssim \norm{v}_{\checkH(\tilde{A})},$$
for all $v \in \checkH(\tilde{A})$.

Together, this shows that $\sigma_0^\ast$ induced the isomorphism between $\hatH(A)$ and $\checkH(\tilde{A})$ and also proves that the pairing given in the conclusion is perfect.
\end{proof}

\section{The maximal domain}

Let $\dom(D_{\max}; Z_{[0,r]})$ be the maximal domain of $D_{\max}$, considered as an operator in $Z_{[0,r]}$ as defined in the geometric setup in Section~\ref{subsec:GeometricSetup}.
The following result is immediate for compact boundary, but in the noncompact case requires a little work.

\begin{proposition}
\label{Prop:WP}
Let $T$ be as in \ref{Hyp:RemControl}.
There exists $C_{T} > 0$ and $T_c \in (0, T)$ such that:
\begin{enumerate}[label=(\roman*), labelwidth=0pt, labelindent=2pt, leftmargin=21pt]
\item  \label{Prop:WP:3}
      for all $u \in \Dk{\infty}(A)$, we have $\ext_{T_c} u \in \dom(D_{\max})$ and
      $$ \norm{\ext_{T_c} u}_{D} \lesssim \norm{u}_{\checkH(A)};$$
\item \label{Prop:WP:4}
      for all $u \in \Ck[c]{\infty}(Z_{[0,T_c)};E)$, we have
      $$ \norm{u\rest{\dM}}_{\checkH(A)} \lesssim \norm{u}_{D}\quad\text{and}\quad \norm{\sym_0^\ast u\rest{\dM}}_{\hatH(A)} \lesssim \norm{u}_{D^\dagger}.$$
\end{enumerate}
\end{proposition}
\begin{proof}
To prove \ref{Prop:WP:3}, write $D = \sym_t( \sigma_0^{-1} D_0 + R_t)$ where  $R_t$ is at most a differential operator of order $1$.
Therefore, $R_t$ is closable and Assumption~\ref{Hyp:RemControl} then yields that $\dom(A) \subset \dom(\close{R_t})$ since $A$ is essentially selfadjoint and with  $\close{R_t}$ the closure of $R_t$ on $\Ck[c]{\infty}(\dM;E)$.
Since $\sym_t$ is uniformly bounded by Assumption~\ref{Hyp:RemControl},
$$\dom(D_{\max}) \supset \dom(D_{0,\max}) \cap \dom(\close{R_t}; Z_{[0,T_c]}).$$
By Proposition~\ref{Prop:ModelBnd} we have that $\ext_{T_c}u \in \dom(D_{0,\max})$ with the required estimate, and therefore, it suffices to show that $\ext_{T_c}u \in \dom(\close{R_t}; Z_{[0,T_c]})$ with the corresponding estimate.

Now, by a limiting argument, we note that the estimate in Assumption~\ref{Hyp:RemControl} still holds on $\dom(A)$.
In particular, for each $t \geq 0$, $\ext_{T_{c}}u(t) \in \dom(\close{R_t})$.
Therefore, to show that $\ext_{T_c} u \in \dom(\close{R_t}; Z_{[0,T_c]})$, with the required trace bound, we use Assumption~\ref{Hyp:RemControl} to write
\begin{align*}
\norm{R_t \ext_{T_c} u}_{\Lp{2}(Z_{[0,T_c]})}^2
 & = \int_{0}^T  \norm{R_t \ext_{T_c} u}_{\Lp{2}(\dM)}^2\ dt                                                               \\
 & \lesssim  \int_{0}^T t^2 \norm{A\ext_{T_c} u}^2_{\Lp{2}(\dM)}\ dt + \int_{0}^T \norm{\ext_{T_c} u}_{\Lp{2}(\dM)}^2\ dt.
\end{align*}
We note
$$A \ext_{T_c}u = A  \eta_{T_c} \exp(-t\modulus{A}_{\epsilon}) u =  \eta_{T_c} A \exp(-t\modulus{A}_{\epsilon}) u$$
and therefore,
\begin{align*}
\int_{0}^T t^2 \norm{A\ext_{T_c} u}^2_{\Lp{2}(\dM)}\ dt
 & \lesssim \int_{0}^T  \norm{t \modulus{A}_{\epsilon} \exp(-t\modulus{A}_\epsilon) u}^2_{\Lp{2}(\dM)}\ dt                                                                         \\
 & \lesssim \int_{0}^T  \norm{t \modulus{A}_{\epsilon}^{\frac32}  \exp(-t\modulus{A}_\epsilon)  \modulus{A}_{\epsilon}^{- \frac12} u}^2_{\Lp{2}(\dM)}\ dt                          \\
 & \lesssim \int_{0}^\infty  \norm{t^{\frac32} \modulus{A}_{\epsilon}^{\frac32}  \exp(-t\modulus{A}_\epsilon)  \modulus{A}_{\epsilon}^{- \frac12} u}^2_{\Lp{2}(\dM)}\ \frac{dt}{t} \\
 & \simeq \norm{\modulus{A}_{\epsilon}^{-\frac12} u}_{\Lp{2}(\dM)}                                                                                                                 \\
 & \simeq\norm{u}_{\dom(\modulus{A}^{\frac12})}                                                                                      \\
 & \lesssim \norm{u}_{\checkH(A)}.
\end{align*}
For the remaining term,
\begin{align*}
\int_{0}^T \norm{\ext_{T_c} u}_{\Lp{2}(\dM)}^2\ dt
 & \lesssim \int_{0}^\infty \norm{\exp(-t\modulus{A}_{\epsilon}) u}_{\Lp{2}(\dM)}^2\ dt                                                                        \\
 & =  \int_{0}^\infty \norm{\modulus{A}^{\frac12}_{\epsilon} \exp(-t\modulus{A}_{\epsilon})  \modulus{A}_{\epsilon}^{-\frac12} u}_{\Lp{2}(\dM)}^2 \frac{dt}{t} \\
 & \simeq \norm{\modulus{A}_{\epsilon}^{-\frac12} u}_{\Lp{2}(\dM)}                                                                                             \\
 & \lesssim \norm{u}_{\checkH(A)}.
\end{align*}
Combining these estimates,
$$
\norm{R_t \ext_{T_c} u}_{\Lp{2}(Z_{[0,T_c]})}^2 \lesssim \norm{u}_{\checkH(A)},$$
and from Proposition~\ref{Prop:ModelBnd}, we have that
$$\norm{\sigma_t \sigma_0^{-1} D_0 \ext_{T_c} u} \simeq \norm{D_0 \ext_{T_c} u} \lesssim \norm{u}_{\checkH(A)}.$$
On extension by $0$, we have that  $\ext_{T_c} u \in \dom(D_{\max})$ and
$$ \norm{\ext_{T_c}u}_{D} = \norm{\ext_{T_c}u}_{\dom(D_{\max};Z_{[0,T_c]})} \lesssim \norm{u}_{\checkH(A)}.$$

For Assertion~\ref{Prop:WP:4}, we first note that, by using $\tilde{\ext}_{T_c}$, the corresponding extension operator with respect to $\tilde{A}$, we obtain Assertion~\ref{Prop:WP:3} for the operator $D^\dagger$.
Let $v \in \Dk{\infty}(A)$ and let $w := (\sigma_0^{-1})^\ast v$.
By Lemma~\ref{Lem:BdyHom}, we have that $w \in \Dk{\infty}(\tilde{A})$ and therefore,
$$ \norm{\tilde{\ext}w}_{D^\dagger} \lesssim \norm{w}_{\checkH(\tilde{A})} \simeq \norm{v}_{\hatH(A)},$$
where in the first inequality we have used Assertion~\ref{Prop:WP:4} for $D^\dagger$ and in the equivalence we have used Lemma~\ref{Lem:BdyHom}.
Therefore, fixing $u \in \Ck[c]{\infty}(Z_{[0,T_c)};E)$ extended by $0$ to the whole of $M$ and using \eqref{A:IntParts}, we obtain
$$ -\inprod{u\rest{\dM},v}_{\checkH(A) \times \hatH(A)} =  -\inprod{u\rest{\dM},\sym_0^\ast w}_{\checkH(A) \times \hatH{A}} = \inprod{Du, \tilde{\ext}_{T_c} w} - \inprod{u, D^\dagger \tilde{\ext}_{T_c} w}.$$
Therefore,
\begin{align*}
\inprod{u\rest{\dM},v}_{\checkH(A) \times \hatH(A)}
\lesssim \norm{Du} \norm{\tilde{\ext}_{T_c} w} + \norm{u} \norm{D^\dagger \tilde{\ext}_{T_c} w}
 & \lesssim \norm{u}_{D} \norm{\tilde{\ext}_{T_c} w}_{D^\dagger} \\
\lesssim \norm{u}_{D} \norm{w}_{\checkH(\tilde{A})}
\simeq \norm{u}_{D} \norm{v}_{\hatH(A)}.
\end{align*}
Since such $v$ are dense in $\hatH(A)$, we obtain that
\[
\norm{u\rest{\dM}}_{\checkH(A)} \lesssim \norm{u}_{D}.
\]
The statement $\norm{u\rest{\dM}}_{\hatH(A)} \lesssim \norm{u}_{ D^\dagger}$ follows by a similar duality argument to establish $\norm{u\rest{\dM}}_{\checkH(\tilde{A})} \lesssim \norm{u}_{\tilde{D}}$ and then using Lemma~\ref{Lem:BdyHom} which asserts that $\norm{\sym_0^\ast u}_{\hatH(A)} \simeq \norm{u}_{\checkH(\tilde{A})}$.
\end{proof}

With the combination of these facts, we obtain the proof of the main theorem.
\begin{proof}[Proof of Theorem~\ref{Thm:MaxDom}]
For Statement~\ref{Thm:MaxDom:1} choose a cutoff function $\xi\in \Ck[c]{\infty}([0,\infty))$ so that $\xi = 1$ on $[0,\frac{1}{2}T]$ and $\xi = 0$ on $[\frac{3}{4}T,T]$.
On taking a section $u \in \Ck[c]{\infty}(M;E)$, we have that
$$ \norm{u \rest{\dM}}_{\checkH(A)} = \norm{(\xi u)\rest{\dM}}_{\checkH(A)} \lesssim \norm{ \xi u}_{D} \lesssim \norm{u}_{D}.$$
Here the penultimate inequality follows from Proposition~\ref{Prop:WP}~\ref{Prop:WP:4} and the ultimate inequality uses
$$\modulus{\sym_D (x,d \xi)} = \modulus{\sym_{D}(x,\partial_t \xi dt)} = \modulus{\partial_t \xi} \modulus{ \sym_{D}(x,\partial_t \xi dt)} \lesssim \modulus{\partial_t \xi(t)} \lesssim 1$$
for $x \in Z_{[0,T)}$ from \ref{Hyp:RemControl}.
The density of $\Ck[c]{\infty}(M;E)$ in $\dom(D_{\max})$ as asserted in Lemma~\ref{Lem:Density}, yields the statement that the trace map extends boundedly from $\dom(D_{\max}) \to \checkH(A)$.
Since $\Dk{\infty}(A)$ is dense in $\checkH(A)$, we obtain that
$$ \norm{\ext_{T_c} u} \lesssim \norm{u}_{\checkH(A)}$$
by applying Proposition~\ref{Prop:WP}~\ref{Prop:WP:3}.
This furnishes us with the fact that the boundary trace map is a surjection.
The argument for $D^\dagger$ is analogous.

Next we prove Assertion~\ref{Thm:MaxDom:1.7}.
For that, fix $u_0 \in \checkH(A)$.
By Assertion~\ref{Thm:MaxDom:1}, we are guaranteed there exists $u \in \dom(D_{\max})$ with $u\rest{\dM} = u_0$.
By the density of $\Ck[c]{\infty}(M;E)$ in $\dom(D_{\max})$, there exists $u_n \to u$ in $\dom(D_{\max})$ with $u_n \in \Ck[c]{\infty}(M;E)$.
Since we have obtained the extension of the boundary trace map as a bounded surjection, we obtain
$$\norm{u_0 - (u_n)\rest{\dM}}_{\checkH(A)} \lesssim \norm{u - u_n}_D \to 0.$$
Clearly, $(u_n)\rest{\dM} \in \Ck[c]{\infty}(\dM;E) = \Ck[c]{\infty}(\dM;E)$.
This shows that $\Ck[c]{\infty}(\dM;E)$ is dense in $\checkH(A)$.
The corresponding statement for $\hatH(A)$ follows from applying this for $D^\dagger$ to obtain the density of $\Ck[c]{\infty}(\dM;F)$ in $\checkH(\tilde{A})$, noting that $\sym_0^\ast:\Ck[c]{\infty}(\dM;F) \to \Ck[c]{\infty}(\dM;F)$ bijectively, and then using Lemma~\ref{Lem:BdyHom} to pull this density across to $\hatH(A)$ via $\sym_0^\ast$.

Assertion~\ref{Thm:MaxDom:2} is the well-known Green's formula for compactly supported smooth sections.
Density of compactly supported sections in $\dom(D_{\max})$ and $\dom((D^{\dagger})_{\max})$ as ensured by Lemma~\ref{Lem:Density} yields the claim.

Lastly, we prove  Assertion~\ref{Thm:MaxDom:1.5}.
First, note  if $u \in \dom(D_{\min})$, then there exists $u_n \in \Ck[cc]{\infty}(M;E)$ with $u_n \to u$ in the graph norm of $D$.
Without loss of generality, we can assume that $\spt u,\ \spt u_n \subset Z_{[0,T)}$.
Then, by Proposition~\ref{Prop:WP:3}, we obtain
$$\norm{u\rest{\dM}}_{\checkH(A)} = \norm{u\rest{\dM} - u_n\rest{\dM}}_{\checkH(A)} \lesssim \norm{u - u_n}_{D} \to 0$$
as $n \to \infty$.
Hence, $u \rest{\dM} = 0$.

Now, suppose that $u\rest{\dM} = 0$.
Then, by invoking \ref{Thm:MaxDom:2}, we obtain that $\inprod{u, D_{\max}^{\dagger} v}  = \inprod{D_{\max} u,v}$ for all $v \in \dom(D_{\max}^\dagger)$.
Therefore, we have that $\modulus{\inprod{u, D_{\max}^\dagger v}} \lesssim \norm{v}$, which yields $u \in \dom((D_{\max}^\dagger)^\ast)$.
However, by construction $D_{\max}^\dagger = D_{\min}^\ast$ and since $D_{\min}$ is closed, $(D_{\max}^\dagger)^\ast = D_{\min}^{\ast\ast} = D_{\min}$.
Therefore, $u \in \dom(D_{\min})$.
\end{proof}

\section{Boundary value problems}
\label{Sec:BVPs}

\subsection{Boundary conditions}
As a consequence of Theorem~\ref{Thm:MaxDom}, given a subspace $B \subset \checkH(A)$, we can define an associated operator $D_{B,\max}$ with domain  $$\dom(D_{B,\max}) = \set{u \in \dom(D_{\max}): u\rest{\dM} \in B}.$$
Conversely, given an operator $D'$ satisfying $D_{\min} \subset D' \subset D_{\max}$, define
$$B' = \set{u\rest{\dM}: u \in \dom(D')}.$$
With this notation, we obtain the following important lemma.

\begin{lemma}
\label{Lem:OpBdy}
\begin{enumerate}[label=(\roman*), labelwidth=0pt, labelindent=2pt, leftmargin=21pt]
\item \label{Lem:OpBdy:1}
      If $D'$ is an operator such that $D_{\min} \subset D' \subset D_{\max}$, we have that
      $$\dom(D') = \set{u \in \dom(D_{\max}): u\rest{\dM} \in B'}.$$
\item \label{Lem:OpBdy:2}
      The operator $D_{B,\max}$ is closed if and only if $B$ is closed.
\end{enumerate}
\end{lemma}
\begin{proof}
The lemma follows immediately from the trace map inducing an isomorphism
\begin{equation*}
\faktor{\dom(D_{\max})}{\dom(D_{\min})} \to \checkH(A),
\end{equation*}
see Remark~\ref{rem:maxmin}.
\end{proof}

\begin{definition}
A boundary condition for $D$ is a \emph{closed} subspace $B \subset \checkH(A)$.
The associated operators are then given by $D_{B,\max}$ and $D_{B}$, with domains
\begin{gather*}
\dom(D_{B,\max}) = \set{ u \in \dom(D_{\max}): u\rest{\dM} \in B} \quad\text{and}\\
\dom(D_{B}) = \set{ u \in \dom(D_{\max}) \intersect \SobH[loc]{1}(M;E): u \rest\dM \in B} .
\end{gather*}
\end{definition}

As a consequence of Lemma~\ref{Lem:OpBdy}, for any boundary condition $B$, the operator $D_{B,\max}$ is closed.

\begin{proposition}
A boundary condition $B$ satisfies $B \subset \SobH[loc]{\frac{1}{2}}(\dM;E)$ if and only if $D_{B} = D_{B,\max}$.
\end{proposition}
\begin{proof}
It is clear that we always have $D_{B} \subset D_{B,\max}$.
Assuming that $B \subset \SobH[loc]{\frac{1}{2}}(\dM;E)$, we prove that $D_{B,\max} \subset D_{B}$.
For that, let $u \in D_{B,\max}$, and so we have that $u \in \SobH[loc]{\frac{1}{2}}(\dM;E)$.
By Theorem~\ref{Thm:Reg}, we then obtain that $u \in \dom(D_{\max}) \intersect \SobH[loc]{1}(M;E)$.
This shows that $u \in \dom(D_{B})$.

For $u \in \dom(D_B)$, we have that $u\rest{\dM} \in \SobH[loc]{\frac{1}{2}}(\dM;E) \intersect B$.
Given that we assume $D_{B} = D_{B,\max}$, we have by Lemma~\ref{Lem:OpBdy}~\ref{Lem:OpBdy:1} that $u \in \dom(D_{B,\max})$ if and only if  $u\rest{\dM} \in B$ which can only happen if $B \subset \SobH[loc]{\frac{1}{2}}(\dM;E)$.
\end{proof}

Given a boundary condition $B$, operator theory implies that the adjoint operator $D_{B,\max}^\ad$ satisfies:
$$D_{B,\max}^\ad \subset (D_{\min})^\ad = D^\dagger_{\max}.$$
Therefore, as a consequence of Theorem~\ref{Thm:MaxDom}~\ref{Thm:MaxDom:2}, we define
\begin{equation}
\label{Eq:Bad}
B^{\ad} := \set{v \in \checkH(\tilde{A}): \inprod{\sym_0 u, v}_{\hatH(\tilde{A}) \times \checkH(\tilde{A})} = 0\ \quad\forall u \in B}.
\end{equation}

\begin{proposition}
\label{Prop:AdjBC}
The space  $$\sym_0^\ast B^{\ad} =  B^{\perp, \hatH(A)} = \set{v \in \hatH(A): \inprod{u,v}_{\checkH(A) \times \hatH(A)} = 0\ \ \forall u \in B}$$ and  $D_{B,\max}^\ad$ satisfies:
$$ \dom(D_{B,\max}^\ad)  = \set{v \in \dom( (D^\dagger)_{\max} ): v \rest{\dM} \in B^\ad}.$$
\end{proposition}
\begin{proof}
Using Lemma~\ref{Lem:Duality} and calculating:
\begin{align*}
B^\ad & = \set{v \in \checkH(\tilde{A}): \inprod{\sym_0 u, v}_{\hatH(\tilde{A}) \times \checkH(\tilde{A})} = 0\ \quad\forall u \in B} \\
      & =  \set{v \in \checkH(\tilde{A}): \inprod{u, \sym_0^\ast v}_{\checkH(A) \times \hatH(A)} = 0\ \quad\forall u \in B}           \\
      & =  \set{(\sym_{0}^{-1})^\ast w \in \checkH(\tilde{A}): \inprod{u, w}_{\checkH(A) \times \hatH(A)} = 0\ \quad\forall u \in B}  \\
      & = (\sym_{0}^{-1})^\ast \set{ w \in \hatH(A): \inprod{u, w}_{\checkH(A) \times \hatH(A)} = 0\ \quad\forall u \in B}.
\end{align*}

The remaining assertion follows directly from Theorem~\ref{Thm:MaxDom}, noting that
$$ \inprod{D_{B,\max}u,v} = \inprod{u, D_{B,\max}^\ad v}$$
for all $u \in \dom(D_{B,\max})$ and $v \in \dom(D_{B,\max}^\ad)$, along with Lemma~\ref{Lem:OpBdy}.
\end{proof}

\begin{definition}
\label{Def:EllReg}
We say that a boundary condition $B$ is \emph{elliptically semi-regular} if $B \subset \SobH[loc]{\frac12}(\dM;E)$ and \emph{elliptically regular} if $B$ is elliptically semi-regular and $B^\ad$ is also elliptically semi-regular.
\end{definition}

\begin{remark}
\label{Rem:EllReg}
Note that since $\sym_0: \Ck{\infty}(M;F) \to \Ck{\infty}(M;F)$, the elliptic semi-regularity of $B^\ad$ is equivalent to the assertion that $B^{\perp, \hatH(A)} \subset \SobH[loc]{\frac12}(\dM;E)$.
\end{remark}

Then we immediately obtain the following.

\begin{corollary}
\label{Cor:EllReg}
A boundary condition $B$ is elliptically regular if and only if $\dom(D_{B,\max}) \subset \SobH[loc]{1}(M;E)$ and $\dom(D_{B,\max}^\ast) \subset \SobH[loc]{1}(M;F)$.\hfill\qed
\end{corollary}

Note that in the compact case, the assertion of an elliptically (semi-)regular boundary condition can be formulated in terms of the adapted operator $A$.
Namely, since $\SobH{\alpha}(\dM;E) = \dom(\modulus{A}^\alpha)$, we can see from Definition~\ref{Def:EllReg} and Corollary~\ref{Cor:EllReg} that the spaces $\SobH{\alpha}(\dM;E)$ for $\alpha = \frac12$ and $\alpha = -\frac12$  can be described in terms of domains $\dom(\modulus{A}^{\frac12})$ and $\dom(\modulus{A}^{\frac12})^\ast$.
Clearly, in the noncompact setting, it is a futile effort to attempt to identify the spaces $\SobH[loc]{\alpha}(\dM;E)$, which are locally convex linear spaces, with domains of powers of $A$, which are Banach spaces.
That being said,  it is important in applications to consider the following class of boundary conditions whose regularity is captured in terms of $A$.

\begin{definition}
\label{Def:AEllReg}
We say a boundary condition $B$ is \emph{$A$-elliptically semi-regular} if $B \subset \dom(\modulus{A}^{\frac12})$.
If further $\sigma_0^\ast B^\ast \subset \dom(\modulus{A}^{\frac12})$, then we say that it is \emph{$A$-elliptically regular}.
\end{definition}

\begin{remark}
If $B$ is $A$-elliptically regular, then $B^*$ is $\tilde{A}$-elliptically regular.
\end{remark}

As the following corollary demonstrates, this is a more restricted class of elliptically (semi-)regular boundary conditions. 

\begin{corollary}
\label{Cor:AEllReg}
Every $A$-elliptically (semi-)regular boundary condition is elliptically (semi-)regular.
\end{corollary}
\begin{proof}
From elliptic regularity, $\dom(\modulus{A}^{\frac12}) \subset \SobH[loc]{\frac12}(\dM;E)$.
Therefore, if $B$ is elliptically semi-regular, then $B \subset \dom(\modulus{A}^{\frac12})$ and so in particular, $B \subset \SobH[loc]{\frac12}(\dM;E)$.
Similarly for $\sym_0^\ast B^\ast$.
\end{proof}

An important way in which the $A$-elliptic regularity can be detected in practice is captured in the following.
\begin{proposition}
\label{Prop:ProjBCAReg}
Suppose that $\sym_0B = B$, $\sym_0^\ast B^{\perp, \hatH(A)} = B^{\perp, \hatH(A)}$ and that $\sym_0 A = - A \sym_0$.
Then $B$ is $A$-elliptically regular.
\end{proposition}
\begin{proof}
Since $B \subset \checkH(A)$, for $u \in B$, we have that $\chi^-(A)u \in \dom(\modulus{A}^{\frac12})$.
We also have that $\sym_0 u \in B$,
$$\dom(\modulus{A}^{\frac12}) \ni \chi^-(A)\sym_0u = \sym_0 \chi^+(A) u,$$
where in the equality, we have used that $\sym_0 A = - A \sym_0$.

To prove that $\chi^+(A)u \in \dom(\modulus{A}^{\frac12})$, it suffices to show  $\sym_0: \dom(\modulus{A}^{\frac12}) \to \dom(\modulus{A}^{\frac12})$ is a Banach space isomorphism.
For that, first note that from \ref{Hyp:RemControl}, we have that $\sym_0: \Lp{2}(\dM;E) \to \Lp{2}(\dM;E)$ is an isomorphism.
Next,
$$ \norm{\modulus{A} \sym_0u } = \norm{A\sgn(A)\sym_0 u} = \norm{A (- \sym_0 \sgn(A) u)} = \norm{\sym_0 \modulus{A}u} \simeq \norm{\modulus{A}u}.$$
This yields $\sym_0: \dom(\modulus{A}) \to \dom(\modulus{A})$ is a Banach space isomorphism.
Therefore, through interpolation, we obtain that $\sym_0: \dom(\modulus{A}^{\frac12}) \to \dom(\modulus{A}^{\frac12})$ is a Banach space isomorphism.
Hence, $\chi^+(A)u \in \dom(\modulus{A}^{\frac12})$ which shows that $B \subset \dom(\modulus{A}^{\frac12})$.

It remains to show that $B^{\perp, \hatH(A)} \subset \dom(\modulus{A}^{\frac12})$.
For that, we first note that by taking adjoints and using selfadjointness of $A$, $\sym_0^\ast A = - A \sym_0^\ast$.
As before, a direct interpolation argument allows us to assert that $\sym_0^\ast: \dom(\modulus{A}^{\frac12}) \to \dom(\modulus{A}^{\frac12})$ is a Banach space isomorphism.

Fix $u \in B^{\perp, \hatH(A)}$.
By construction, we have that $\chi^+(A)u \in \dom(\modulus{A}^{\frac12})$.
Using the condition $\sym_0^\ast B^{\perp, \hatH(A)} = B^{\perp, \hatH(A)}$, we have a $v \in B^{\perp ,\hatH(A)}$ such that $u = \sym_0^\ast v$.
Then, $\chi^-(A) u = \chi^-(A)\sym_0^\ast v = \sym_0^\ast \chi^+(A)v \in \dom(\modulus{A}^{\frac12})$ since $\sym_0^\ast$ preserves $\dom(\modulus{A}^{\frac12})$ and $\chi^+(A) v \in \dom(\modulus{A}^{\frac12})$ as $v \in \hatH(A)$.

Since we have shown that $B \subset \dom(\modulus{A}^{\frac12})$ and $B^{\perp, \hatH(A)} \subset \dom(\modulus{A}^{\frac12})$, the boundary condition $B$ is $A$-elliptically regular.
\end{proof}

\begin{remark}
Note that both properties $\sym_0 B = B$ and $\sym_0^\ast B^{\perp, \hatH(A)}$ are required and this is what contributes to the nontriviality of this statement.
For instance, if $B = 0$ then $B^{\perp, \hatH(A)} = \hatH(A)$.
The condition  $\sym_0A = - A \sym_0$ means that $\sym_0: \checkH(A)  \to \hatH(A)$ but $\hatH(A) \neq \checkH(A)$.
From this, it is easy to see that $B = 0$ is not an $A$-elliptically regular boundary condition.
\end{remark}

In the context of ``geometric'' first-order operators, $A$-elliptically regular  boundary conditions imply important approximation properties.
These properties are important in calculations, particularly in determining Fredholm extensions.
Before we present the main approximation result,  we note the following lemma.

\begin{lemma}
\label{Lem:Approx}
For $\chi \in \Ck[c]{\infty}(\dM;\R)$, we have that $\chi:\dom(\modulus{A}^{\frac12}) \to \dom(\modulus{A}^{\frac12})$ with
\begin{equation}
\label{Eq:HalfNorm}
\norm{\chi y}_{\dom(\modulus{A}^{\frac12})} \leq C' \norm{\chi}_{\Lp{\infty}}^{\frac12}(\norm{\chi}_{\Lp{\infty}(\dM)} + C^2 \norm{\sym_{A}(d \chi)}_{\Lp{\infty}(\dM;TM)})^{\frac12} \norm{y}_{\dom(\modulus{A}^{\frac12})},
\end{equation}
where $C' < \infty$ is a universal constant.
\end{lemma}
\begin{proof}
We obtain this through interpolation, noting that, due to the fact that $\modulus{A}$ is nonnegative selfadjoint, $\dom(\modulus{A}^{\frac12}) = [\Lp{2}(\dM;E), \dom(\modulus{A})]_{\theta = \frac12}$.
First,
$$\norm{\chi y}_{\Lp{2}(\dM)} \leq \norm{\chi}_{\Lp{\infty}(\dM)} \norm{y}_{\Lp{2}(\dM)}.$$
Also, we have a constant $C_0$ such that
\begin{align*}
\tfrac{1}{C_0} \norm{\chi y}_{\dom(\modulus{A})}^2
 & \leq \norm{\modulus{A} (\chi y)}_{\Lp{2}(\dM)} + \norm{\chi y}_{\Lp{2}(\dM)}                                    \\
 & = \norm{\sgn(A) A (\chi y)}_{\Lp{2}(\dM)} + \norm{\chi y}_{\Lp{2}(\dM)}                                         \\
 & \leq \norm{ \sym_{A}(d \chi)y  + \chi Ay}_{\Lp{2}(\dM)} + \norm{\chi}_{\Lp{\infty}(\dM)} \norm{y}_{\Lp{2}(\dM)} \\
 & \leq \norm{\sym_{A}(d \chi)}_{\Lp{\infty}(\dM)} \norm{y} + \norm{\chi}_{\Lp{2}(\dM)} (\norm{Ay} + \norm{y})         \\
 & \leq C_0 (C^2 \norm{\nabla \chi}_{\Lp{2}(\dM)} +  \norm{\chi}_{\Lp{2}(\dM)}) \norm{y}_{\dom(\modulus{A})}
\end{align*}
where in the third line we used the locality of $A$ and in the ultimate inequality that $\modulus{\sym_{A}(\xi)} \leq C^2 \norm{\xi}$.
By interpolation, the norms also interpolate and hence, we obtain \eqref{Eq:HalfNorm}.
\end{proof}

\begin{proposition}
\label{Prop:AEllRegApprox}
Suppose that $M$ carries a complete Riemannian metric $g$  and constant $C < \infty$ such that $\modulus{ \sym_{D}(\xi)}_{h^E \to h^E} \leq C \modulus{\xi}_{g}$.
Then, if $u \in \dom(D_{\max})$ and $u\rest{\dM} \in \dom(\modulus{A}^{\frac12})$, there exists a sequence $u_n \in \Ck[c]{\infty}(M;E)$ such that $u_n \to u$ in the graph norm of $D$ and $u_n\rest{\dM} \to u\rest{\dM}$ in $\dom(\modulus{A}^{\frac12})$.
\end{proposition}
\begin{proof}
First, note that by Corollary~\ref{Cor:Dinf}, we have that $\Dk{\infty}(A) = \cap_{k} \dom(\modulus{A}^k)$ is dense in $\checkH(A)$.
Therefore, Proposition~\ref{Prop:WP}~\ref{Prop:WP:3} is valid on all of $\checkH(A)$ and in particular on $\dom(\modulus{A}^{\frac12})$.

Fix $u$ as in the hypothesis.
In order to obtain smoothness up to the boundary for the approximating sequence, we first note that we can find $u^0_k \in \Ck{\infty}(\dM;E) \cap \dom(\modulus{A}^{\frac12})$ such that $u^0_k \to u\rest{\dM}$ in the $\dom(\modulus{A}^{\frac12})$ norm.
For instance,
$$u^0_k := \exp\cbrac{\tfrac{1}{k} \modulus{A}}(u\rest{\dM})$$
defines such a sequence where  the required convergence is a consequence of the fact that the semigroup commutes with $\modulus{A}^{\frac12}$.

Now, let
$$v^0_k := \ext_{T_c}(u^0_k)\quad\text{and}\quad v := \ext_{T_c}(u\rest{\dM}).$$
By Proposition~\ref{Prop:WP}~\ref{Prop:WP:3}, we have that both $v, v^0_k \in \dom(D_{\max)}$.
Moreover, $v\rest{\dM} = u \rest{\dM}$ and therefore, $w := u -v \in \dom(D_{\min})$.
Therefore, there exists $w_n \in \Ck[cc]{\infty}(M;E)$ such that $w_n \to w$ in $\dom(D_{\max})$.

Moreover, by the properties of $u_0^k$, we have in addition that $v^0_k \in \Ck{\infty}(M;E)$.
By design, we have that $v^0_k \to v$ in $\dom(D_{\max})$ and also that $v^0_k\rest{\dM} \to u\rest{\dM}$ in $\dom(\modulus{A}^{\frac12})$.

Fixing a base point $p \in M$, by the hypothesis of metric completeness of $g$, we can find $\chi_m \in \Ck[c]{\infty}(M;[0,1])$ with $\chi_n = 1$ on $B_{m}(p)$ (where the radius of the geodesic ball is $n$) and $\chi_m = 0$ outside of $B_{10m}$ further satisfying:
$$ \modulus{\nabla \chi_m}_{g} \leq \frac{C'}{m}\quad\text{and}\quad \modulus{\chi_m - 1} \to 0$$
everywhere on $M$.

Let $v_{m,n} = \chi_m v_n^0$.
By construction, $v_{m,n} \in \Ck[c]{\infty}(M;E)$.
Moreover, $v_{m,n}\rest{\dM} = \chi_m\rest{\dM} v_n^0\rest{\dM} = \chi_m u_n^0 $.
Therefore,
\begin{align*}
 \norm{(u  - v_{m,n})\rest{\dM}}_{\dom(\modulus{A}^{\frac12})}                                                                 
 \leq \norm{(\chi_m -1)(u_n^0)}_{\dom(\modulus{A}^{\frac12})}  + \norm{(u_n^0 -u)\rest{\dM}}_{\dom(\modulus{A}^{\frac12})}
\end{align*}
Clearly, the second term tends to zero as $n \to \infty$, so we examine the first.
For a fixed $n$,
\begin{align*}
 & \norm{(\chi_m -1)(u_n^0)\rest{\dM}}_{\dom(\modulus{A}^{\frac12})}                                                                                                                                    \\
 & \quad\leq \norm{(\chi_m - 1)}_{\Lp{\infty}(M)}^{\frac12}(\norm{\chi_m - 1}_{\Lp{\infty}(\dM)} + C^2 \norm{\nabla \chi_m}_{\Lp{\infty}(\dM;TM)})^{\frac12} \norm{u_n^0}_{\dom(\modulus{A}^{\frac12})} \\
 & \quad\leq \norm{(\chi_m - 1)}_{\Lp{\infty}(M)}^{\frac12}(\norm{\chi_m - 1}_{\Lp{\infty}(\dM)} + C^2C')^{\frac12} \norm{u_n^0}_{\dom(\modulus{A}^{\frac12})}
\to 0,
\end{align*}
by applying \eqref{Eq:HalfNorm} with $\chi_m -1 $ in place of $\chi$.
Also, $v_{m,n} \to v^0_{n}$ as $m \to \infty$ in $\dom(D_{\max})$ and $v^0_{n} \to v$ in $\dom(D_{\max})$.
Therefore, through a diagonal argument, we can find a subsequence $v_n := v_{f(n),n}$ (where $f$ is chosen via the diagonal argument) such that $v_n \to v$ in $\dom(D_{\max})$ and $v_n\rest{\dM} \to u\rest{\dM}$ in $\dom(\modulus{A}^{\frac12})$.

Define $u_n = v_n + w_n$ and note that $u_n \in \Ck[c]{\infty}(M;E)$.
Since $v_n \to v$ and $w_n \to u - v$ in $\dom(D_{\max})$, it is clear that $u_n \to u$ in $\dom(D_{\max})$.
Now, $u_n \rest{\dM} = v_n \rest{\dM}$ since $w_n \rest{\dM} = 0$.
This finishes the proof.
\end{proof}

Under this assumption, an $A$-elliptically semi-regular boundary conditions further enjoys the following regularity property at the level of the domain of the operator. 
\begin{corollary}
\label{Cor:ASemiRegNearBoundary} 
Assume the hypothesis of Proposition~\ref{Prop:AEllRegApprox}. 
Then, there exists  $T_d \in (0,T)$ such that  for any $A$-elliptically-semi-regular boundary condition $B$ and $u \in \dom(D_B)$ with $\spt u \subset Z_{[0,T_d)}$, we have
$$ \norm{u} +  \norm{\partial_t u} + \norm{A u} \lesssim \norm{u}_{\dom(D_B)}$$ 
where the implicit constant depends on $B$. 
\end{corollary}
\begin{proof}
Let $T_d \in (0,T)$ be chosen later.
From Lemma~\ref{Lem:Ell}, for $v \in \Ck[c]{\infty}(Z_{[0,T_d)})$, 
\begin{equation}
\label{Eq:ASemiEst} 
\begin{aligned}
\norm{\partial_t v}^2_{\Lp{2}(Z_{[0,T_d)})} &+ \norm{Av}^2 _{\Lp{2}(Z_{[0,T_d)})}  \\
&\leq
\modulus{\inprod{\modulus{A}\sgn(A) v\rest{\dM}, v\rest{\dM}}_{\Lp{2}(\dM)}} + \norm{\sym_0^{-1} D_0v}^2_{\Lp{2}(Z_{[0,T_d)})} \\
&\leq
\norm{\modulus{A}^{\frac12} v\rest{\dM}}_{\Lp{2}(\dM)}^2 + \norm{\sym_0^{-1} D_0 v}^2_{\Lp{2}(Z_{[0,T_d)})} \\ 
&\leq
\norm{\modulus{A}^{\frac12} v\rest{\dM}}_{\dom(\modulus{A}^{\frac12})}^2 + \norm{\sym_0^{-1} D_0 v}^2_{\Lp{2}(Z_{[0,T_d)})}.
\end{aligned}
\end{equation}
Now, since $D = \sym_t(\partial_t + A + R_t) = \sym_t(\sym_0^{-1} D_0 + R_t)$, we can write $\sym_0^{-1} D_0 = \sym_t^{-1} D - R_t$.
Then, the last term in Eq.~\eqref{Eq:ASemiEst} is:  
\begin{equation}
\label{Eq:ASemiEst1}  
\begin{aligned} 
\norm{\sym_0^{-1} D_0 v}^2_{\Lp{2}(Z_{[0,T_d)})} 
&= 
\int_{0}^{T_d} \norm{(\sym_t^{-1} D - R_t)v}_{\Lp{2}(\dM)}^2\ dt \\ 
&\leq
C_1 \int_{0}^{T_d}  \norm{Dv}^2_{\Lp{2}(\dM)}\ dt + C_1 \int_{0}^{T_d}  \norm{R_t v}^2_{\Lp{2}(\dM)}\ dt \\
&\leq 
C_1 \norm{Dv}^2_{\Lp{2}(Z_{[0,T_d)})} + C_2 \int_{0}^{T_d}  t^2 \norm{Av}^2_{\Lp{2}(\dM)}\ dt  + C_2 \int_{0}^{T_d} \norm{v}_{\Lp{2}(\dM)}^2\ dt \\
&\leq 
C_1 \norm{Dv}^2_{\Lp{2}(Z_{[0,T_d)})} + C_2 T_d^2 \int_{0}^{T_d} \norm{Av}^2_{\Lp{2}(\dM)}\ dt  + C_2 \norm{v}_{\Lp{2}(Z_{[0,T_d)})}^2 \\
&\leq
\max\set{C_1, C_2} \cbrac{\norm{Dv}^2_{\Lp{2}(Z_{[0,T_d)})} + \norm{v}_{\Lp{2}(Z_{[0,T_d)})}^2} + C_2 T_d^2 \int_{0}^{T_d} \norm{Av}^2_{\Lp{2}(\dM)}\ dt.
\end{aligned} 
\end{equation}
The constants $C_1$ and $C_2$ depend only on universal constants and the constant $C$ from Assumption~\ref{Hyp:RemControl}.

Now, choose $T_d < \min\set{ T, \tfrac{1}{\sqrt{2C_2}}}$, so that $C_2 T_d^2 < \tfrac12$.
Then, we find 
\begin{equation*}
\norm{\sym_0^{-1} D_0 v}^2_{\Lp{2}(Z_{[0,T_d)})} 
\leq
 \max\set{C_1, C_2} \cbrac{\norm{Dv}^2_{\Lp{2}(Z_{[0,T_d)})} + \norm{v}_{\Lp{2}(Z_{[0,T_d)})}^2} + \frac12  \norm{Av}^2_{\Lp{2}(Z_{[0,T_d)})}.
\end{equation*}

Putting this into Eq.~\eqref{Eq:ASemiEst} and subtracting the term $\frac12  \norm{Av}^2_{\Lp{2}(Z_{[0,T_d)})}$ from both sides, 
\begin{align*}
\norm{\partial_t v}^2_{\Lp{2}(Z_{[0,T_d)})} &+ \norm{Av}^2 _{\Lp{2}(Z_{[0,T_d)})} + \norm{v}_{\Lp{2}(Z_{[0,T_d)})}^2   \\ 
&\leq 
2 \norm{\modulus{A}^{\frac12} v\rest{\dM}}_{\dom(\modulus{A}^{\frac12})}^2 + 3 \max\set{C_1, C_2} \cbrac{ \norm{Dv}^2_{\Lp{2}(Z_{[0,T_d)})} + \norm{v}^2_{\Lp{2}(Z_{[0,T_d)})}}
\end{align*}
In other words, since we assumed $\spt v \subset Z_{[0,T_d)}$,
\begin{equation}
\label{Eq:ASemiEst2} 
\norm{\partial_t v}_{\Lp{2}(Z_{[0,T_d)})} + \norm{Av}_{\Lp{2}(Z_{[0,T_d)})} + \norm{v}_{\Lp{2}(Z_{[0,T_d)})}
\lesssim \norm{v}_{\dom(D_{\max})} + \norm{ \modulus{A}^{\frac12} v\rest{\dM}}_{\Lp{2}(\dM)} 
\end{equation}  

Now, fix $u \in B$  with $\spt u_n \subset Z_{[0,T_d)}$ and let $u_n \in \Ck[c]{\infty}(M)$ such that $u_n \to u$ in $\dom(D_B)$ and $u_n\rest{\dM} \to u\rest{\dM}$ in $\dom(\modulus{A}^{\frac12})$.
Without loss of generality, we can assume that $\spt u_n \subset Z_{[0,T_d)}$.
As a consequence of Eq.~\eqref{Eq:ASemiEst2},
\begin{align*} 
\norm{\partial_t u}^2_{\Lp{2}(Z_{[0,T_d)})} + \norm{Au}^2 _{\Lp{2}(Z_{[0,T_d)})} + \norm{u}_{\Lp{2}(Z_{[0,T_d)})}^2   
&\lesssim 
\norm{\modulus{A}^{\frac12} u\rest{\dM}}_{\dom(\modulus{A}^{\frac12})}^2 + \norm{u}^2_{\dom(D_{\max})} \\ 
&\simeq 
\norm{u}_{\checkH(A)}^2 + \norm{u}^2_{\dom(D_{\max})} 
\lesssim 
\norm{u}_{\dom(D_{\max})}^2,
\end{align*}
where we used the isomorphism $\norm{u}_{\dom(\modulus{A}^{\frac12})} = \norm{u}_{B} \simeq \norm{u}_{\checkH(A)}$ since $B \subset \dom(\modulus{A}^{\frac12}) \subset \checkH(A)$ is closed in both spaces.
\end{proof}

\subsection{Regularity revisited}

In the situation when $\dM$ is compact, given two spectral cuts $r > q$, we saw that $\chi^-(A_r) \Lp{2}(\dM;E) \cap \chi^+(A_q) \Lp{2}(\dM;E)$ is finite dimensional and by elliptic regularity, a subspace of smooth sections.
For selfadjoint $A$, this space is the same as $\chi_{[q,r)}(A) \Lp{2}(\dM;E)$.
In the present context, with $\dM$ being noncompact in general, it is unlikely that $\chi_{[q,r)}(A) \Lp{2}(\dM;E)$ is finite dimensional.
However, it is worth considering whether this is a smooth subspace of sections, despite potential infinite dimensionality.

\begin{proposition}
\label{Prop:IntSmooth}
For any $\alpha\ge0$ and any Borel set $S \subset \R$ that is bounded, the subspace $\chi_{S}(A) \dom(\modulus{A}^{\alpha})^\ast  \subset \Ck{\infty}(\dM;E)$.
\end{proposition}
\begin{proof}
From Proposition~\ref{Prop:IntSmoothAbstract} we obtain
$$
\chi_{S}(A) \dom(\modulus{A}^{\alpha})^\ast  \subset \bigcap_{k=0}^\infty\dom(A^{2k}).
$$
By elliptic regularity,
\[
\bigcap_{k=0}^\infty\dom(A^{2k}) \subset \bigcap_{k=0}^\infty\SobH[loc]{2k}(\dM;E)=\Ck{\infty}(\dM;E).
\qedhere
\]
\end{proof}

We now consider regularity questions via adapted boundary operators $A$.

\begin{theorem}
\label{Thm:Reg2}
The spaces $\chi^{\pm}(A) \dom(\modulus{A}^{\frac12}) \subset \SobH[loc]{\frac12}(\dM;E)$ and
\begin{align*}
\dom&(D_{\max})\cap
\SobH[loc]{k}(M;E)                                                                                                               \\
&= 
\set{u \in \dom(D_{\max}): Du \in \SobH[loc]{k-1}(M;E)\ \text{with}\ \chi^+(A)(u\rest{\dM}) \in \SobH[loc]{k-\frac12}(\dM;E) } .
\end{align*}
\end{theorem}
\begin{proof}
By elliptic regularity we have
$$
\chi^{\pm}(A) \dom(\modulus{A}^{\frac12}) \subset \dom(\modulus{A}^{\frac12})\subset \SobH[loc]{\frac12}(\dM;E).
$$

To prove the remaining equality, note that from Theorem~\ref{Thm:Reg}, we need to consider $u\rest{\dM}$.
But we established that $\chi^-(A)(u \rest{\dM}) \in \SobH[loc]{\frac12}(\dM;E)$ and so
\begin{equation*}
u\rest\dM \in \SobH[loc]{\frac12}(\dM;E)\ \iff\ \chi^+(A)(u \rest{\dM}) \in  \SobH[loc]{\frac12}(\dM;E).
\qedhere
\end{equation*}
\end{proof}

\subsection{Boundary conditions via projectors}

In the context of compact boundary, given a classical pseudo-differential projector $P$ of order zero, the resulting pseudo-local boundary condition was written as  $B := \close{P \SobH{\frac12}(\dM;E)}^{\checkH(A)}$.
The fact that renders this notion to be well-defined is that pseudo-differential operators of order zero act boundedly on the Sobolev scale.
Unfortunately, in our present situation, we no longer have this luxury. 
Nevertheless, we will define the following general class of boundary conditions which capture pseudo-local boundary conditions of the compact setting.

\begin{definition}[Projection boundary condition]
\label{Def:ProjBC}
Let $P:  \dom(\modulus{A}^{\frac12})^\ast\to \dom(\modulus{A}^{\frac12})^\ast$ be a bounded projector, which restricts to a bounded projection on $\dom(\modulus{A}^{\frac12})$.
Then, we call  $B := \close{P \dom(\modulus{A}^{\frac12})}^{\checkH(A)}$  a \emph{projection boundary condition}.
\end{definition}

\begin{remark}
\label{Rem:ProjBC}
By definition, $I-P$, $P^\ast$ and $I-P^\ast$ all define projection boundary conditions.
That is, they satisfy the condition listed in Definition~\ref{Def:ProjBC}.
\end{remark}

\begin{lemma}
\label{Lem:ProjBCDensity}
For a projection boundary condition $B$ obtained via a projector $P$, $P\dom(\modulus{A}^{\frac12})$ is dense in $P\dom(\modulus{A}^{\frac12})^\ast$.
\end{lemma}
\begin{proof}
Let $u \in P \dom(\modulus{A}^{\frac12})^\ast$, i.e., $u = Pv$ for $v \in \dom(\modulus{A}^{\frac12})^\ast$.
Let $v_n \in \dom(\modulus{A}^{\frac12})$ such that $v_n \to v$ in $\dom(\modulus{A}^{\frac12})^\ast$.
Then, defining $u_n = Pv_n$,
$$ \norm{u_n - u}_{\dom(\modulus{A}^{\frac12})^\ast} = \norm{Pv_n - Pv}_{\dom(\modulus{A}^{\frac12})^\ast} \lesssim \norm{v_n -v}_{\dom(\modulus{A}^{\frac12})^\ast} \to 0.$$
By the hypothesis on $P$ in Definition~\ref{Def:ProjBC}, we have that $u_n \in P\dom(\modulus{A}^{\frac12})$.
\end{proof}

The following demonstrates an alternative characterisation of this class of boundary conditions, as well as the related adjoint boundary condition.

\begin{proposition}
\label{Prop:LocAdj}
If $B$ is a projection boundary condition, then:
\begin{enumerate}[label=(\roman*), labelwidth=0pt, labelindent=2pt, leftmargin=21pt]
\item \label{Prop:LocAdj:Adj}
      $\sym_0^\ast B^\ast =  (I - P^\ast) \dom(\modulus{A}^{\frac12})^{\ast}  \cap \hatH(A)$, and
\item \label{Prop:LocAdj:Int}
      $B = P\dom(\modulus{A}^{\frac12})^\ast \cap {\checkH(A)}$.
\end{enumerate}
\end{proposition}
\begin{proof}
From Proposition~\ref{Prop:AdjBC}, we have that $\sym_0^\ast B^\ast = B^{\perp, \hatH(A)}$.
We compute this explicitly in this setting:
\begin{align}
B^{\perp, \hatH(A)}
 & = \set{v \in \hatH(A): \inprod{u, v}_{\checkH(A) \times \hatH(A)} = 0 \quad \forall u \in B} \nonumber                                                 \\
 & = \set{v \in \hatH(A): \inprod{u, v}_{\checkH(A) \times \hatH(A)} = 0 \quad \forall u \in \close{P\dom(\modulus{A}^{\frac12})}^{\checkH(A)}} \nonumber \\
 & = \set{v \in \hatH(A): \inprod{u, v}_{\checkH(A) \times \hatH(A)} = 0 \quad \forall u \in P\dom(\modulus{A}^{\frac12})} \nonumber                      \\
 & = \set{v \in \hatH(A): \inprod{Pu', v}_{\checkH(A) \times \hatH(A)} = 0 \quad \forall u \in \dom(\modulus{A}^{\frac12})} \label{Eq:Adj}.
\end{align}
Now, note that for $w \in \checkH(A)$, and all $v \in \hatH(A)$,
\begin{equation}
\label{Eq:ProjSplit}
\begin{aligned}
\inprod{w,v}_{\checkH(A) \times \hatH(A)}
 & = \inprod{ \chi^-(A)w, \chi^-(A)v}_{\dom(\modulus{A}^{\frac12}) \times \dom(\modulus{A}^{\frac12})^\ast}                \\
 & \qquad\qquad+  \inprod{ \chi^+(A)w, \chi^+(A)v}_{\dom(\modulus{A}^{\frac12})^\ast  \times \dom(\modulus{A}^{\frac12})}.
\end{aligned}
\end{equation}
If further $w  \in \dom(\modulus{A}^{\frac12})$,
\begin{align*}
\inprod{ \chi^+(A)w, \chi^+(A)v}_{\dom(\modulus{A}^{\frac12})^\ast  \times \dom(\modulus{A}^{\frac12})}
 & = \inprod{ \chi^+(A)w, \chi^+(A)v}_{\Lp{2}(\dM) \times \Lp{2}(\dM)}                                   \\
 & = \inprod{ \chi^+(A)w, \chi^+(A)v}_{\dom(\modulus{A}^{\frac12})\times \dom(\modulus{A}^{\frac12})^\ast }.
\end{align*}
Therefore, substituting this back into Eq.~\eqref{Eq:ProjSplit}, we have that
$$
\inprod{w,v}_{\checkH(A) \times \hatH(A)} = \inprod{w, v}_{\dom(\modulus{A}^{\frac12})\times \dom(\modulus{A}^{\frac12})^\ast }.
$$
Putting this into Eq.~\eqref{Eq:Adj}, we have that
\begin{align*}
B^{\perp, \hatH(A)}
 & = \set{v \in \hatH(A): \inprod{Pu', v}_{\dom(\modulus{A}^{\frac12})\times \dom(\modulus{A}^{\frac12})^\ast } = 0 \quad \forall u \in \dom(\modulus{A}^{\frac12})} \\
 & = (I - P^\ast)\dom(\modulus{A}^{\frac12})^\ast \cap \hatH(A).
\end{align*}
This proves \ref{Prop:LocAdj:Adj}.

To prove \ref{Prop:LocAdj:Int}, we note that
\begin{align*}
B & = ((I - P^\ast)\dom(\modulus{A}^{\frac12})^\ast \cap \hatH(A))^{\perp, \checkH(A)}                                                                                                                                               \\
  & = \set{u \in \checkH(A): \inprod{u, v}_{\dom(\modulus{A}^{\frac12})\times \dom(\modulus{A}^{\frac12})^\ast } = 0 \quad \forall v \in (I - P^\ast)\dom(\modulus{A}^{\frac12})^\ast \text{ and } v \in \hatH(A)}                  \\
  & =  \set{u \in \checkH(A): \inprod{u, (I - P^\ast) v'}_{\dom(\modulus{A}^{\frac12})\times \dom(\modulus{A}^{\frac12})^\ast } = 0 \quad \forall v' \in \dom(\modulus{A}^{\frac12})^\ast \text{ and } (I - P^\ast)v' \in \hatH(A)}
\end{align*}
Note, however, that as noted in Remark~\ref{Rem:ProjBC}, $(I - P^\ast)$ also defines a projection boundary condition.
Therefore, on application of Lemma~\ref{Lem:ProjBCDensity}, we have that  $(I - P^\ast) \dom(\modulus{A}^{\frac12})$ is dense in $(I - P^\ast) \dom(\modulus{A}^{\frac12})^\ast$.
Moreover, for $v' \in \dom(\modulus{A}^{\frac12})$, we have automatically that $(I - P^\ast)v' \in \hatH(A)$, and hence,
\begin{align*}
B & = \set{u \in \checkH(A): \inprod{u, (I - P^\ast) v'}_{\dom(\modulus{A}^{\frac12})\times \dom(\modulus{A}^{\frac12})^\ast } = 0 \quad \forall v' \in \dom(\modulus{A}^{\frac12})^\ast \text{ and } (I - P^\ast)v' \in \hatH(A)} \\
  & =\set{u \in \checkH(A): \inprod{u, (I - P^\ast)w}_{\dom(\modulus{A}^{\frac12})\times \dom(\modulus{A}^{\frac12})^\ast } = 0 \quad \forall w \in \dom(\modulus{A}^{\frac12})}                                                   \\
  & =\set{u \in \checkH(A): \inprod{(I - P)u, w}_{\dom(\modulus{A}^{\frac12})^\ast \times \dom(\modulus{A}^{\frac12})} = 0 \quad \forall w \in \dom(\modulus{A}^{\frac12})}                                                        \\
  & = P \dom(\modulus{A}^{\frac12})^\ast \cap \checkH(A).
\qedhere
\end{align*}
\end{proof}

\subsection{Atiyah-Patodi-Singer boundary condition}
\label{sec:APS}

Now we consider perhaps the most fundamental boundary condition, originally formulated in the compact setting by Atiyah-Patodi-Singer.

\begin{definition}[Atiyah-Patodi-Singer boundary condition]
Given an adapted boundary operator $A$, we define the Atiyah-Patodi-Singer boundary condition with respect to this operator to be
$$ B_{\APS}(A) := \chi_{(-\infty,0)}(A)\dom(\modulus{A}^\frac12).$$
We define the associated operator as $D_{\APS(A)} := D_{B_{\APS}(A)}$.
\end{definition}

Clearly this is a projection boundary condition as we have defined the in the previous section.
In the context of noncompact boundary, we cannot expect the operator $D_{\APS(A)}$ to be Fredholm.
Nevertheless, we obtain the following which demonstrates that it has the expected regularity.

\begin{proposition}\label{Prop:APS=Aelliptic}
The APS boundary condition for an adapted operator $A$ is $A$-elliptically regular.
In particular, it is elliptically regular.
\end{proposition}
\begin{proof}
Since, by definition, $B_\APS(A)\subset \dom(\modulus{A}^{\frac12})$, by Definition~\ref{Def:AEllReg} it suffices to show that $\sym_0^\ast B_{\APS}(A)^\ast \subset \SobH[loc]{\frac12}(\dM;E)$.
We recall that $\hatH(A) = \chi_{(-\infty,0)}(A) \dom(\modulus{A}^{\frac12})^\ast \oplus \chi_{[0,\infty)}(A) \dom(\modulus{A}^{\frac12})$ and compute:
\begin{align*}
B_{\APS}(A)^{\perp,\hatH(A)}
 & =
\set{u \in \hatH(A): \inprod{u,v}_{\hatH(A) \times \checkH(A)} = 0\ \quad\forall v \in B_{\APS}(A)}           \\
 & =
\set{u \in \hatH(A): \inprod{u, \chi^-(A)v}_{\hatH(A) \times \checkH(A)} = 0 \quad \forall v \in B_{\APS}(A)} \\
 & =
[\chi_{(-\infty,0)}(A) \dom(\modulus{A}^{\frac12})]^{\perp, \hatH(A)}                                         \\
 & =
\chi_{[0,\infty)}(A) \dom(\modulus{A}^{\frac12}).
\end{align*}
Since $B_{\APS}(A)^{\perp,\hatH(A)} = \sym_0^\ast B_{\APS}(A)^\ast$ by Proposition~\ref{Prop:AdjBC}, $B$ is $A$-elliptically regular.
By Corollary~\ref{Cor:AEllReg}, it is in particular elliptically regular.
\end{proof}

\begin{remark}\label{rem:BAPS-adjoint}
In case $\sym_0$ anticommutes with $A$, $\sym_0 \chi_{[0,\infty)}(A) =  \chi_{(-\infty,0]}(A)$, and the proof shows that the adjoint boundary condition of $B_{\APS}(A)$ is given by
\begin{align*}
B_{\APS}(A)^* 
&= 
\sym_0 B_{\APS}(A)^{\perp,\hatH(A)} 
= 
\sym_0 \chi_{[0,\infty)}(A) \dom(\modulus{A}^{\frac12}) \\
&=
\chi_{(-\infty,0]}(A) \dom(\modulus{A}^{\frac12}) 
=
B_{\APS}(A) \oplus \ker(A).
\end{align*}
\end{remark}

\subsection{Local boundary conditions}

Local boundary conditions are an important class of boundary conditions. 
In the compact case, they arise as a special case of projection boundary conditions. 
The noncompact situation deviates from this, as the pointwise projectors governing such a boundary condition may not have the required decay to be a projection boundary condition.
Nevertheless, we capture the notion of a local boundary condition here and study when they coincide with projection boundary conditions. 
In particular, we are motivated to study chiral boundary conditions which we present towards the end of this subsection.

\begin{definition}
Let $E' \subset E\rest{\dM}$ be a smooth subbundle of $E$.
Then
\[
B_{E'} := \close{\Ck[c]{\infty}(\dM;E')}^{\checkH(A)}
\]
is called a \emph{local boundary condition}.
\end{definition}

Local boundary conditions, as expected, are preserved under multiplication by compactly supported smooth functions.
To show this, we first establish the following lemma.

\begin{lemma}
\label{Lem:CzechBound}
If $\xi \in \Ck[c]{\infty}(\dM)$, then $\xi \id: \checkH(A) \to \checkH(A)$ boundedly.
\end{lemma}
\begin{proof}
Let $u \in \Dk{\infty}(A) = \cap_{\alpha} \dom(\modulus{A}^{\alpha})$.
Then, by Proposition~\ref{Prop:WP}~\ref{Prop:WP:3}, we have that $v := \ext_{T_c}u \in \dom(D_{\max})$ and clearly, $v\rest{\dM} = u$.
Let $\tilde{\xi} \in \Ck[c]{\infty}(M)$ such that $\tilde{\xi}\rest{\dM} = \xi$.
Since $\modulus{[D, \tilde{\xi}]}\lesssim  1$, $\tilde{\xi}v \in \dom(D_{\max})$.
Moreover, $(\tilde{\xi}v)\rest{\dM} = \xi u$.
Putting this together,
$$ \norm{\xi u}_{\checkH(A)} \lesssim \norm{\tilde{\xi}v}_{\dom(D_{\max})} \lesssim \norm{v}_{\dom(D_{\max})} \lesssim \norm{u}_{\checkH(A)},$$
where the first inequality follows from Theorem~\ref{Thm:MaxDom}~\ref{Thm:MaxDom:1} and the ultimate inequality  from Proposition~\ref{Prop:WP}~\ref{Prop:WP:3}.
\end{proof}

\begin{remark}
Note that it is unclear how to perform the estimate in Lemma~\ref{Lem:CzechBound} directly,  the boundedness of the  commutator of  $\chi^{\pm}(A)$ and $\xi$ is not known.
\end{remark}

Note that local boundary conditions ``localise'' in the sense that multiplication by compactly supported functions preserve the boundary condition.
\begin{proposition}
\label{Prop:LocLoc}
Let $B$ be a local boundary condition.
\begin{enumerate}[label=(\roman*), labelwidth=0pt, labelindent=2pt, leftmargin=21pt]
\item \label{Prop:LocLoc:1}
      If $\xi \in \Ck[c]{\infty}(\dM;E)$, then $\xi B \subset B$.
\item \label{Prop:LocLoc:2}
      If $\chi \in \Ck[c]{\infty}(M;E)$, then $\chi \dom(D_{B}) \subset \dom(D_B)$.
\end{enumerate}
\end{proposition}
\begin{proof}
Let $u \in B$.
By construction, there exists $u_n \in \Ck[c]{\infty}(\dM;E')$ such that $u_n \to u$ in $\checkH(A)$ and $u_n(x) \in E'_x$ for all $x \in \dM$.
Now,
$$ \norm{\xi u_n - \xi u}_{\checkH(A)}  = \norm{ \xi (u_n - u)}_{\checkH(A)} \lesssim \norm{u_n - u}_{\checkH(A)},$$
where the inequality follows from Lemma~\ref{Lem:CzechBound}.
Clearly, $\xi u_n \in \Ck[c]{\infty}(\dM;E')$ and $(\xi u_n)(x) \in E'_x$ for all $x \in \dM$.
Therefore, $\xi u \in B$.
This proves \ref{Prop:LocLoc:1}.

To prove \ref{Prop:LocLoc:2}, let $u \in \dom(D_B)$.
That is, $u \in \dom(D_{\max})$ and $u \rest{\dM} \in B$.
Clearly, $\chi u \in \dom(D_{\max})$ and from what we have proved, setting $\xi = \chi\rest{\dM}$, we have that $(\chi u)\rest{\dM} = \xi u\rest{\dM} \in B$.
\end{proof}

As in the compact boundary case, we obtain the following which provides elliptically regular boundary conditions.

\begin{theorem}
\label{Thm:LocEllReg}
Suppose that $E\rest{\dM} = E_- \oplus E_+$ is a fibrewise orthogonal splitting.
Moreover, assume that $\sigma_A(x,\xi)$ interchanges $(E_-)_x$ and $(E_+)_x$ for each $0 \neq \xi \in T^\ast_x \dM$.
Then, the local boundary conditions $B_{E_-}$ and $B_{E_+}$ are elliptically regular local boundary conditions.
\end{theorem}
\begin{proof}
Let $B := B_{E_-}$.
We first show that $B \subset \SobH[loc]{\frac12}(\dM;E)$.
For that, fix $x\in M$ and let $V_x$ and $U_x$ be the neighbourhoods guaranteed by Lemma~\ref{Lem:ExtOp} as well as $\delta > 0$.
Furthermore, let  $N_x$ be the manifold given by Corollary~\ref{Cor:EmbMf}.
Let $\chi \in \Ck[c]{\infty}(M)$ such that $\spt \chi \subset [0, \delta) \times U_x$ with $\chi = 1$ on $V_x$.

By the fact that $B$ is a local boundary condition, using Proposition~\ref{Prop:LocLoc}, we obtain that $\chi u \in \dom(D_B)$ when $u \in \dom(D_B)$.
Note that if $A^N$ is an adapted operator for $\tilde{D}$ from Corollary~\ref{Cor:EmbMf}, then $\sym_0(x, A^N) = \sym_0(x,A)$ for $x \in \partial N$.
Letting $P^N$ be the projector along this splitting in $\partial N$, we have that on the  induced bundle $\tilde{E}\rest{\dM}$ on $N_x$ that $\sym_0(x,A^N)$ interchanges $\tilde{E}'$ and $\tilde{E}''$.
Therefore, using Theorem~2.15 in \cite{BBan} and reasoning as in Corollary~7.23 in \cite{BB12}, we obtain that $\chi u \in \SobH{\frac12}(N_x;\tilde{E})$.
That is, $\chi u\rest{\partial N} \in \SobH{\frac12}(\partial N;\tilde{E})$.
Since $x$ was arbitrary, this shows that $u \in \SobH[loc]{\frac12}(\dM;E)$.

Proposition~\ref{Prop:LocAdj} yields that  $B^\ast$ is also a local boundary condition and applying this argument to $B^\ast$ in  place of $B$ yields that $B^\ast \subset \SobH[loc]{\frac12}(\partial N;F)$.

The argument for $\close{(I -P)\dom(\modulus{A}^{\frac12})}^{\checkH(A)}$  proceeds exactly on replacing $P$ by $(I-P)$.
\end{proof}

In the compact case, local boundary conditions can be seen as a particular case of pseudo-local boundary conditions.
This is due to the fact that a projection defining local boundary conditions are always pseudo-differential operators of order zero and hence bounded on Sobolev scales.
In this present context, that may no longer be the case.
However, the following proposition says that when the bundle projection defines a projection boundary condition, then it agrees with the local boundary condition.

\begin{proposition}
\label{Prop:LocalProjBC}
Let $E\rest{\dM} = E_- \oplus E_+$ be a bundle splitting and let $P_{\pm}$ be the associated bundle projection to $E_{\pm}$ along $E_{\mp}$.
Suppose that $P_\pm $ defines a projection boundary condition.
Then, the projection boundary condition coincides with the local boundary condition with respect to $E_\pm $.
\end{proposition}
\begin{proof}
Let
$$
\tilde{B}_{E_\pm } = \close{\set{u \in \dom(\modulus{A}^{\frac12}: u(x) \in (E_\pm )_x\ \text{almost-everywhere}}}^{\checkH(A)}.$$
Since $P_\pm \dom(\modulus{A}^{\frac12}) \subset \dom(\modulus{A}^{\frac12})$ as it defines a projection boundary condition, it is readily verified that
\[
\set{u \in \dom(\modulus{A}^{\frac12}): u(x) \in (E_\pm )_x} = P_\pm  \dom(\modulus{A}^{\frac12}).
\]
Therefore $\tilde{B}_{E_\pm } = \close{P_\pm \dom(\modulus{A}^{\frac12})}^{\checkH(A)}$.

It remains to prove that $B_{E_\pm } = \tilde{B}_{E_\pm }$.
Since $\Ck[c]{\infty}(\dM;E_\pm ) \subset \tilde{B}_{E_\pm }$, it is immediate  that $B_{E_\pm } \subset \tilde{B}_{E_\pm }$.

We prove the reverse containment.
Given $u \in \dom(\modulus{A}^{\frac12})$ with $u(x) \in (E_\pm )_x$, by the density of $\Ck[c]{\infty}(\dM;E)$ in $\dom(\modulus{A}^{\frac12})$, we have $u_n \to u$ in $\dom(\modulus{A}^{\frac12})$.
Since $P_\pm $ is a smooth bundle projection, $P_\pm u_n \in \Ck[c]{\infty}(\dM;E_\pm )$, but also $P_\pm u_n \in \dom(\modulus{A}^{\frac12})$.
Therefore,
$$\norm{P_\pm u_n - u}_{\dom(\modulus{A}^{\frac12})}
= \norm{P_\pm u_n - P_\pm u}_{\dom(\modulus{A}^{\frac12})}
= \norm{P_\pm (u_n - u)}_{\dom(\modulus{A}^{\frac12})}
\lesssim \norm{u_n - u} \to 0$$
as $n \to \infty$.
Therefore, $\Ck[c]{\infty}(\dM;E_\pm ) \subset \set{u \in \dom(\modulus{A}^{\frac12}): u(x) \in {E_\pm }_x}$ is dense in $\dom(\modulus{A}^{\frac12})$ and $\dom(\modulus{A}^{\frac12})$ embeds densely into $\checkH(A)$.
This yields that $\tilde{B}_{E_\pm } \subset B_{E_\pm }$.
\end{proof}

A particular class of local boundary conditions which are of interest are \emph{chiral} boundary conditions.

\begin{definition}
A \emph{chirality operator} for $A$ is a bundle homomorphism $\Xi \in \Ck{\infty}(\End(E\rest{\dM}))$ which satisfies $\Xi^2 = I$ and anticommutes with $A$, i.e.\ $\Xi A = - A \Xi$.
\end{definition}

Since $\Xi$ anticommutes with $A$, it also anticommutes with the principal symbol $\sigma_A(x,\xi)$ for every $\xi\in T_x^*\dM$ and $x\in \dM$.
Therefore, the eigenspaces $E_{\pm,x}$ of $\Xi_x$ for the eigenvalues $\pm 1$ have the same dimension for all $x\in\dM$ and we obtain a bundle splitting $E|_{\dM}=E_+\oplus E_-$.
If $\Xi$ is selfadjoint, this splitting is orthogonal.

\begin{proposition}[Chiral boundary conditions]
\label{Prop:ChiralReg}
Let $\Xi$ be a  chirality operator for $A$ and let $E\rest{\dM} = E_+ \oplus E_-$ be the corresponding splitting into eigenbundles of $\Xi$.
Then, the local boundary conditions $B_{E_{\pm}}$ are $A$-elliptically regular.
\end{proposition}
\begin{proof}
Since $\Xi:\Lp{2}(\dM;E) \to \Lp{2}(\dM;E)$ and $\Xi A = - A \Xi$, by Corollary~\ref{Cor:IdemProjs}, the fibrewise projectors $P_{\pm}$ and $P_{\pm}^\ast$ define projection boundary conditions.
Using Proposition~\ref{Prop:LocalProjBC} and Proposition~\ref{Prop:LocAdj}, we obtain
\begin{equation}
\label{Eq:LocProjDesc}
B_{E_{\pm}} 
= 
P_{\pm} \dom(\modulus{A}^{\frac12})^\ast \cap \checkH(A)\ 
\text{ and }\ 
B_{E_{\pm}}^{\perp} 
= 
P_{\mp}^\ast  \dom(\modulus{A}^{\frac12})^\ast \cap \hatH(A),
\end{equation}
where we used $(I - P_{\pm}^\ast) = P_{\mp}^\ast$ in the second equality.

First, we prove that $B_{E_{\pm}} \subset \dom(\modulus{A}^\frac12)$.
We automatically have that $\chi_{(-\infty,0]}(A)u \in \dom(\modulus{A}^{\frac12}) \in \dom(\modulus{A}^\frac12)$ whenever $u \in B_{E_{\pm}}$.
For such $u \in B_{E_{\pm}}$, from Eq.~\eqref{Eq:LocProjDesc}, we have that $u = P_{\pm} u$.
Then,
\begin{align*}
\chi_{(-\infty,0]}(A)u
 & = \chi_{(-\infty,0]}(A)P_{\pm} u           \\
 & = \tfrac12  \chi_{(-\infty,0]}(A)(I + \Xi)u \\
 & = \tfrac12 \cbrac{ \chi_{(-\infty,0]}(A)u \pm \chi_{(-\infty,0]}(A) \Xi u}.
\end{align*}
Since $\Xi A = - A \Xi$, we have that
$$\chi_{(-\infty,0]}(A) \Xi = \Xi \chi_{(-\infty,0]}(-A) = \pm \Xi \chi_{[0,\infty)}(A).$$
Therefore,
$$
2\chi_{(-\infty,0]}(A)u  = \chi_{(-\infty,0]}(A)u \pm \Xi \chi_{[0,\infty)}(A)u \iff \Xi \chi_{[0,\infty)}(A)u = \pm \chi_{(-\infty,0]}(A)u.$$
This allows us to conclude that $\Xi \chi_{[0,\infty)}(A)u  \in \dom(\modulus{A}^{\frac12})$.
Now we invoke Proposition~\ref{Prop:IdemProjs} with $\alpha = \frac12$ to obtain that $\Xi: \dom(\modulus{A}^{\frac12}) \to \dom(\modulus{A}^{\frac12})$.
Therefore,
$$ \dom(\modulus{A}^{\frac12}) \ni \Xi ( \Xi \chi_{[0,\infty)}(A)u) = \chi_{[0,\infty)}(A)u,$$
and hence $B_{E_{\pm}} \subset \dom(\modulus{A}^{\frac12})$.

It remains to show that $B_{E_{\pm}}^{\perp, \hatH(A)} \subset \dom(\modulus{A}^{\frac12})$.
In this case, choosing an equivalent norm for $\hatH(A)$ appropriately, we obtain that $\chi_{[0,\infty)}(A)v \in \dom(\modulus{A}^{\frac12})$
for $v \in B_{E_{\pm}}^{\perp,\hatH(A)}$.
Using Eq.~\eqref{Eq:LocProjDesc}, we have that $v = P_{\mp}^\ast v$, and using the fact that $\Xi^\ast A = - \Xi^\ast A$, we get
$\chi_{[0,\infty)(A)}\Xi^\ast = \Xi^\ast \chi_{[0,\infty)(-A)} = \Xi^\ast \chi_{(-\infty,0]}(A)$.
Mirroring the argument we have just made, we obtain that  $\Xi^\ast \chi_{(-\infty,0]}(A)u \in \dom(\modulus{A}^{\frac12})$.
Again, by Proposition~\ref{Prop:IdemProjs}, we have that $\Xi^\ast: \dom(\modulus{A}^{\frac12}) \to \dom(\modulus{A}^{\frac12})$ and therefore,
$$ \dom(\modulus{A}^{\frac12}) \ni \Xi^\ast (\Xi^\ast \chi_{(-\infty,0]}(A)u) = \chi_{(-\infty,0]}(A)u.$$

Together, this proves that $B_{E_{\pm}}$ are $A$-elliptically regular boundary conditions.
\end{proof}

\begin{remark}
Note that in Proposition~\ref{Prop:ChiralReg}, we do not ask for $\Xi$ to be selfadjoint.
Since $A$ is selfadjoint, the adjoint involution $\Xi^*$ also anticommutes with $A$.
Hence, $\Xi^*$ is a chirality operator for $A$ as well, and Proposition~\ref{Prop:ChiralReg} applies accordingly.
\end{remark}

\begin{remark}\label{rem:ChiralAdjoint}
There are two distinct assumptions on  $\sym_0$ which allow us to identify the adjoint boundary condition.
\begin{enumerate}[label=(\arabic*), labelwidth=0pt, labelindent=2pt, leftmargin=26pt]
\item \label{rem:ChiralAdjoint.1}
If $\sym_0$ commutes with $\Xi$, then the adjoint boundary condition for $B_{E_\pm}$ are given by
\begin{align*}
B_{E_\pm}^*
&=
\sym_0(P_{\mp}^\ast  \dom(\modulus{A}^{\frac12})^\ast \cap \hatH(A))
=
P_{\mp}^\ast  \sym_0(\dom(\modulus{A}^{\frac12})^\ast) \cap \checkH(A)\\
&=
P_{\mp}^\ast  \dom(\modulus{A}^{\frac12})^\ast \cap \checkH(A)
=
B_{\tilde{E}_\mp},
\end{align*}
where $\tilde{E}_\pm$ are the $\pm 1$-eigenbundles of $\Xi^*$.
\item\label{rem:ChiralAdjoint.2}
If instead we assume that $\sym_0$ anticommutes with $A$, then $\tilde{\Xi} := \sym_0\Xi^*\sym_0^{-1}$ is a chirality operator.
We find
\begin{align*}
B_{E_\pm}^*
&=
\sym_0(P_{\mp}^\ast  \dom(\modulus{A}^{\frac12})^\ast \cap \hatH(A))
=
\sym_0P_{\mp}^\ast\sym_0^{-1}  \sym_0\dom(\modulus{A}^{\frac12})^\ast \cap \checkH(A)\\
&=
\sym_0P_{\mp}^\ast\sym_0^{-1}  \dom(\modulus{A}^{\frac12})^\ast \cap \checkH(A)
=
B_{\tilde{E}_\mp}, 
\end{align*}
where now $\tilde{E}_\pm$ are the $\pm 1$-eigenbundles of $\tilde{\Xi}$.
\end{enumerate}

\end{remark}

\subsection{Matching conditions}

In the case of compact boundary, an important boundary condition for the purpose of topological editing and relative index theory is the notion of ``matching boundary conditions''.
We phrase this in our noncompact situation also.

The underlying geometric setup to phrase this condition is under the assumption that $\dM$ splits into two components that can be identified with each other, but opposite orientations.
In this case, if $A_0$ is an adapted  boundary operator  on the component $(\dM)_1$, then $-A_0$ is an adapted boundary operator on $(\dM)_2$ and $A = A_0 \oplus (-A_0)$ is an adapted boundary operator on $\dM$.
Note in this situation, $u = (u_1, u_2) \in \checkH(A) = \checkH(A_0) \times \hatH(A_0)$ and $v = (v_1, v_2) \in \hatH(A) = \hatH(A_0) \times \checkH(A_0)$. 
Therefore,
\begin{equation}
\label{Eq:H2IP}
\inprod{u,v}_{\checkH(A) \times \hatH(A)} = \inprod{u_1, v_1}_{\checkH(A_0) \times \hatH(A_0)} + \overline{\inprod{v_2, u_2}}^{\mathrm{conj}}_{\hatH(A_0) \times \checkH(A_0)}.
\end{equation}
This, along with the matching boundary condition and its regularity, is made precise in the following.

\begin{proposition}
\label{Prop:MatchingBC}
Suppose that $\dM = (\dM)_1 \sqcup (\dM)_2$ where $(\dM)_1 = (\dM)_2 =: N$ with $A = A_0 \oplus (-A_0)$, where $A_0$ is a selfadjoint adapted boundary operator on $(\dM)_1$.
Let
$$B_{\Match} := \set{(u, u): u \in \dom(\modulus{A_0}^{\frac12})}.$$
Then $B_{\Match}$ is an $A$-elliptically regular boundary condition.
\end{proposition}
\begin{proof}
We compute $B_{\Match}^{\perp, \hatH(A)}$.
First, we note that $B_{\Match}^{\perp, \hatH(A)} = B_{\Match}^{\perp, \dom(\modulus{A}^{\frac12})^\ast} \cap \hatH(A)$.
This follows from a computation mirroring the proof of Proposition~\ref{Prop:LocAdj}~\ref{Prop:LocAdj:Adj}, in particular using \eqref{Eq:ProjSplit}.
Therefore, let $v = (v_1, v_2) \in B_{\Match}^{\perp, \dom(\modulus{A}^{\frac12})^\ast}$.
Then, from \eqref{Eq:H2IP},
\begin{align*}
0 & = \inprod{(u,u),(v_1,v_2)}_{\checkH(A) \times \hatH(A)}                                                                                                                                                   \\
  & = \inprod{u, v_1}_{\dom(\modulus{A}^{\frac12}) \times \dom(\modulus{A}^{\frac12})^\ast} + \overline{\inprod{v_2,u}}^{\mathrm{conj}}_{\dom(\modulus{A}^{\frac12})^\ast \times \dom(\modulus{A}^{\frac12})} \\
  & = \inprod{u, v_1 + v_2}_{\dom(\modulus{A}^{\frac12}) \times \dom(\modulus{A}^{\frac12})^\ast}
\end{align*}
That is $v_1 = -v_2$.
But since $(v_1, v_2) \in \hatH(A) = \hatH(A_0) \times \checkH(A_0)$, we have that $v_1 = -v_2 \in \hatH(A_0) \cap \checkH(A_0) \subset \dom(\modulus{A}^{\frac12})$.
Therefore, $\sym_0^\ast B^{\ast} \subset \dom(\modulus{A}^{\frac12})$ and by a repetition of this argument, it is easy to see that $B = (B^{\perp, \hatH(A)})^{\perp, \checkH(A)}$.
\end{proof}

Note that the argument here to establishing the $A$-elliptic regularity differs from that of the compact boundary case, where the equivalence between elliptic regularity and an \emph{``elliptic'' graphical decomposition} of a boundary condition was utilised.
The latter approach has the advantage that this boundary condition can be continuously deformed to the boundary condition $B_{\APS}(A) = \chi^-(A)\dom(\modulus{A}^{\frac12}) = \dom(\modulus{A_0}^{\frac12})$.
It is unlikely that  such a graphical decomposition would exist in the noncompact case.
That being said, for Callias operators (see Subsection~\ref{S:Callias}), both graphical decompositions and deformations of boundary conditions are readily accessible as in the compact boundary case.

\section{Coercivity and Fredholmness}
\label{Sec:Freddy}

In the compact boundary case, the study of Fredholm boundary conditions were preceded by a notion of coercivity for the operator.
To some extent, this case can be recovered but this requires extra information on the adapted operator as we will see in Subsection~\ref{Sec:Freddy:Discrete}.
In the absence of such a condition, we need to instead consider a special class of boundary conditions, which are outlined in Subsection~\ref{Sec:Freddy:BCs}.

\subsection{Coercivity through special boundary conditions}
\label{Sec:Freddy:BCs} 

We formulate a notion of coercivity here, where since we are in the noncompact setting, we also need to take boundary condition into account.
In the remark following this definition, the relation to the compact case is outlined.

\begin{definition}[Coercive with respect to a compact set]
\label{Def:Coercive}
The operator $D$ is said to be \emph{$B$-coercive with respect to $K$}, a compact set, if:
\begin{enumerate}[label=(\roman*), labelwidth=0pt, labelindent=2pt, leftmargin=21pt]
\item \label{Def:Coercive:1}
      there exists $\eta_{K} \in \Ck[c]{\infty}(M)$ with $\eta_K = 1$ on $K$ with
      $$ \eta_K \dom(D_B) \subset \dom(D_B);\qquad $$
\item \label{Def:Coercive:2}
      there is a constant $c > 0$ such that
      $$ \norm{Du} \geq c \norm{u}$$
      for all $u \in \dom(D_B)$ with $\spt u \cap K = \emptyset$.
\end{enumerate}
\end{definition}

\begin{remark}
\begin{enumerate}[label=(\alph*), wide, labelwidth=0pt, labelindent=0pt]
\item
      If $D$ is $B$-coercive with respect to $K$, then a similar conclusion holds for all $u \in \dom(D_{\min})$ since $\dom(D_{\min}) \subset \dom(D_B)$.
\item
      Suppose that $K$ is contained in the interior of $M$, $K \subset \interior{M}$, and $D$ satisfies \ref{Def:Coercive:2}.
      Then $\eta_K$ can be chosen to also have support in $\interior{M}$, hence \ref{Def:Coercive:1} holds automatically.
\item
      If $B$ is a local boundary condition, then \ref{Def:Coercive:1} is automatic for any compact $K$ including those that intersect the boundary.
\item
      If $\dM$ is compact, then $\dM\subset K$ can be assumed without loss of generality.
      Hence, \ref{Def:Coercive:1} is again automatic.
      More generally, if $\dM$ contains a finite number of compact connected components, they can be assumed to be contained in $K$.
\end{enumerate}
\end{remark}

\begin{proof}[Proof of Theorem~\ref{Thm:Coercive}]
We prove that if $\set{u_n} \subset \dom(D_B)$ is a bounded sequence (with respect to the graph norm) such that $D_B u_n \to v \in \Lp{2}(M;E)$, then $\set{u_n}$ has a convergent subsequence.
By Proposition~A.3 in \cite{BB12}, this implies that $\ker(D_B)$ is finite dimensional and $\ran(D_B)$ is closed.

Let $\{u_n\}$ be such a sequence.
Since $\norm{u_n - u_m}_D \simeq \norm{u_n - u_m} + \norm{D(u_n - u_m)}$ with the second summand tending to $0$, it suffices to find a subsequence $u_{n_k}$ that converges in $\Lp{2}(M;E)$.
Letting $\tilde{K}=\spt(\eta_K)$, we find
\begin{align}
\norm{u_n - u_m}
 & \leq \norm{\eta_{K} (u_n - u_m)}  + \norm{(1 - \eta_{K}) (u_n - u_m)} \notag                                                    \\
 & \leq \norm{\eta_{K} (u_n - u_m)} + c^{-1} \norm{D((1 - \eta_{K})(u_n - u_m))} \notag                                            \\
 & \leq \norm{\eta_{K} (u_n - u_m)} + c^{-1} \norm{\sym_{D}(d\eta_{K})(u_n - u_m)} + c^{-1} \norm{(1-\eta_{K}) D(u_n - u_m)}\notag \\
 & \lesssim \norm{u_n - u_m}_{\Lp{2}(\tilde{K})} + \norm{D(u_n - u_m)}.
\label{eq:aboveestimate}
\end{align}
In the second inequality, we used that $D$ is $B$-coercive with respect to $K$.

Since $B$ is semi-regular, we have that $\dom(D_B) \subset \SobH[loc]{1}(M;E)$.
In particular, using the boundedness of the sequence $\{u_n\}$ in $\dom(D_B)$, we have on the compact set $\tilde{K}$,
$$ \norm{u_n}_{\SobH{1}(\tilde{K})} \lesssim \norm{u_n}_D \lesssim 1.$$
From this, we find a convergent subsequence $u_{n_j} \to u'$ in $\Lp{2}(\tilde{K})$.
Then, by \eqref{eq:aboveestimate}, we have that
$$
\norm{u_{n_i}  - u_{n_j}} \lesssim \norm{u_{n_i} - u_{n_j}}_{\Lp{2}(\tilde{K})} + \norm{D(u_{n_i} - u_{m_j})} \to 0$$
as $i, j \to \infty$.
Therefore, we have a $u \in \Lp{2}(M;E)$ such that $u_{n_i} \to u$ and since $Du_{n_i} \to v$, by the fact that $D_B$ is closed, we have  $u \in \dom(D_B)$ with $v = D_B u$.
This is the required convergent subsequence in $\dom(D_B)$.
\end{proof}

\begin{remark}
The dependency on the extended setup to obtain these conclusions could be dropped.
Definition~\ref{Def:Coercive} could be equivalently formulated to simply hold for a closed extension $D_{c}$.
In Theorem~\ref{Thm:Coercive}, semi-regularity of $B$ is only used to assert $\dom(D_{B}) \subset \SobH[loc]{1}(M;E)$.
The latter condition could replace the former to obtain a slightly modified definition of coercivity in the absence of a boundary condition in order to drop the dependency on the extended setup.
\end{remark}

\begin{remark}
In \cite{GN} Definition~4.7, the authors define a notion of coercivity that is operator-theoretic.
Despite this abstract definition, it does not necessarily generalise the notion of coercivity which we have defined here.

In the present context, given a boundary condition $B$, the authors in \cite{GN} state that $\dom(D_B)$-coercive if for all $u \in \dom(D_B) \cap \ker(D_B)^\perp$ we have $\norm{u} \lesssim \norm{D_B u}$.
Assume that $\ker(D_B) = 0$.
Therefore, $\dom(D_B)$-coercive is to say that for all $u \in \dom(D_B)$, $\norm{u} \lesssim \norm{D_B u}$.
In particular, we have that $\dom(D_B)$-coercive implies $B$-coercive with respect to $\emptyset$ in our sense.
This is clearly much stronger than $B$-coercive with respect to some $K$ compact.
\end{remark}

\begin{remark}
Further, in \cite{GN} Definition~6.1, the authors define a geometric notion of coercivity.
There, they say that $D$ is coercive at infinity if there exists $K \subset M$ if $\norm{u} \lesssim \norm{Du}$ for all $u \in \Ck[c]{\infty}(M;E)$ with $\spt u \subset M \setminus K$.
Let us assume this.

Choose $\eta_K \in \Ck[c]{\infty}(M;[0,1])$ with $\eta_K = 1$ on $K$.
Let $\tilde{K} \subset K$ such that $\spt \eta_K \subset \tilde{K}$.
Let $u \in \dom(D_{\max})$ with $\spt u \subset M \setminus \tilde{K}$.
Then,  there exists $u_n \in \Ck[c]{\infty}(M;E)$ such that $u_n \to u$ in the graph norm of $D$.
Since $u = (1 - \eta_K)u$, we have that
$$ \norm{\eta_K u_n}_D = \norm{\eta_K (u_n - u)}_D \lesssim \norm{u_n - u}_D \to 0.$$
Therefore, $(1 - \eta_K)u_n \to u$ in the graph norm of $D$.
However, since $D$-is coercive at infinity  as we have assumed, we ave that then  $\norm{(1 - \eta_K)u_n} \lesssim \norm{D((1 - \eta_K) u_n)}$.
By letting $n \to \infty$, we obtain that $\norm{u} \lesssim \norm{Du}$ for all $u \in \dom(D_{\max})$ with $\spt u \subset M \setminus \tilde{K}$.

That is, $D$ coercive at infinity in the sense of \cite{GN} implies $D$ is $\checkH(D)$-coercive on $\tilde{K}$ in our sense.
These two notions can only be related for $B = \checkH(D)$ and shows that our notion is a refinement.
\end{remark}

We include the following immediate observation which is useful in solving PDEs with prescribed boundary conditions.

\begin{proposition}
\label{Prop:Solv}
Suppose that $B$ be a boundary condition for $D$ and $B^{\rc}$ any complementary subspace such that $\checkH(D) = B \oplus B^{\rc}$.
Let $\gamma u = u \rest{\dM}: \dom(D_{\max}) \to \checkH(D)$ and $P_{B^{\rc}, B}: \checkH(D) \to B^{\rc}$ the projection to $B^{\rc}$ along $B$.
Then,
\begin{enumerate}
\item
      $D \oplus P_{B^{\rc},B} \circ \gamma: \dom(D_{\max}) \to \Lp{2}(M;E) \oplus B^{\rc}$ has closed range and finite dimensional kernel if and only if $D_B$ has closed range and finite dimensional kernel.
\item
      $D \oplus P_{B^{\rc},B}\circ \gamma: \dom(D_{\max}) \to \Lp{2}(M;E) \oplus B^{\rc}$ is Fredholm of index $k$ if and only if $D_B$ is Fredholm and of index $k$.
\end{enumerate}
\end{proposition}
\begin{proof}
This follows immediately by invoking Proposition~A.1 in \cite{BB12} on choosing $H = \dom(D_{\max})$, $E$ to be $\Lp{2}(M;E)$, $F$ to be $B^{\rc}$ and $P = P_{B^{\rc},B} \circ \gamma$.
\end{proof}

\begin{corollary}
\label{Cor:Solv}
If  $B$ be a semi-regular boundary condition and $D$ be $B$-coercive with respect to compact $K$, then $D \oplus P_{B^{\rc},B} \circ \gamma: \dom(D_{\max}) \to \Lp{2}(M;E) \oplus B^{\rc}$ has closed range and finite dimensional kernel.
If further $D^\dagger$ is $B^\ast$-coercive with respect to a compact $K'$, then $D \oplus P_{B^{\rc},B} \circ \gamma: \dom(D_{\max}) \to \Lp{2}(M;E) \oplus B^{\rc}$ is Fredholm with the same index as $D_B$.
\end{corollary}

\begin{remark}
A boundary condition $B$ is always complemented since $\checkH(D)$ is a Hilbert space.
I.e., we can take $B^{\rc} = B^{\perp}$ with respect to any inner product in $\checkH(D)$.
\end{remark}

\subsection{Discrete spectrum} 
\label{Sec:Freddy:Discrete} 

In this subsection, we consider the situation that the adapted operator $A$ has discrete spectrum.
This, in particular, allows us to more closely mirror the situation encountered for compact $\dM$. 
More precisely, the coercivity requirement can be placed wholly on the operator and the geometry near the boundary, which allows for a larger class of boundary conditions beyond those studied in subsection~\ref{Sec:Freddy:BCs}.

To begin with, we note the following.
We recall that by $Y_{[0,r)}$, we denote the cylinder $[0,r) \times\dM$ for $r > 0$. 

\begin{lemma}
Suppose that the adapted operator on the boundary $A$ has discrete spectrum.
For $r > 0$ fixed, the space 
$$ \SobH[r]{1}(A,\partial_t) := \set{u \in \Lp{2}(Y_{[0,r)};E): \int_{0}^r \int_{\dM} \modulus{(\partial_t u)(t,x)}^2 + \modulus{Au(t,x)}^2\ d\mu_{\tau}(x)\ dt < \infty},$$
equipped with norm
$$ \norm{u}_{\SobH[r]{1}(A,\partial_t)}^2 = \norm{u}_{\Lp{2}(Y_{[0,r)};E)}^2 + \norm{\partial_t u}_{\Lp{2}(Y_{[0,r)};E)}^2 + \norm{Au}_{\Lp{2}(Y_{[0,r)};E)}^2$$ 
embeds compactly into $\Lp{2}(Y_{[0,r)}; E)$. 
\end{lemma}
\begin{proof}
First, we note that for $\dom(\partial_t) = \SobH{1}([0,r))$, we have that $\dom(\sqrt{\partial_t^\ast \partial_t}) = \dom(\partial_t)$.
So let $\tau = \sqrt{\partial_t^\ast \partial_t}$, which is a selfadjoint operator. 
Since $\SobH{1}([0,r)) \embed \Lp{2}([0,r))$ compactly,  Theorem~\ref{Thm:AbsCpctEmbed} guarantees us that $\tau$ has discrete spectrum. 
Let $\spec(\tau) = \set{\theta_k \geq 0}$ with orthogonalised eigenfunctions $\set{t_i}$.
Similarly, since $\dom(A) \embed \Lp{2}(\dM;E)$ is compact, we have that $\spec(A) = \set{ \lambda_i}$ with  orthonormalised eigensections $\set{a_i}$.
Consequently, given $u \in \Lp{2}(Y_{[0,r)};E)$, we can write $u(t,x) = \sum_{j} u_j(t) a_j = \sum_{i,j} u_{ij} t_i (t) a_i(x)$ since each $t \mapsto u_j(t) \in \Lp{2}([0,r)$ and therefore, $\norm{u}_{\Lp{2}(Y_{[0,r)})}^2 = \sum_{ij} \modulus{u_{ij}}^2$.

Consider the operator $\Gamma u = (\tau u, Au)$ with domain $\dom(\Gamma) = \SobH[r]{1}(A,\partial_t)$.
Via the expansion, $\Gamma u = (\sum_{ij} \theta_i u_{ij} t_i a_j, \sum_{ij} \lambda_j u_{ij} t_i a_j)$.
Now, for $v = (v^1, v^2) \in \dom(\Gamma^\ast)$, 
\begin{align*}  
\inprod{\Gamma^\ast v, u} 
=
\inprod{v, \Gamma u} 
&=
\inprod{\sum_{kl} v^1_{kl} t_k a_l, \sum_{ij} \theta_i u_{ij} t_i a_j} + \inprod{\sum_{kl} v^2_{kl} t_k a_l, \sum_{ij} \lambda_j u_{ij} t_i a_j} \\ 
&=
\sum_{ij}  \theta_i v^1_{ij} u_{ij} + \sum_{ij} v^2_{ij} \lambda_j u_{ij} \\ 
&= \sum_{ij} (\theta_i v^1_{ij} + \lambda_j v^2_{ij}) u_{ij}.
\end{align*}
That is,  $\Gamma^\ast (v_1, v_2) = \Gamma^\ast \tau v^1 + A v_2$.
From operator theory,  $\Gamma^\ast \Gamma  = (\tau^2 + A^2)$ is selfadjoint and moreover $\dom(\sqrt{\Gamma^\ast \Gamma}) = \dom(\Gamma) = \SobH[r]{1}(A,\partial_t)$ and $\norm{\sqrt{\tau^2 +A^2} u} = \norm{\tau u} + \norm{Au}$.
Clearly then, the compactness of the embedding of $\SobH[r]{1}(A,\partial_t) \embed \Lp{2}(Y_{[0,r)};E)$ is equivalent to the discreteness of the spectrum of $\sqrt{\tau^2 + A^2}$ which is equivalent to the discreteness of the spectrum of  $\Xi := \tau^2 + A^2$, which is again selfadjoint.

Now, for $\zeta \not \in \set{\theta_i^2 + \lambda_j^2}$, define the (possibly unbounded) map 
$$R_{\Xi}(\zeta)u := \sum_{ij} \frac{1}{\zeta - \theta_i^2 + \lambda_j^2} u_{ij} t_i a_j.$$
We claim that this is bounded.
Since $\tau^2 = \partial_t^\ast \partial_t$, which is a second-order differential operator, we have that $\theta_i^2 \to \infty$ and therefore, $\theta_i^2 + \lambda_j^2 \to \infty$.
Hence, we can find $N = N(\zeta)> 0$ such that for all $i,j \geq N$, $\modulus{\zeta - \theta_i^2 - \lambda_j^2} \geq 1$.
Therefore,
\begin{align*} 
\norm{R_{\Xi}(\zeta)u}^2 
&\leq 
\sum_{i,j \leq N} \frac{1}{\modulus{\zeta - \theta_i^2 - \lambda_j^2}} \modulus{u_{ij}}^2 + \sum_{i,j \leq N} \frac{1}{\modulus{\zeta - \theta_i^2 - \lambda_j^2}} \modulus{u_{ij}}^2 \\ 
&\leq \max\set{\frac{1}{\modulus{\zeta - \theta_i^2 - \lambda_j^2}}: i,j \leq N}  \sum_{i,j < N} \modulus{u_{ij}}^2 + \sum_{i,j \geq N}  \modulus{u_{ij}}^2 \\
&\leq \cbrac{1 + \max\set{\frac{1}{\modulus{\zeta - \theta_i^2 - \lambda_j^2}}: i,j \leq N}} \sum_{ij} \modulus{u_{ij}}^2 \\ 
&= \cbrac{1 + \max\set{\frac{1}{\modulus{\zeta - \theta_i^2 - \lambda_j^2}}: i,j \leq N}}  \norm{u}^2.
\end{align*}
Moreover, by definition, it is clear that $I = R_{\Xi}(\zeta - \Xi) = (\zeta - \Xi)R_{\Xi}$ and therefore, we conclude that $\zeta \not \in \spec(\Xi)$.
Taking the contrapositive, we see that $\zeta \in \spec(\Xi)$ implies that $\zeta \in \set{\theta_i^2 + \lambda_j^2}$ which is a discrete set.
Now, suppose that $0 = (\theta_l^2 + \lambda_m^2 - \Xi)u$, which occurs if and only if $0 = \sum_{ij} (\theta_l^2 + \lambda_m^2 - \theta_i^2 - \lambda_j^2) u_{ij} t_i a_j$. 
But since $\tau$ and $A$ both have discrete spectrum, $(\theta_l^2 + \lambda_m^2 - \theta_i^2 - \lambda_j^2) = 0$ for only finitely many $i$ and $j$ and therefore, each eigenspace is finite dimensional. 
That is, $\Xi$ has discrete spectrum. 
From Theorem~\ref{Thm:AbsCpctEmbed}, we deduce that $\SobH[r]{1}(A,\partial_t) \embed \Lp{2}(Y_{[0,r)};E)$ compactly. 
\end{proof} 

With the aid of this, we provide the proof of the following. 

\begin{proof}[Proof of Theorem~\ref{Thm:Freddy:Discrete}]
Let $T_d \in  (0,T)$ be from  Corollary~\ref{Cor:ASemiRegNearBoundary}.
Define $\xi: Y_{[0,T_d)} \to [0,1]$ with $\xi(t,x) = f(t)$, where $f \in \Ck[c]{\infty}([0,T_d])$ with $f = 1$ on $[0,\tfrac 14 T_d ]$ and $f = 0$ outside of $[0,\tfrac 34 T]$.
Then, writing $\chi_1 = \Phi^{-1}\xi$ and extending this to the entirety of $M$ by $0$, we obtain that $\chi_1 \in \Ck{\infty}(M)$ with $\spt \chi_1 Z_{[0,\tfrac34 T]}$ and $\chi_1 = 1$ on $Z_{[0,\tfrac 14 T]}$. 
By Assumption~\ref{Hyp:RemControl}, we have that $\modulus{[D, \chi_1]} \leq C \sup_{[0,1]} |f'| $.

Now, let $K'$ be a compact set with $K \setminus Z_{[0,\tfrac14]} \subsetneqq K'$  and let $\chi_2 = 1$ on $K \setminus Z_{[0,\tfrac14]}$ and $\chi_2 = 0$ outside of $K'$.

Let $u_k$ be a bounded sequence in $\dom(D_B)$ such that $D_B u_k \to v$.
We show that $u_k$ has a convergence subsequence in $\Lp{2}(M;E)$.
For that, write 
$$u_k = \chi_1 u_k + (1 - \chi_1)u_k = \chi_1 u_k + (1 - \chi_1)\chi_2 u_k  + (1 - \chi_1)(1 - \chi_2) u_k.$$
Clearly, $\spt \chi_1 u_k \subset Z_{[0,\tfrac14]}$, $\spt (1 - \chi_1)\chi_2 u_k \subset K'$ is compact and $\spt (1 - \chi_1)(1 - \chi_2) u_k \cap K \subset \spt (1 - \chi_1)(1 - \chi_2) u_k \cap K' = \emptyset$ with $\spt (1 - \chi_1)(1 - \chi_2) u_k \in \dom(D_{\min})$.

Now, 
\begin{align*} 
\norm{u_k - u_l}_{\Lp{2}(M)} 
&\leq 
\norm{\chi_1 (u_k - u_l)}_{\Lp{2}(M)} +   \norm{ (1 - \chi_1)\chi_2 (u_k - u_l) }_{\Lp{2}(M)} \\ 
&\qquad\qquad + \norm{(1 - \chi_1)(1 - \chi_2) (u_k - u_l)}_{\Lp{2}(M)} \\ 
&\leq 
\norm{\chi_1 (u_k - u_l)}_{\Lp{2}(Y_{[0,\frac12)})} + \norm{ (1 - \chi_1)\chi_2 (u_k - u_l) }_{\Lp{2}(K')}   \\
&\qquad\qquad+\norm{(1 - \chi_1)(1 - \chi_2) (u_k - u_l)}_{\Lp{2}(M)} \,  .
\end{align*}
By Corollary~\ref{Cor:ASemiRegNearBoundary}, we have that $ \norm{\chi_1 (u_k - u_l)}_{\SobH[T]{1}(A,\partial_t)} \lesssim \norm{\chi_1 (u_k - u_l)}_{\dom(D_B)}$ and therefore bounded, so the first term contains a convergent subsequence.
The second term is compactly supported and since $A$-semi-regular implies semi-regular, we again have a convergent subsequence.
Lastly,
\begin{align*} 
\norm{(1 - \chi_1)(1 - \chi_2) (u_k - u_l)}_{\Lp{2}(M)} 
& \leq C^{-1} \norm{D((1 - \chi_1)(1 - \chi_2) (u_k - u_l))}_{\Lp{2}(M)} \\ 
&\leq C^{-1} \norm{[D, d(2\chi_1\chi_2 + \chi_1 + \chi_2)]}_{\Lp{\infty}(M)} \norm{D(u_k - u_l)}_{\Lp{2}(M)} 
\end{align*}
The first term is uniformly bounded by the choice we made for $\chi_1$ and the compactness of the support of $\chi_2$ and since $D_Bu_k \to v$, we have that the convergence subsequences we have passed to in the first two terms also provide convergence in the latter term. 
Therefore, $D_B$ has finite dimensional kernel and closed range.

If further $D^\dagger$ satisfies the $0$-coercivity assumption, we see that $B^* \subset \dom(\modulus{\tilde{A}}^{\frac12})$.
Since \ref{Itm:Freddy:Discrete1} is also satisfied for $\sym_{D^\dagger}$, and $\tilde{A} = -(\sym_0^{-1})^\ast A \sym_0^\ast$  has discrete spectrum since $A$ has discrete spectrum, we obtain that $D_B^\ast = (D^\dagger)_{B^*}$ has closed range and finite dimensional kernel.
Therefore, $D_B$ is Fredholm.
\end{proof}

\section{Dirac-type operators}
\label{Sec:DiracType}

Dirac-type operators are undoubtedly the most important first-order elliptic operators in geometry.
In this section we derive geometric conditions that ensure that the results of the present work can be applied to them.

Let $D$ be a Dirac-type operator with respect to a Riemannian metric on $M$, i.e.\ its principal symbol satisfies the Clifford relation $\sigma_D(\xi)^*\sigma_D(\xi) = |\xi|^2\cdot\id$ for all covectors $\xi$.
This implies $|\sigma_D(\xi)|=|\xi|$ and, in particular, $D$ is elliptic.

Moreover, if the Riemannian metric is complete, Theorem~\ref{Thm:cherwolf} ensures that $D$ and $D^\dagger$ are complete.
Choosing $\vec{T}$ as the interior pointing unit normal vector field and $\tau$ the associated covector field, we are in the minimal setup.
Thus, the conditions of the minimal setup are satisfied for any Dirac-type operator if the underlying Riemannian metric is complete.
In particular, Theorems~\ref{Thm:Trace1} and \ref{Thm:Reg} apply.

\subsection{Twisted spinorial Dirac operators}
Ensuring conditions \ref{Hyp:ExtFirst}--\ref{Hyp:ExtLast} of the geometric setup is more subtle.
We specialise to twisted spinorial Dirac operators.
For this, we assume that $M$ carries a spin structure.
Let $SM\to M$ be the associated spinor bundle, equipped with its natural Hermitian metric and connection $\nabla^{SM}$.
Let $C\to M$ be a Hermitian vector bundle with compatible connection $\nabla^C$.
The twisted spinorial Dirac operator $D$ maps sections of $SM\otimes C$ to sections of $SM\otimes C$.
It is a formally selfadjoint Dirac-type operator and, in particular, elliptic.

The Weitzenböck formula for twisted spinorial Dirac operators says
\begin{equation}
D^2 = \nabla^*\nabla + \tfrac14 S + \KK^C
\label{eq:Weitzen}
\end{equation}
where $S$ is the scalar curvature of $M$ and $\KK^C = \tfrac12 \sum_{i,j=1}^{n} e_i\cdot e_j\cdot R^C(e_i,e_j)$, see Theorem~8.17 in Chapter~II of \cite{LM}.
Here $R^C$ denotes the curvature tensor of $C$, the dot $\cdot$ denotes Clifford multiplication, and $e_1,\dots,e_n$ is a local orthonormal tangent frame.

If $n=\dim(M)$ is odd, we put $E=F=SM\otimes C$.
If $n$ is even, then the spinor bundle splits into the chirality subbundles, $SM=S^+M\oplus S^-M$, and the Dirac operator interchanges these bundles.
In this case we put $E=S^+M\otimes C$ and $F=S^-M\otimes C$.

In both cases, $E|_{\dM}$ and $F|_{\dM}$ can be canonically identified with the twisted spinor bundle $S\dM\otimes C$ of $\dM$.
The intrinsic twisted Dirac operator $D^{\dM}$ of the boundary is one possible choice of adapted boundary operator.
Thus, a general adapted boundary operator will be of the form $A=D^{\dM} + V$ where $V$ is a symmetric endomorphism field of $E|_{\dM}$.
We always choose $\tilde A=A$.

We now list geometric assumptions which ensure that we are in the geometric setup so that all the results obtained in the previous sections can be applied.

\begin{assumptions}\label{DiracAssumptions}
Given a Hermitian vector bundle $C$ with compatible connection over a Riemannian manifold $M$, the following assumptions will be useful:
\begin{enumerate}[label=(A\arabic*), labelwidth=0pt, labelindent=2pt, leftmargin=26pt]
\item \label{eq:Hyp:DiracFocal}
      The focal radius of $\dM$ is positive, i.e.\ there exists an $r_0>0$ such that the normal exponential map induces a diffeomorphism $[0,r_0)\times\dM \to U$ where $U$ is a neighbourhood of $\dM$ in $M$.
\item \label{eq:Hyp:DiracFundamental}
      The pointwise norms of the second fundamental forms of $\dM$ and of its parallel hypersurfaces $\{x\in M : \mathrm{dist}(x,\dM)=r\}$ are uniformly bounded for small $r>0$.
\item \label{eq:Hyp:DiracRicci}
      The pointwise norm of the Ricci curvature tensor of $M$ is uniformly bounded on a distance tube of $\dM$.
\item \label{eq:Hyp:DiracTwist}
      The curvature tensor $R^C$ of $C$ is uniformly bounded on a distance tube of $\dM$.
\item \label{eq:Hyp:DiracV}
      The potential $V$ is uniformly bounded on $\dM$, i.e.\ there exists $c_0>0$ such that $|V|\le c_0$.
\end{enumerate}
\end{assumptions}

Note that these assumptions are automatic if $\dM$ is compact.

\begin{proposition}\label{prop:Dirac}
Let $M$ be a complete Riemannian spin manifold with smooth (and possibly noncompact) boundary $\dM$.
Let $C\to M$ be a Hermitian vector bundle with compatible connection.
Suppose that Assumptions~\ref{DiracAssumptions} hold.

Then conditions~\ref{Hyp:StdFirst}--\ref{Hyp:StdLast} of the minimal setup and \ref{Hyp:ExtFirst}--\ref{Hyp:ExtLast} of the geometric setup are satisfied for the twisted spinorial Dirac operator $D$ on $M$ together with the adapted operator $A=D^{\dM}+V$ on $\dM$.
\end{proposition}

\begin{proof}
Denote the induced connection on $SM\otimes C$ by $\nabla$ and $\cdot$ means Clifford multiplication.
At a point $x\in\dM$ choose an orthonormal basis $e_1,\dots,e_{n-1}$ of $T_x\dM$.
Then $e_1,\dots,e_{n-1},\vec{T}$ is an orthonormal basis of $T_x M$.
Then the Dirac operators take the form
\begin{align*}
D = \sum_{j=1}^{n-1}e_j\cdot \nabla_{e_j} + \vec{T}\cdot \nabla_{\vec{T}}
\quad\text{ and }\quad
A = -\sum_{j=1}^{n-1}\vec{T}\cdot e_j\cdot \nabla_{e_j} + \tfrac12 H + V
\end{align*}
where $H$ is the mean curvature of $\dM$.
This implies Condition~\ref{Hyp:ExtFirst}.

By \ref{eq:Hyp:DiracFocal} there exists an $r_0>0$ such that the normal exponential map $Y_{[0,r_0)}\to M$, $(t,x)\mapsto \exp_x(t\vec{T}(x))$, is a diffeomorphism onto its image $\tilde U$.
Extend $\vec{T}$ to $\tilde U$ as the velocity field of the normal geodesics, i.e.\ $\vec{T}(\exp_x(t\vec{T}(x))) = \frac{d}{dt}|_t\exp_x(t\vec{T}(x))$.

On $\tilde U$ the second fundamental form of the parallel hypersurfaces coincides with $\nabla\vec{T}$ because $\vec{T}$ is their unit normal vector field.
After possibly decreasing $r_0$, Assumption~\ref{eq:Hyp:DiracFundamental} gives us the bound
\begin{equation}
|\nabla\vec{T}| \le C_1
\label{eq:nablaTbound}
\end{equation}
on $\tilde U$.
In particular, the mean curvature $H=-\mathrm{tr}(\nabla\vec{T})=-\div(\vec{T})$ is bounded on $\tilde U$ by
\begin{equation}
|H|\le (n-1)C_1.
\label{eq:HBound}
\end{equation}
We define $f\colon \tilde U\to \R$ by
$$
f(\exp_x(t\vec{T}(x))) := \exp\left(\int_0^t H(\exp_x(s\vec{T}(x)))\, ds\right) .
$$
Then $f|_{\dM}=1$ and we have at $p=\exp_x(t\vec{T}(x))\in\tilde U$ that
\begin{align*}
\vec{T}(f)(p)
=
\tfrac{d}{dt}|_t f(\exp_x(t\vec{T}(x)))
=
f(p)H(p).
\end{align*}
The vector field $f\vec{T}$ is divergence free because
\begin{align*}
\div(f\vec{T})
=
\vec{T}(f) + f\div\vec{T}
=
fH - fH
=
0.
\end{align*}
From \eqref{eq:HBound} we get a bound for $f$ on $\tilde U$,
$$
\exp(-r_0C_1) \le f \le \exp(r_0C_1) \quad\text{ and }\quad |f\vec{T}|\le\exp(r_0C_1) .
$$

Choose $r_1\in (0,r_0\exp(-r_0C_1))$.
Then the integral curves of $f\vec{T}$ starting at $\dM$ lie in $\tilde U$ for parameter values in $[0,r_1]$.
The flow of $f\vec{T}$ yields a diffeomorphism $\Psi\colon Y_{[0,r_1)}\to U \subset \tilde U$.
The inverse $\Phi := \Psi^{-1}\colon U\to Y_{[0,r_1)}$ satisfies all conditions in \ref{Hyp:Metric}.
In particular, the last assertion holds because $f\vec{T}$ is divergence free and hence its flow is volume preserving.
So far, we have only used the upper bound on $H$.

As to \ref{Hyp:ExtLast}, we note that $|dt|=\frac{1}{f}$ on $U$ since $dt$ is dual to $\partial_t = f\vec{T}$.
Therefore, $\exp(-r_0C_1) \le |dt| \le \exp(r_0C_1)$.
Since the principal symbol of a Dirac-type operator is given by Clifford multiplication which is an isometry if the covector has length $1$, we find
$$
\exp(-r_0C_1) \le |\sigma_t| \le \exp(r_0C_1)
$$
as well as
$$
\exp(-r_0C_1) \le |\sigma_t^{-1}| \le \exp(r_0C_1) .
$$
The spinor bundles over $U$ and over $Y_{[0,r_1)}$ are identified by parallel translation along the normal geodesics, i.e.\ along the integral curves of $\vec{T}$.
This identification is a pointwise isometry.
We study how the Dirac operator changes along such an integral curve.
Fix $x\in\dM$ and choose an orthonormal basis $e_1,\dots,e_{n-1}$ of $T_x\dM$.
Then $e_1,\dots,e_{n-1},\vec{T}$ is an orthonormal basis of $T_xM$.
We extend the basis by parallel translation along the integral curve of $\vec{T}$.
The parallel extension of $\vec{T}$ is $\vec{T}$ itself because it is the velocity field of a geodesic.
The Dirac operator then takes the form
\begin{equation*}
D
=
\sigma_t \partial_t + \sum_{j=1}^{n-1} e_j(t)\cdot \nabla_{e_j(t)}
=
\sigma_t \Big(\partial_t + \sigma_t^{-1}\sum_{j=1}^{n-1} e_j(t)\cdot \nabla_{e_j(t)}\Big) .
\end{equation*}
Thus the remainder term in \ref{Hyp:ExtLast} is given by
\begin{align*}
R_t
 & =
\sigma_t^{-1}\sum_{j=1}^{n-1} e_j(t)\cdot \nabla_{e_j(t)} - \sigma_0^{-1}\sum_{j=1}^{n-1} e_j(0)\cdot \nabla_{e_j(0)} \\
 & =
\sigma_0^{-1}\bigg(f\sum_{j=1}^{n-1} e_j(t)\cdot \nabla_{e_j(t)} - \sum_{j=1}^{n-1} e_j(0)\cdot \nabla_{e_j(0)} \bigg) .
\end{align*}
The second equation follows from $\sigma_t = \frac{1}{f}\sigma_0$ since Clifford multiplication is parallel, and hence $\sigma_t^{-1} = f\sigma_0^{-1}$.
Since $\sigma_0$ is an isometry, we only need to estimate the expression in parentheses.
Using that the $e_j$ and Clifford multiplication are parallel we compute for $u$ parallel along the normal geodesic
\begin{align}
\nabla_{\vec{T}}\bigg(f\sum_{j=1}^{n-1} e_j\cdot \nabla_{e_j}u\bigg)
=
fH\sum_{j=1}^{n-1} e_j\cdot \nabla_{e_j}u + f\sum_{j=1}^{n-1} e_j\cdot \nabla_{\vec{T}}\nabla_{e_j}u .
\label{eq:Rtdot}
\end{align}
Writing $R$ for the curvature tensor on the spinor bundle, the curvature tensor of $SM\otimes C$ is given by $R^{SM\otimes C}=R\otimes \id_C + \id_{SM}\otimes R^C$.
Denote by $\Ric$ the Ricci curvature of $M$, considered as an endomorphism field.
We find
\begin{align}
\sum_{j=1}^{n-1} e_j\cdot \nabla_{\vec{T}}\nabla_{e_j}u
 & =
\sum_{j=1}^{n-1} e_j\cdot \big(R^{SM\otimes C}(\vec{T},e_j)+ \nabla_{e_j}\nabla_{\vec{T}}+\nabla_{\nabla_{\vec{T}}e_j-\nabla_{e_j}\vec{T}}\big)u \notag \\
 & =
\sum_{j=1}^{n-1} e_j\cdot \big(R(\vec{T},e_j)\otimes \id_C + \id_{SM}\otimes R^C(\vec{T},e_j) - \nabla_{\nabla_{e_j}\vec{T}}\big)u \notag               \\
 & =
-\tfrac12 \Ric(\vec{T})\cdot u + \sum_{j=1}^{n-1} e_j\cdot\otimes R^C(\vec{T},e_j)u - \sum_{j=1}^{n-1}e_j \cdot\nabla_{\nabla_{e_j}\vec{T}}u .
\label{eq:DiracVarBound}
\end{align}
Now we use Assumption~\ref{eq:Hyp:DiracRicci} saying that there exists a constant $C_2$ such that
\begin{equation}
|\Ric| \le C_2
\label{eq:RicciBound}
\end{equation}
on $U$.
Combining \eqref{eq:nablaTbound}, \eqref{eq:DiracVarBound}, \eqref{eq:RicciBound}, and Assumption~\ref{eq:Hyp:DiracTwist} yields
\begin{equation}
\bigg|\sum_{j=1}^{n-1} e_j\cdot \nabla_{\vec{T}}\nabla_{e_j}u\bigg|
\le
\tfrac{C_2}{2}|u| + C_3|u|  + (n-1)C_1|\nabla u|
\le C_4(|u|+|\nabla u|) .
\label{eq:RtBound}
\end{equation}
Recall that the adapted boundary operator is given by $A=-\sum_{j=1}^{n-1} \vec{T}\cdot e_j\cdot \nabla_{e_j}+\frac12 H + V$.
This, together with \eqref{eq:Rtdot}, \eqref{eq:RtBound}, and \ref{eq:Hyp:DiracV} yields
\begin{align*}
|\nabla_{\vec{T}}R_t u|
 & =
\bigg|\nabla_{\vec{T}}\bigg(f\sum_{j=1}^{n-1} e_j\cdot \nabla_{e_j}u\bigg)\bigg|                                       \\
 & \le
f \cdot|H| \cdot|\vec{T}\cdot (A-\tfrac12 H-V)u| + f\cdot C_4(|u|+|\nabla u|) \\
 & \le
C_5(|Au|+|u|+|\nabla u|)
\end{align*}
and therefore
\begin{align*}
|\nabla_{\vec{T}}R_t u|^2
 & \le
C_6(|Au|^2+|u|^2+|\nabla u|^2) .
\end{align*}
Since $u$ is parallel in $\vec{T}$-direction, the term $\nabla u$ only involves derivatives in directions tangential to $\dM$.
The connection $\nabla$ of the spinor bundle on $M$ differs from the corresponding connection $\nabla^{\dM}$ on $\dM$ by a universal expression in the second fundamental form.
Since the latter is uniformly bounded, we find
\begin{equation}
|\nabla_{\vec{T}}R_t u|^2
\le
C_7(|Au|^2+|u|^2+|\nabla^{\dM} u|^2) .
\end{equation}
This gives us
\begin{align*}
\tfrac{d}{dt}|R_tu|^2
 & =
\partial_{f\vec{T}}|R_tu|^2                          \\
 & =
2f\mathsf{Re}\langle\nabla_{\vec{T}}R_tu,R_tu\rangle \\
 & \le
2f|\nabla_{\vec{T}}R_tu||R_tu|                       \\
 & \le
f(|\nabla_{\vec{T}}R_tu|^2+|R_tu|^2)                 \\
 & \le
C_8(|Au|^2+|u|^2+|\nabla^{\dM} u|^2+|R_tu|^2) .
\end{align*}
Using $R_0=0$ and Gronwall's lemma we get
\begin{align*}
|R_tu|^2
\le
(\exp(C_8t)-1)(|Au|^2+|u|^2+|\nabla^{\dM} u|^2).
\end{align*}
If $u$ is smooth and compactly supported on $\dM$ we integrate over $\dM$ and obtain
\begin{align}
\|R_tu\|^2_{\Lp{2}(\dM)}
 & \le
(\exp(C_8t)-1)(\|Au\|^2_{\Lp{2}(\dM)}+\|u\|^2_{\Lp{2}(\dM)}+\|\nabla^{\dM} u\|^2_{\Lp{2}(\dM)}) .
\label{eq:RtL2estimate}
\end{align}
The Weitzenböck formula for Dirac operator on the boundary says
$$
(A-V)^2 = (\nabla^{\dM})^*\nabla^{\dM} + \tfrac14 S^{\dM} + \KK^{C,\dM}
$$
where $S^{\dM}$ is the scalar curvature of $\dM$ and $\KK^{C,\dM} = \tfrac12 \sum_{i,j=1}^{n-1} e_i\cdot e_j\cdot R^C(e_i,e_j)$.
Assumption~\ref{eq:Hyp:DiracTwist} yields
\begin{equation}
|\KK^{C,\dM}| \le C_9.
\label{eq:KC}
\end{equation}
The Gauss equation implies
$$
S^{\dM} = S^{M} - 2\langle \Ric(\vec{T}),\vec{T}\rangle + H^2 -|\nabla\vec{T}|^2
$$
along the boundary where $S^{M}$ is the scalar curvature of $M$.
Thus our bounds on $\Ric$ and $\nabla\vec{T}$ imply a uniform bound
\begin{equation}
|S^{\dM}| \le C_{10} .
\label{eq:SdM}
\end{equation}
The estimates \eqref{eq:KC}, \eqref{eq:SdM} and \ref{eq:Hyp:DiracV} imply
\begin{align*}
\|\nabla^{\dM} u\|^2_{\Lp{2}(\dM)}
 & =
((\nabla^{\dM})^*\nabla^{\dM}u,u)_{\Lp{2}(\dM)}                 \\
 & =
\Big(\big((A-V)^2-\tfrac14 S^{\dM}-\KK^{C,\dM}\big)u,u\Big)_{\Lp{2}(\dM)} \\
 & \le
\|Au\|^2_{\Lp{2}(\dM)} + C_{11}\|u\|^2_{\Lp{2}(\dM)} .
\end{align*}
Inserting this into \eqref{eq:RtL2estimate} gives us for $t\in [0,r_1]$
\begin{align*}
\|R_tu\|^2_{\Lp{2}(\dM)}
 & \le
(\exp(C_8t)-1)\cdot C_{11}\cdot (\|Au\|^2_{\Lp{2}(\dM)}+\|u\|^2_{\Lp{2}(\dM)}) \\
 & \le
t\cdot C_8\cdot \exp(C_8r_1)\cdot C_{11}\cdot (\|Au\|^2_{\Lp{2}(\dM)}+\|u\|^2_{\Lp{2}(\dM)}) .
\end{align*}
Thus \ref{Hyp:ExtLast} holds with $T=r_1$.
\end{proof}

As a consequence, Theorem~\ref{Thm:MaxDom} applies to twisted spinorial Dirac operators provided the manifold is complete and Assumptions~\ref{DiracAssumptions} hold.

\begin{remark}
Proposition~\ref{prop:Dirac} still holds if $M$ is possibly not spin but the operator $D$ is locally a twisted Dirac operator, provided Assumption~\ref{eq:Hyp:DiracTwist} holds for the curvature tensors of the local twist bundles near the boundary with a uniform constant.

For example, let $D$ be the Dirac operator of a spin$^c$ manifold and assume that the curvature of the determinant line bundle $L$ is uniformly bounded on a distance tube of $\dM$.
Then, locally, $D$ is a twisted Dirac operator with a coefficient bundle $C$ such that $C\otimes C=L$.
The curvatures are related by $R^C=\frac12 R^L$.
Thus, the locally occurring curvatures $R^C$ are uniformly bounded on the distance tube and Proposition~\ref{prop:Dirac} applies.
\end{remark}

Next we discuss the Fredholm property of boundary value problems for twisted spinorial Dirac operators.
We focus on two cases, the nonlocal APS boundary conditions and certain local boundary conditions.

\begin{proposition}\label{Prop:DiracBVP}
Let $M$ be a complete Riemannian spin manifold with smooth boundary $\dM$.
Let $C\to M$ be a Hermitian vector bundle, equipped with a compatible connection.
Let $D$ be the corresponding twisted Dirac operator.
Suppose that Assumptions~\ref{DiracAssumptions} hold.

Let $B=B_{\APS}(A)$ the APS boundary condition for $D$.
Let $K\subset \interior{M}$ be a compact subset.

If $\frac12 H+V\ge0$ on $\dM$ and there exists a constant $c>0$ such that
$$
\tfrac14 S^M + \KK^C \ge c
$$
on $M\setminus K$ in the sense of symmetric endomorphisms, then $\ker(D_B)$ is finite dimensional and $\ran(D_B)$ is closed.
\end{proposition}

\begin{proof}
For $u\in \Ck[c]{\infty}(M;E)$ with $\spt u \cap K = \emptyset$ the Weitzenböck formula \eqref{eq:Weitzen} gives us
\begin{align}
0 & =
\inprod{D^2 u, u}_{\Lp{2}(M)} - \inprod{\nabla^*\nabla u, u}_{\Lp{2}(M)} - \inprod{(\tfrac14 S^M + \KK^C) u, u}_{\Lp{2}(M)} \notag \\
  & =
\|Du\|_{\Lp{2}(M)}^2 - \int_{\dM}\inprod{\vec{T}\cdot Du,u}
- \|\nabla u\|_{\Lp{2}(M)}^2 - \int_{\dM}\inprod{\nabla_{\vec{T}}u,u}                                                        
- \inprod{(\tfrac14 S^M + \KK^C) u, u}_{\Lp{2}(M)} \notag                                                                    \\
  & \le
\|Du\|_{\Lp{2}(M)}^2 - \int_{\dM}\inprod{(\nabla_{\vec{T}}+\vec{T}\cdot D)u,u} - c\|u\|_{\Lp{2}(M)}^2 \notag                         \\
  & =
\|Du\|_{\Lp{2}(M)}^2 + \int_{\dM}\inprod{(A-\tfrac12 H-V)u,u} - c\|u\|_{\Lp{2}(M)}^2 \notag                                            \\
  & \le
\|Du\|_{\Lp{2}(M)}^2 + \int_{\dM}\inprod{Au,u} - c\|u\|_{\Lp{2}(M)}^2 \notag                                                         \\
  & \le
\|Du\|_{\Lp{2}(M)}^2 + \|\chi_{[0,\infty)}(A)|A|^{\frac12}u\|_{\Lp{2}(\dM)}^2 - c\|u\|_{\Lp{2}(M)}^2 .
\label{eq:DiracCoerc}
\end{align}
We choose a compact subset $K'\subset \interior{M}$ such that $K$ is contained in the interior of $K'$.
We show that $D$ is $B$-coercive with respect to $K'$.

Since $K'$ is contained in the interior of $M$, we can find $\eta_{K'}\in\Ck[c]{\infty}(M)$ with $\eta_{K'}=1$ on $K'$ and $\eta_{K'}=0$ on $\dM$.
Then $\eta_{K'}\dom(D_B)\subset\dom(D_B)$.
Moreover, let $\chi\in\Ck{\infty}(M)$ with $\chi\equiv 0$ on a neighbourhood of $K$ and $\chi\equiv 1$ outside $K'$.

Now let $u\in \dom(D_B)$ with $\spt u\cap K'=\emptyset$.
By Propositions~\ref{Prop:AEllRegApprox} and \ref{Prop:APS=Aelliptic}, there exist $u_n \in \Ck[c]{\infty}(M;E)$ such that $u_n \to u$ in the graph norm of $D$ and $u_n\rest{\dM} \to u\rest{\dM}$ in $\dom(\modulus{A}^{\frac12})$.
Put $\tilde{u}_n := \chi u_n\in\Ck[c]{\infty}(M;E)$.
Since $\chi\equiv 1$ on $\dM$ we have $\tilde{u}_n\rest{\dM} = u_n\rest{\dM} \to u\rest{\dM}$ in $\dom(\modulus{A}^{\frac12})$.
Moreover, $\tilde{u}_n \to \chi u = u$ and $D\tilde{u}_n = \chi D{u}_n + \nabla\chi\cdot {u}_n \to \chi D{u} + \nabla\chi\cdot {u} = Du$ in $L^2(M;E)$.
Applying \eqref{eq:DiracCoerc} to $\tilde{u}_n$ yields
\begin{align*}
\|D\tilde{u}_n\|_{\Lp{2}(M)}^2
 & \ge
c\|\tilde{u}_n\|_{\Lp{2}(M)} - \|\chi_{[0,\infty)}(A)|A|^{\frac12}\tilde{u}_n\|_{\Lp{2}(\dM)}^2
\end{align*}
and hence, by passing $n\to\infty$,
\begin{align*}
\|Du\|_{\Lp{2}(M)}^2
 & \ge
c\|u\|_{\Lp{2}(M)} - \|\chi_{[0,\infty)}(A)|A|^{\frac12}u|_{\Lp{2}(\dM)}^2
=
c\|u\|_{\Lp{2}(M)} .
\end{align*}
Thus $D$ is $B$-coercive with respect to $K'$.
Since $D^\dagger=D$ and $B^*$ coincides with $B$ up to the kernel of $A$, the operator $D^\dagger$ is $B^*$-coercive with respect to $K'$.
Theorem~\ref{Thm:Coercive} concludes the proof.
\end{proof}

\begin{corollary}
Let $M$ be a complete Riemannian spin manifold with smooth boundary $\dM$.
Let $C\to M$ be a Hermitian vector bundle, equipped with a compatible connection.
Let $D$ be the corresponding twisted Dirac operator.
Suppose that Assumptions~\ref{DiracAssumptions} hold.

Let $A=D^{\dM}$ be the intrinsic Dirac operator of the boundary, i.e.\ $V=0$.
Let $B=B_{\APS}(A)$ the APS boundary condition for $D$.
Let $K\subset \interior{M}$ be a compact subset.

If $H\ge0$ on $\dM$ and there exists a constant $c>0$ such that
$$
\tfrac14 S^M + \KK^C \ge c
$$
on $M\setminus K$, then $D_B$ is a Fredholm operator.
\end{corollary}

\begin{proof}
If $V=0$ then $A$ anticommutes with $\sigma_0$, hence $B^*$ coincides with $B$ up to the kernel of $A$ by Remark~\ref{rem:BAPS-adjoint}.
Therefore, Proposition~\ref{Prop:DiracBVP} applies to $D$ with the boundary condition $B$ and to $D^\dagger$ with boundary condition $B^*$.
Thus, $D_B$ and its adjoint have finite dimensional kernel and closed range, hence $D_B$ is Fredholm.
\end{proof}

Next we consider local boundary conditions.

\begin{proposition}\label{Prop:DiracLocalBVP}
Let $M$ be a complete Riemannian spin manifold with smooth boundary $\dM$.
Let $C\to M$ be a Hermitian vector bundle, equipped with a compatible connection.
Let $D$ be the corresponding twisted Dirac operator.
Suppose that Assumptions~\ref{DiracAssumptions} hold.

Let $\Xi$ be a chirality operator and let $B=B_{E_\pm}$ be one of the two corresponding local boundary conditions.
Let $K\subset M$ be a compact subset.

If $\frac12 H+V\ge0$ on $\dM$ and there exists a constant $c>0$ such that
$$
\tfrac14 S^M + \KK^C \ge c
$$
on $M\setminus K$ in the sense of symmetric endomorphisms, then $\ker(D_B)$ is finite dimensional and $\ran(D_B)$ is closed.
\end{proposition}

\begin{proof}
Let $u\in \Ck[c]{\infty}(M;E)$ with $\spt u \cap K = \emptyset$.
As in \eqref{eq:DiracCoerc} we find
\begin{equation*}
0 \le \|Du\|_{\Lp{2}(M)}^2 + \int_{\dM}\inprod{Au,u} - c\|u\|_{\Lp{2}(M)}^2 .
\end{equation*}
We choose a compact subset $K'\subset M$ such that $K$ is contained in the interior of $K'$.
We show that $D$ is $B$-coercive with respect to $K'$.

We pick $\eta_{K'}\in\Ck[c]{\infty}(M)$ with $\eta_{K'}\equiv 1$ on $K'$.
Since the boundary condition $B$ is local, we have that $\eta_{K'}\dom(D_B)\subset\dom(D_B)$.

For $u\in \dom(D_B)$ with $\spt u\cap K'=\emptyset$, Propositions~\ref{Prop:AEllRegApprox} and \ref{Prop:ChiralReg} allow us to use the same approximation as in the proof of Proposition~\ref{Prop:DiracBVP} to get
\begin{align}
\|Du\|_{\Lp{2}(M)}^2
 & \ge
c\|u\|_{\Lp{2}(M)} - \int_{\dM}\inprod{Au,u} .
\label{eq:DiracLocal1}
\end{align}
Now suppose $B=B_{E_+}$, the case $B=B_{E_-}$ being completely analogous.
Then we have along $\dM$
\begin{align*}
\inprod{Au,u}
&=
\inprod{A\Xi_+u,\Xi_+u}
=
-\inprod{\Xi_+Au,\Xi_+u}
=
-\inprod{Au,\Xi_+u}
=
-\inprod{Au,u},
\end{align*}
hence $\inprod{Au,u}=0$.
Therefore, \eqref{eq:DiracLocal1} implies
\[
\|Du\|_{\Lp{2}(M)}^2 \ge c\|u\|_{\Lp{2}(M)} .
\]
Thus $D$ is $B$-coercive with respect to $K'$.
Theorem~\ref{Thm:Coercive} concludes the proof.
\end{proof}

\begin{corollary}\label{Cor:DiracLocalBVP}
Let $M$ be a complete Riemannian spin manifold with smooth boundary $\dM$.
Let $C\to M$ be a Hermitian vector bundle, equipped with a compatible connection.
Let $D$ be the corresponding twisted Dirac operator.
Let $A=D^{\dM}$ be the intrinsic Dirac operator of the boundary, i.e.\ $V=0$.
Suppose that Assumptions~\ref{DiracAssumptions} hold.

Let $\Xi$ be a chirality operator and let $B=B_{E_\pm}$ be one of the two corresponding local boundary conditions.
Let $K\subset M$ be a compact subset.

If $H\ge0$ on $\dM$ and there exists a constant $c>0$ such that
$$
\tfrac14 S^M + \KK^C \ge c
$$
on $M\setminus K$ in the sense of symmetric endomorphisms, then $D_B$ is a Fredholm operator.
\end{corollary}

\begin{proof}
Since $V=0$, the boundary operator anticommutes with $\sym_0$.
By Remark~\ref{rem:ChiralAdjoint}~\ref{rem:ChiralAdjoint.2}, the adjoint boundary condition $B^*$ is again a chiral boundary condition.
Corollary~\ref{Cor:Coercive} implies that $D_{B}$ is Fredholm.
\end{proof}

The authors of \cite{GN} consider boundary value problems for the spin$^c$ Dirac operator in the context of noncompact boundary.
These results are obtained under the assumption of  bounded geometry.
Their analysis resorts to the results of \cite{BB12} through localisation and therefore are confined to the study of local boundary conditions.
To obtain Fredholmness results, they impose additional operator theoretic assumptions in addition to their geometric conditions.

\subsection{Callias potentials}
\label{S:Callias}

Assume here that $D: \Ck{\infty}(M;E) \to \Ck{\infty}(M;E)$ is a formally selfadjoint Dirac-type operator on a complete Riemannian manifold $M$.
In this case, recall that an adapted operator $A$ can be chosen so that $\sym_0 A = - A \sym_0$. 
Throughout this subsection, we fix $A$ to be such an operator. 

If we write $\DDD = D + \imath \Phi$ for a symmetric potential $\Phi$, the formal adjoint is given by $\DDD^\dagger = D - \imath \Phi$.
Moreover, we have $\DDD^\dagger\DDD = D^2 + \Phi^2 + \imath[D,\Phi]$ and $\DDD\DDD^\dagger = D^2 + \Phi^2 - \imath[D,\Phi]$.
A relevant class of potentials is characterised in the following definition.

\begin{definition}[Callias potential]
Let $\Phi \in \Ck{\infty}(M;\End(E))$ be a symmetric potential such that $[D,\Phi]$ is $0$-th order.
We say that $\Phi$  is a Callias potential for $D$ if there exists a constant $\Lambda > 0$ and a compact subset $K \subset M$ such that 
$$ 
h^E_x( (\Phi^2 + \imath[D,\Phi])(x) v,v) > \Lambda \modulus{v}_{h^E_x}^2
$$
for all $v \in E_x$ whenever $x \in M \setminus K$.
\end{definition}

Note that $\Phi$ is a Callias potential in particular means that $\Phi(x)^2 + \imath[D,\Phi](x) > \Lambda$ for $x \in M \setminus K$ in the sense of symmetric endomorphisms.

By the completeness of the metric on $M$, we immediately obtain $D$ and $\DDD$ satisfy \ref{Hyp:StdFirst}-\ref{Hyp:StdLast}.
From here on, we further assume that $D$ satisfies \ref{Hyp:ExtFirst}-\ref{Hyp:ExtLast}. 
Let $Z_{[0,T)}$ be the  cylindrical neighbourhood for $\dM$ with adapted operator $A$ so that $D = \sym_t(\partial_t + A + R_t)$.
Then, setting $\Phi_0 = \Phi\rest{\dM}$, 
$$\DDD = \sym_t(\partial_t + A - \imath \sym_0 \Phi_0 + (R_t  + \imath(\sym_0 \Phi_0 - \sym_t \Phi)) = \sym_t(\partial_t + \AAA + \RRR_t),$$
where we let 
\begin{align*}
\AAA &:= A - \imath \sym_0 \Phi_0 \quad\text{ and }\quad
\RRR_t :=  R_t  + \imath(\sym_0 \Phi_0 - \sym_t \Phi).
\end{align*}
Moreover, $\RRR_t$ does not differentiate in $t$ and therefore it is a remainder term for $\DDD$.

Since $[D,\Phi]$ is of order zero, the principal symbol of $D$ commutes with $\Phi$.
In particular, $\sym_0$ commutes with $\Phi_0$.
As $\sym_0$ is skewsymmetric, $\AAA$ is obtained by adding a skewsymmetric potential $\sym_0 \Phi_0$. 

Let $[X,Y]_+ = XY + YX$, the anticommutator of $X$ and $Y$. 
\begin{definition}[Para-Callias\footnote{We thank Claudia Grabs for suggesting this nomenclature.} potential] 
Let $\Psi \in \Ck{\infty}(M;\End(E))$ be a skewsymmetric potential such that $[D,\Psi]$ is $0$-th order. 
We say that $\Psi$ is a para-Callias potential if there exists a constant $\Lambda > 0$ and a compact subset $K \subset M$ such that 
$$ 
h^E_x( (\imath\Psi)^2 + \imath[D,\Psi]_+)(x) v,v) > \Lambda \modulus{v}_{h^E_x}^2
$$
for all $v \in E_x$ whenever $x \in M \setminus K$.
\end{definition} 

In this case, the obtained operator is $\DDD := D + \imath \Psi$, which is again formally selfadjoint.

Let us consider the Callias potential $\Phi$.
We see that $-\sym_0 \Phi_0$ on the boundary is skewsymmetric.
Therefore, the adapted boundary operator $\AAA$ arising from adding formally selfadjoint and satisfies
$$\AAA^\dagger \AAA = \AAA^2  = A^2  + \Phi_0^2 + \imath \sym_0[A,\Phi_0].$$
Clearly $\Phi_0^2 + \imath \sym_0[A,\Phi_0]$ is symmetric. 

\begin{proposition}
\label{Prop:SetupSticks}
Let $D: \Ck{\infty}(M;E) \to \Ck{\infty}(M;E)$ be a formally selfadjoint Dirac-type operator and  $\Phi \in \Ck{\infty}(M;\End(E))$ a Callias potential for $D$ on a complete Riemannian manifold $M$. 
Assume the following.
\begin{enumerate}[label=(\roman*), labelwidth=0pt, labelindent=2pt, leftmargin=21pt]
\item 
\label{Itm:SetupSticks:1}
$D$ satisfies the  extended setup \ref{Hyp:ExtFirst}-\ref{Hyp:ExtLast}. 
\item 
\label{Itm:SetupSticks:2} 
There exists $t' \leq T$ such that $\modulus{\sym_t(x) \Phi(t,x) - \sym_0(x)\Phi_0(x)} \leq C'$ uniformly for $(t,x) \in Z_{[0,t')}$.
\item 
\label{Itm:SetupSticks:3}
The potential on the boundary $-\sym_0 \Phi_0$ is a para-Callias potential for $A$. 
\end{enumerate}
Then, $\DDD$ satisfies  \ref{Hyp:StdFirst}-\ref{Hyp:StdLast} and \ref{Hyp:ExtFirst}-\ref{Hyp:ExtLast}.
\end{proposition}
\begin{proof}
Since the metric is assumed to be complete, $D$ satisfies \ref{Hyp:StdFirst}-\ref{Hyp:StdLast} as does  $\DDD$, as as their principal symbols coincide.
It is easily verified that $\AAA$ satisfies \ref{Hyp:ExtFirst} and \ref{Hyp:Metric}, particularly since it is itself Dirac type and by the completeness of the metric, it is selfadjoint. 
The first inequality of \ref{Hyp:RemControl} is also a consequence of the fact that it is satisfied for $D$. 

It remains to show that the remainder terms are appropriately controlled in terms of $\AAA$.
Here, we move from $T \in (0,T_0)$ in \ref{Hyp:RemControl} to $t_0 := \min\set{T, t'}$.
Now fix $u \in \Ck[c]{\infty}(\dM;E)$.
Then, 
\begin{align*} 
\norm{\RRR_t u}_{\Lp{2}(\dM)} 
\leq 
\norm{R_t u + \imath (\sym_0 \Phi_0 - \sym_t \Phi)u}  
\leq 
\norm{R_t u} + \norm{(\sym_0 \Phi_0 - \sym_t \Phi)u} 
\leq 
C t\norm{A u} + C'\norm{u},
\end{align*}
where the ultimate inequality follows from Assumption~\ref{Itm:SetupSticks:2}.

We show that $\norm{Au} \lesssim \norm{\AAA u} + \norm{u}$.
For that, recall that $\AAA^\dagger\AAA = \AAA^2  = A^2 + \imath \sym_0[A,\Phi_0] + \Phi_0^2$ and therefore, 
$$
\inprod{\AAA^2 u, u} = \inprod{A^2 u, u} + \inprod{(\imath \sym_0[A,\Phi_0] + \Phi_0^2)u, u}.
$$
Since $-\sym_0\Phi_0$ is para-Callias, on setting $\Psi = - \imath \Phi_0$ and we have  $\Lambda_{\AAA} > 0$ and $K_{\AAA}$, such that 
$$
\Lambda_{\AAA}
<
(\imath \Psi)^2 + \imath[A,\Psi] 
= 
(-\imath \sym_0 \Phi_0)^2 + \imath[A, -\sym_0 \Phi_0]_+ 
=
\Phi_0^2 + \imath\sym_0[A,\Phi_0] $$ 
on $\dM \setminus K_{\AAA}$. 
Therefore, 
\begin{align*} 
\big\langle (\imath \sym_0[A,\Phi_0] &+ \Phi_0^2)u, u\big\rangle_{\Lp{2}(\dM)} \\ 
&=  \inprod{(\imath \sym_0[A,\Phi_0] + \Phi_0^2)u, u}_{\Lp{2}(\dM \setminus K_{\AAA})} + \inprod{(\imath \sym_0[A,\Phi_0] + \Phi_0^2)u, u}_{\Lp{2}(K_{\AAA})} \\
&\geq
\norm{u}^2_{\Lp{2}(\dM \setminus K_{\AAA})} +  \inprod{(\imath \sym_0[A,\Phi_0] + \Phi_0^2)u, u}_{\Lp{2}(K_{\AAA})} \\ 
&\geq 
\norm{u}^2_{\Lp{2}(\dM \setminus K_{\AAA})} - \Lambda_{\AAA} \norm{u}^2_{\Lp{2}(K_{\AAA})}. 
\end{align*}
That is,
$$\inprod{(\imath \sym_0[A,\Phi_0] + \Phi_0^2)u, u}_{\Lp{2}(\dM)} + \Lambda_{\AAA} \norm{u}_{\Lp{2}}(\dM)^2 \geq \norm{u}^2_{\Lp{2}(\dM \setminus K_{\AAA})} \geq 0.$$
Therefore, 
\begin{equation}
\label{Eq:Rem}
\begin{aligned}
\norm{Au}^2 
&\leq  \inprod{A^2 u, u} + \inprod{(\imath \sym_0[A,\Phi_0] + \Phi_0^2)u, u}_{\Lp{2}(\dM)} + \Lambda_{\AAA} \norm{u}_{\Lp{2}(\dM)}^2  \\ 
&\lesssim  \inprod{\AAA^2u,u}  + \norm{u}^2 
= \norm{\AAA u}^2  + \norm{u}^2.
\end{aligned}
\end{equation}

We now compute the required bound on the remainder term $\tilde{\RRR}_t$ corresponding to $\DDD^\dagger$.
Note that $ \tilde{\RRR}_t = R_t + \imath (\phi_t \Phi - \sym_0 \Phi_0)$, and as before, the latter term is bounded in $\Lp{2}(\dM)$.
Therefore,  it remains to show that $\norm{R_t u } \lesssim t \norm{\tilde{\AAA}u} + \norm{u}$, for the adapted operator  $\tilde{\AAA} =A+\imath \sym_0\Phi_0$.
From \ref{Hyp:RemControl} for $D$, as before, we have $\norm{R_t u} \lesssim t\norm{A u} + \norm{u}$ and so it suffices to show, as before, $\norm{A u} \lesssim \norm{\tilde{\AAA} u}  + \norm{u}$.
The crucial observation here is that since $\sym_0$ anticommutes with $A$ and commutes with  $\Phi_0$, $\tilde{\AAA} = -\sym_0^{-1} \AAA \sym_0$.
Since $\sym_0$ is an isometry, we obtain  $\norm{\AAA u} \simeq \norm{\tilde{\AAA} u}$. 
The desired estimate $\norm{A u} \lesssim \norm{\tilde{\AAA} u} + \norm{u}$ then follows from Eq.~\eqref{Eq:Rem}.

Together, these calculations show $\DDD$ satisfies \ref{Hyp:StdFirst}-\ref{Hyp:StdLast} and \ref{Hyp:ExtFirst}-\ref{Hyp:ExtLast}.

\end{proof}

\begin{remark}
\label{Rem:paraCallias} 
Assumption~\ref{Itm:SetupSticks:3} in Proposition~\ref{Prop:SetupSticks} can be replaced by the assumption that $+\sym_0 \Phi_0 \in \Ck{\infty}(\dM;\End(E))$ is a para-Callias potential.
This is seen from the fact that $A+\imath \sym_0\Phi_0 = -\sym_0^{-1} (A - \imath \sym_0 \Phi_0) \sym_0$ which shows that $\norm{(A + \imath \sym_0 \Phi_0)u} \simeq \norm{(A - \imath \sym_0 \Phi_0)u}$ with equality of domains.
\end{remark}

In order to consider semi-Fredholmness and Fredholmness, we need to ensure that $\AAA$ has discrete spectrum.
To that end, motivated by \cite{BS0}, we define the following. 

\begin{definition}[Strongly para-Callias potential]
Let $\Psi \in \Ck{\infty}(M;\End(E))$ be skewsymmetric. 
We say that $\Psi$ is \emph{strongly para-Callias potential} if  for any $R > 0$, there exists a compact subset $K_R \subset M$ such that:
$$ h^E_x ( (\imath \Psi)^2+ \imath  [D,\Psi]_+)(x) v,v) \geq R \modulus{v}_{h^E_x}^2.$$
for all $v \in E_x$ and for all $x \in M \setminus K_R$.
\end{definition}

Similar reasoning as found in Section 3.10 in \cite{BS0} shows the following proposition.
\begin{proposition}
\label{Prop:DiscreteSpecA}
Let $\mathbb{A}$ be a formally selfadjoint Dirac-type operator on a complete Riemannian manifold without boundary.
Let $\Psi$ be a strongly para-Callias potential for $\mathbb{A}$.
Then the operator $\mathbb{A}+\Psi$ has discrete spectrum. 
\hfill\qed
\end{proposition}

Proposition~\ref{Prop:DiscreteSpecA} will be applied to $\mathbb{A}=A$ and $\Psi=-\sym_0 \Phi_0$ on $\dM$.
\begin{theorem}
\label{Thm:Callias}
Let $D: \Ck{\infty}(M;E) \to \Ck{\infty}(M;E)$ be a formally selfadjoint Dirac-type operator on a complete Riemannian manifold $M$ satisfying the  extended setup \ref{Hyp:ExtFirst}-\ref{Hyp:ExtLast}.
Assume the following. 
\begin{enumerate}[label=(\roman*), labelwidth=0pt, labelindent=2pt, leftmargin=21pt]
\item\label{Thm:Callias:1}
$\Phi \in \Ck{\infty}(M;\End(E))$ is a Callias potential.
\item\label{Thm:Callias:2}
$-\sym_0 \Phi_0 \in \Ck{\infty}(\dM;\End(E))$ is a strongly para-Callias potential.
\item\label{Thm:Callias:3}
There exists $t' \in (0,T)$ and $C' < \infty $  such that $\modulus{\sym_t(x) \Phi(t,x) - \sym_0(x)\Phi_0(x)} \leq C'$ uniformly for $(t,x) \in Z_{[0,t')}$.
\end{enumerate} 
If $B$ is an $A$-semi-regular boundary condition for $\DDD = D + \imath \Phi$, then $\DDD_B$ has finite dimensional kernel and closed range. 
\end{theorem}
\begin{proof}
The fact that $-\sym_0 \Phi_0$ is a strongly para-Callias potential implies that it is also a para-Callias potential, and since the hypothesis of the theorem provide the remainder of the assumptions of Proposition~\ref{Prop:SetupSticks}, we have that $\DDD$ satisfies the minimal and extended setups. 
By the same assumption, we have by Proposition~\ref{Prop:DiscreteSpecA} that $\AAA$ has discrete spectrum.
Moreover, the fact that $\Phi$ is a Callias potential for $D$ yields
$$ \norm{\DDD u} \geq \Lambda \norm{u}$$ 
for all $u \in \Ck[cc]{\infty}(M;E)$ with $\spt u \cap K = \emptyset$.
Therefore, $\DDD$ is $0$-coercive.
Assuming that $B$ is $A$-semi-regular,  so that \ref{Itm:Freddy:Discrete4} is satisfied, we invoke Theorem~\ref{Thm:Freddy:Discrete}. 
This yields that $\DDD_B$ has finite dimensional kernel and closed range.
\end{proof}

\begin{remark}
Assumption~\ref{Thm:Callias:2} in Theorem~\ref{Thm:Callias} can be replaced by the assumption that $+\sym_0 \Phi_0 \in \Ck{\infty}(\dM;\End(E))$ is a strongly para-Callias potential.
This is for the same reason as in Remark~\ref{Rem:paraCallias}, though in this case, the fact that $A + \imath \sym_0 \Phi_0 = -\sym_0^{-1}(A - \imath \sym_0 \Phi_0) \sym_0$ ensures that $A + \imath\sym_0 \Phi_0$ has discrete spectrum if and only if $A - \imath\sym_0 \Phi_0$ has discrete spectrum. 
\end{remark}

\begin{corollary}
\label{Cor:Callias}
Let $D: \Ck{\infty}(M;E) \to \Ck{\infty}(M;E)$ be a formally selfadjoint Dirac-type operator on a complete Riemannian manifold $M$ satisfying the  extended setup \ref{Hyp:ExtFirst}-\ref{Hyp:ExtLast}.
Assume the following. 
\begin{enumerate}[label=(\roman*), labelwidth=0pt, labelindent=2pt, leftmargin=21pt]
\item $\pm\Phi \in \Ck{\infty}(M;\End(E))$ are Callias potentials.
\item $-\sym_0 \Phi_0 \in \Ck{\infty}(\dM;\End(E))$ is a strongly para-Callias potential.
\item There exists $t' \in (0,T)$ and $C' < \infty $  such that $\modulus{\sym_t(x) \Phi(t,x) - \sym_0(x)\Phi_0(x)} \leq C'$ uniformly for $(t,x) \in Z_{[0,t')}$.
\end{enumerate} 
If $B$ is an $A$-regular boundary condition, then $\DDD_B$ is Fredholm.
\end{corollary}

\begin{proof}
Theorem~\ref{Thm:Callias} can be applied to $\DDD$ with boundary condition $B$ and to $\DDD^\dagger$ with boundary condition $B^*$.
Hence, $\DDD_B$ is a Fredholm operator.
\end{proof}

\begin{remark}
\label{Rem:Callias} 
In the literature, the condition that is used to call the operator $\DDD$ to be Callias type with the estimate  
\[
h^E_x( \Phi^2(x) - \modulus{[D,\Phi](x)}_{h^E_x \to h^E_x} u, u)
\geq \Lambda \modulus{u}_{h^E_x}^2
\]
for all $u \in E_x$ and $x \in M \setminus K$.
Since 
\begin{align*} 
\modulus{ h^{E}_x (\imath [D,\pm \Phi](x) u(x) , u(x))} 
=
\modulus{ h^{E}_x (\imath [D,\Phi](x) u(x) , u(x))} 
\leq 
\modulus{[D,\Phi](x)}_{h^E_x \to h^E_x}  \modulus{u(x)}_{h^E_x},
\end{align*} 
we find 
\begin{align*} 
h^E_x(\imath \sym_0 [D,\pm\Phi](x) + \Phi^2(x) u, u) 
&= 
h^E_x(\imath \sym_0 [D,\pm\Phi](x) u, u) + h^E_x(\Phi^2(x) u, u)   \\ 
&\geq 
-\modulus{[D,\Phi](x)}_{h^E_x \to h^E_x} \modulus{u}_{h^E_x}^2 + h^E(\Phi^2(x) u u, u) \\ 
&= 
 h^E_x(  \Phi^2(x) - \modulus{[D,\Phi](x)}_{h^E_x \to h^E_x} u, u).
\end{align*}
Thus, if $\Phi$ is a potential giving rise to a Callias-type operator in the classic sense, then $\pm\Phi$ are Callias potentials in our sense. 
The same remark applies to (strongly) para-Callias potentials.
\end{remark}

In particular, as a consequence of Theorem~\ref{Thm:Callias} coupled with Remark~\ref{Rem:Callias}, we obtain the Fredholmness results obtained in \cites{BS0, BS1, BS2} through our setup. 
Their setup is more restrictive with $R_t = 0$, which is conceptually the assumption that  $D$ is cylindrical with respect to a uniformly cylindrical neighbourhood. 
Moreover, the authors assume that $\Phi$ is constant in $t$. 
Therefore, the assumptions we have made in Theorem~\ref{Thm:Callias} are satisfied in their setup.
Nevertheless, the primary motivation of these papers were to compute the index of these operators and associated eta invariants which require more restrictive assumptions as demanded by the authors.

\appendix

\section{Auxiliary functional analytic tools}

\subsection{Czech spaces for selfadjoint operators}

Here we collect some functional analytic tools which were used in the main body of the paper.
Throughout this section let $\Hil$ be a Hilbert space and $T$ be a (generally unbounded) selfadjoint operator on $\Hil$.
By the Borel functional calculus we are able to construct bounded projectors $\chi_{S}(T)$, where $S \subset \R$ is a Borel set and $\chi_S$ its characteristic function.
Let us write $\chi^+(T) = \chi_{[0,\infty)}(T)$ and $\chi^-(T) = \chi_{(-\infty, 0)}(T)$ and define $\modulus{T} := T \sgn(T)$, where $\sgn(T) = \chi^+(T) - \chi^{-}(T)$.
Also, write $S_{\mu+}^o$ to be the open sector of angle $\mu$ centred at the origin which is symmetric about the positive real line.
Let $S_{\mu+}$ be the closure of $S_{\mu+}^o$.
We say $\psi \in \Psi(S_{\mu+}^o)$ if it is holomorphic on $S_{\mu+}^o$ and there exists some $\alpha > 0$ and $C > 0$ with $\modulus{\psi(\zeta)} \leq C \min\set{\modulus{\zeta}^\alpha, \modulus{\zeta}^{-\alpha}}$.

\begin{proposition}
\label{Prop:FCProp}
The following hold:
\begin{enumerate}[label=(\roman*), labelwidth=0pt, labelindent=2pt, leftmargin=21pt]
\item \label{Prop:FCProp:1}
      $\chi_{I}(T)$ is a bounded selfadjoint projector,
\item \label{Prop:FCProp:2}
      $\dom(T) = \dom(\modulus{T})$,
\item \label{Prop:FCProp:3}
      $\modulus{T} \geq 0$ and selfadjoint,
\item \label{Prop:FCProp:4}
      $T = \modulus{T}\sgn(T)$,
\item \label{Prop:FCProp:5}
      there is a splitting $\Hil = \nul(\modulus{T}) \stackrel{\perp}{\oplus} \close{\ran(\modulus{T})} = \nul(T) \stackrel{\perp}{\oplus} \close{\ran(T)}$,
\item  \label{Prop:FCProp:6}
      $\modulus{T}$ admits an \Hinfty-functional calculus: for every $\mu > 0$ and nonzero $\psi \in \Psi(S_{\mu+}^o)$,
      $$ \int_{0}^\infty \norm{\psi(t \modulus{T})u}^2\ \frac{dt}{t} \simeq \norm{u}^2$$
      for $u \in \close{\ran(\modulus{T})}$.
\end{enumerate}
\end{proposition}
\begin{proof}
Assertions~\ref{Prop:FCProp:1}--\ref{Prop:FCProp:5} are immediate consequences from Borel functional calculus.
Using condition~\ref{Prop:FCProp:3}, and utilising Corollary~7.1.6  in \cite{Haase} we obtain that $\modulus{A}$ has an \Hinfty-functional calculus for each $\mu > 0$.
Since  \ref{Prop:FCProp:5} gives us that $\modulus{A}\rest{\close{\ran(\modulus{A})}}$ is injective, the McIntosh Theorem, Theorem 7.3.1  in \cite{Haase}, furnishes us with the required estimate in \ref{Prop:FCProp:6}.
\end{proof}

\begin{lemma}
\label{Lem:Frac0}
For $\alpha > 0$, the spaces $\dom(\modulus{T}^\alpha)$ are Hilbert, and hence reflexive.
Any dense subset $\core \subset \Hil$ is also dense in $\dom(\modulus{T}^\alpha)^\ast$, the dual space of $\dom(\modulus{T}^\alpha)$.
\end{lemma}
\begin{proof}
To see that $\dom(\modulus{T}^\alpha)$ are Hilbert, note that $\inprod{u,v}_{ \dom(\modulus{T}^\alpha)} = \inprod{u,v} + \inprod{ \modulus{T}^\alpha u, \modulus{T}^\alpha v}$ induces the norm $\norm{\cdot}_{\dom(\modulus{T}^\alpha)}$.

For $v \in \core$, let $F_v(u) = \inprod{v, u}$, where $u \in \dom(\modulus{T}^\alpha)$, which yields that $F_v \in \dom(\modulus{T}^\alpha)^\ast$.
Take $\Lambda = \set{F_v: v \in \core}$ and  let $\close{\Lambda}$ be the closure of this set in $\dom(\modulus{T}^\alpha)^\ast$.
Now, suppose that $\xi \in \dom(\modulus{T}^\alpha)^\ast \setminus \close{\Lambda}$ and by the Hahn-Banach theorem, let $l \in \dom(\modulus{T}^\alpha)^{\ast\ast}$ such that $l(\xi) \neq 0$ and $l\rest{\close{\Lambda}} \equiv 0$.
By reflexivity of $\dom(\modulus{T}^\alpha)$, there exists $f_l \in \dom(\modulus{T}^\alpha)$ such that $l(\xi') = \xi'(f_l)$ for all $\xi' \in \dom(\modulus{T}^\alpha)^\ast$.
But then, $0 = l(F_v) = F_v(f_l) = \inprod{v, f_l}$ for all $v \in \core$ which is a dense subset of $\Hil$ and therefore, we get that $f_l = 0$.
This yields a contradiction on recalling $0 \neq l(\xi) = \xi(f_l)$.
\end{proof}

As we saw in the proof of Lemma~\ref{Lem:Frac0}, the graph norm on $\dom(\modulus{T}^\alpha)$ carries an inhomogeneous term.
In calculations, it can be cumbersome to carry this inhomogeneity.
Therefore, we instead consider the the operator $\modulus{T}_\epsilon = \modulus{T} + \epsilon I$ for some $\epsilon > 0$.
In the following lemma, we show that this operator is invertible and the homogeneous norm of $\modulus{T}_\epsilon^\alpha$ is equivalent (up to a constant)  to the usual graph norm on $\modulus{T}^\alpha$.
Moreover, we will see that it provides us with a way of computing the norm on $\dom(\modulus{T}^\alpha)^\ast$.

\begin{lemma}
\label{Lem:Frac}
The operator $\modulus{T}_\epsilon > 0$ and in particular invertible.
Moreover, it has an \Hinfty-functional calculus.
For $\alpha \geq 0$,  $\dom(\modulus{T}^\alpha) = \dom(\modulus{T}_\epsilon^\alpha)$ with estimate
$$\norm{\modulus{T}^\alpha}^2 + \norm{u}^2 \simeq \norm{\modulus{T}_\epsilon^\alpha}^2.$$
The inner product $\inprod{\cdot,\cdot}$ on $\Hil$ extends to a linear pairing between $\dom(\modulus{T}^\alpha)$ and $\dom(\modulus{T}^\alpha)^\ast$.
For all $u \in \Hil \subset \dom(\modulus{T}^\alpha)^\ast$, we have that $\norm{\modulus{T}_\epsilon^{-\alpha}u} \simeq \norm{u}_{\dom(\modulus{T}^\alpha)^\ast}.$
\end{lemma}
\begin{proof}
It is easy to see that $\modulus{T}_\epsilon > 0$ and invertible.
This operator has \Hinfty-functional calculus  by Corollary 7.1.6 in \cite{Haase} and by Theorem 6.6.9 in \cite{Haase}.
Through interpolation theory, we obtain $\dom(\modulus{T}^\alpha) = \dom(\modulus{T}_\epsilon^\alpha)$ with the norm estimate $\norm{\modulus{T}^\alpha u}^2 + \norm{u}^2 \simeq \norm{\modulus{T}_\epsilon^\alpha}$.

Next, we show that $\inprod{\cdot,\cdot}$ extends to a perfect pairing between $\dom(\modulus{T}^\alpha)$ and $\dom(\modulus{T})^\ast$.
Let $f \in \Hil$ and note that
$$ \norm{f}_{\dom(\modulus{T}^\alpha)^\ast}
= \sup_{0 \neq v \in \dom(\modulus{T}^\alpha)} \frac{\modulus{f(v)}}{\norm{v}_{\dom(\modulus{T}^\alpha)}}
= \sup_{0 \neq v \in \dom(\modulus{T}^\alpha)} \frac{\modulus{\inprod{f,v}}}{\norm{v}_{\dom(\modulus{T}^\alpha)}}.$$
For $v \in \dom(\modulus{T}^\alpha)$, letting $F_v \in \dom(\modulus{T}^\alpha)^{\ast\ast}$ with $\norm{v}_{\dom(\modulus{T}^\alpha)} = \norm{F_v}_{\dom(\modulus{T}^\alpha)^{\ast\ast}}$, we have
\begin{align*}
\norm{F_v}_{\dom(\modulus{T}^\alpha)^{\ast\ast}}
 & = \sup_{0 \neq w \in \dom(\modulus{T})^\ast} \frac{\modulus{F_v(w)}}{\norm{w}_{\dom(\modulus{T}^\alpha)^\ast}} \\
 & =  \sup_{0 \neq w \in \dom(\modulus{T})^\ast} \frac{\modulus{w(v)}}{\norm{w}_{\dom(\modulus{T}^\alpha)^\ast}}  \\
 & = \sup_{0 \neq f \in \Hil} \frac{\modulus{f(v)}}{\norm{f}_{\dom(\modulus{T}^\alpha)^\ast}}                     \\
 & =  \sup_{0 \neq f \in \Hil} \frac{\modulus{\inprod{f,v}}}{\norm{f}_{\dom(\modulus{T}^\alpha)^\ast}}.
\end{align*}
On combining these two calculations, along with the fact that $\dom(|T|)$ is dense in both $\dom(\modulus{T}^\alpha)$ and $\dom(\modulus{T}^\alpha)^\ast$ and contained in $\Hil$, we see that $\inprod{\cdot,\cdot}:\dom(\modulus{T}^\alpha) \times \dom(\modulus{T}^\alpha)^\ast \to \C$, extended by the $\Hil$-inner product, is a perfect pairing.

Lastly, we show that $\norm{\modulus{T}^{-\alpha}_\epsilon u} \simeq \norm{u}_{\dom(\modulus{T}^\alpha)^\ast}$ when $u \in \Hil$.
Using the fact that $\norm{\modulus{T}^\alpha_\epsilon v} \simeq \norm{v}_{\dom(\modulus{T}^\alpha)}$, and since $\modulus{T}^\alpha_\epsilon$ is invertible, for every $v \in \dom(\modulus{T}^\alpha)$ there exists $w \in \Hil$ with $\modulus{T}^\alpha_\epsilon v = w$, we get
\begin{align*}
\norm{\modulus{T}^{-\alpha}_\epsilon u}
 & = \sup_{0 \neq w \in \Hil} \frac{\modulus{\inprod{\modulus{T}^{-\alpha}_\epsilon u,w}}}{\norm{w}}                                                                              \\
 & = \sup_{0 \neq v \in \dom(\modulus{T}^\alpha)}  \frac{\modulus{\inprod{\modulus{T}^{-\alpha}_\epsilon u,\modulus{T}^\alpha_\epsilon v}}}{\norm{\modulus{T}^\alpha_\epsilon v}} \\
 & \simeq \sup_{0 \neq v \in \dom(\modulus{T}^\alpha)}  \frac{\modulus{\inprod{u,v}}}{\norm{v}_{\dom(\modulus{T}^\alpha)}}                                                        \\
 & \simeq \norm{u}_{\dom(\modulus{T}^\alpha)^\ast}.
\qedhere
\end{align*}
\end{proof}

\begin{remark}
\label{Rem:FracPowerBd}
As a consequence of Lemma~\ref{Lem:Frac0}, $\Hil$ itself is a dense subset in $\dom(\modulus{T}^\alpha)^\ast$, and therefore, we see that $\modulus{T}_\epsilon^{-\alpha}:\Hil \to \Hil$ extends boundedly to $\modulus{T}_\epsilon^{-\alpha}: \dom(\modulus{T}^\alpha)^\ast \to \Hil$.
Therefore, for $u \in \dom(\modulus{T}^\alpha)^\ast$, we can write
$$\norm{u}_{\dom(\modulus{T}^\alpha)^\ast} \simeq \norm{\modulus{T}^{-\alpha}_\epsilon u} = \lim_{n \to \infty} \norm{\modulus{T}^{-\alpha}_\epsilon u_n},$$
where $\dom(T) \ni u_n \to u$.
\end{remark}

\begin{lemma}
\label{Lem:ProjDualBdd}
For all $\alpha \geq 0$,  $\chi_{I}(T): \dom(\modulus{T}^\alpha) \to \dom(\modulus{T}^\alpha)$ is a bounded projector.
By duality, this extends to a bounded projector
$\chi_{I}(T): \dom(\modulus{T}^\alpha)^\ast \to \dom(\modulus{T}^\alpha)^\ast$ for $\alpha \in [0,1]$.
In particular, this holds for $\chi^{\pm}(T)$.
\end{lemma}
\begin{proof}

We have already noted that $\chi_{I}(T): \Hil \to \Hil$ boundedly.
Next, from Lemma~\ref{Lem:Frac}, we have that $\norm{u}_{\dom(\modulus{T}^\alpha)} \simeq \norm{ \modulus{T}_\epsilon^\alpha u}$, and thus,
$$
\norm{\chi_{I}(T)u}_{\dom(\modulus{T}^\alpha)} \simeq  \norm{ \modulus{T}_\epsilon^\alpha \chi_{I}(T) u} \simeq \norm{\chi_{I}(T) \modulus{T}_\epsilon^\alpha u} \lesssim \norm{\modulus{T}_\epsilon ^\alpha u} \simeq \norm{u}_{\dom(\modulus{T}^\alpha)}.$$
Now, by duality, we have that $\chi_{I}(T)^\ast: \dom(\modulus{T}^\alpha)^\ast \to \dom(\modulus{T}^\alpha)^\ast$ but since $\chi_{I}(T)^\ast = \chi_{I}(T)$ on $\Hil$ and since $\Hil$ is dense in $\dom(\modulus{T}^\alpha)^\ast$, the conclusion follows.
\end{proof}

Define for $k \in [0,\infty)$,
$$ \OpSobH{k}(T) = \dom(\modulus{T}^k)\quad\text{and}\quad\OpSobH{-k}(T) = \OpSobH{k}(T)^\ast.$$

Recalling $\chi^- = \chi_{(-\infty,0)}$ and $\chi^+ = \chi_{[0,\infty)}$, define the space:
\begin{equation}
\label{Eq:AbsCheck}
\checkH(T) = \chi^{-}(T) \OpSobH{\frac12}(T) \oplus\chi^{+}(T) \OpSobH{-\frac12}(T).
\end{equation}
with norm
\begin{equation}
\norm{u}_{\checkH(T)}^2 := \norm{\chi^{-}(T)u}^2_{\OpSobH{\frac{1}{2}}(T)} + \norm{\chi^+(T)u}_{\OpSobH{-\frac{1}{2}}(T)}.
\end{equation}
Also, define $\hatH(T) := \checkH(-T)$.

\begin{lemma}
\label{Lem:AbsDensity}
\begin{enumerate}[label=(\roman*), labelwidth=0pt, labelindent=2pt, leftmargin=21pt]
\item \label{Lem:AbsDensity:Closedness}
      If $B \subset \OpSobH{\frac12}(T)$ is a subspace that is closed in $\OpSobH{\frac12}(T)$ and in $\checkH(T)$, then $\norm{u}_{\checkH(T)} \simeq \norm{u}_{\OpSobH{\frac12}(T)}$ for $u\in B$ for $u\in B$.
\item \label{Lem:AbsDensity:Density}
      Let $\core \subset \OpSobH{\frac12}(T)$ be a subspace such that $\chi^{\pm}(T) \core \subset \core$.
      Suppose that it dense  in both $\OpSobH{\frac12}(T)$ and in $\Hil$.
      Then, $\core$ is dense in $\checkH(T)$.
\end{enumerate}
\end{lemma}
\begin{proof}
For Assertion~\ref{Lem:AbsDensity:Closedness}, we first show that $\chi^+(T)\OpSobH{\frac12}(T) \embed \chi^+(T)\OpSobH{-\frac12}(T)$ is a dense embedding.
Using Lemma~\ref{Lem:Frac},
$$ \norm{\chi^+(T) u}_{\OpSobH{-\frac12}(T)}
\simeq \norm{ \modulus{T}_{\epsilon}^{-\frac12} \chi^+(T) u}
\simeq \norm{\modulus{T}_{\epsilon}^{-1} \chi^+(T) \modulus{T}_{\epsilon}^{\frac12} u}
\lesssim \norm{\modulus{T}_{\epsilon}^{\frac12} \chi^+(T) u}
\simeq \norm{\chi^+(T)u}_{\OpSobH{\frac12}(T)}.$$
Therefore, for $u \in B$,
\begin{align*}
\norm{u}_{\checkH(T)}
 & \simeq \norm{\chi^-(T)u}_{\OpSobH{\frac12}(T)} + \norm{\chi^+(T)u}_{\OpSobH{-\frac12}(T)}  \\
 & \lesssim \norm{\chi^-(T)u}_{\OpSobH{\frac12}(T)} + \norm{\chi^+(T)u}_{\OpSobH{\frac12}(T)}
\simeq \norm{u}_{\OpSobH{\frac12}(T)}
\simeq \norm{u}_{B}.
\end{align*}
Since $(B, \norm{\cdot}_{\checkH(T)})$ and $(B, \norm{\cdot}_{\OpSobH{\frac12}(T)})$ are Banach spaces, we have that the norms $\norm{\cdot}_{\checkH(T)}$ and $\norm{\cdot}_{\OpSobH{\frac12}(T)}$ are equivalent on $B$.

Now we prove Assertion~\ref{Lem:AbsDensity:Density}.
By Lemma~\ref{Lem:Frac0}, $\core \subset  \OpSobH{\frac12}(T)$   is a dense subspace $\OpSobH{-\frac12}(T)$.
Therefore, there exist $w_n \in \core$ such that $v_n \to \chi^-(T)u$ in $\OpSobH{\frac12}(T)$ and $w_n \to \chi^+(T)u$ in $\OpSobH{-\frac12}(T)$.
It is easy to see that $\chi^-(T)v_n \to \chi^-(T)u$ and $\chi^+(T)w_n \to \chi^+(T)u$.
Moreover, $\chi^-(T)v_n, \chi^+(T)w_n \in \core$ by assumption and therefore, $u_n := \chi^-(T)v_n + \chi^+(T)w_n \in \core$ and it is immediate that $u_n \to u$ in $\checkH(T)$.
\end{proof}

Let $\ext: \Hil \to ([0, \infty) \mapsto \Hil)$ be defined by
$$ (\ext u)(t) := \exp(-t\modulus{T})u.$$
Moreover, define the following space
$$ \Dk{\infty}(T) := \cap_{k=1}^\infty \dom(\modulus{T}^k).$$

\begin{lemma}
\label{Lem:ExtReg}
For all $l,m \geq 0$, The operator $\ext$ extends to a  map $\dom(\modulus{T}^l)^\ast \to ((0,\infty)\mapsto  \dom(\modulus{T}^m))$ linear for each $t > 0$ with $\ran(\ext) \subset \Ck{\infty}((0,\infty); \Dk{\infty}(T))$.
If $u \in \Dk{\infty}(T)$, then
$$\ext u \in \Ck{\infty}([0,\infty); \Dk{\infty}(T)).$$
\end{lemma}
\begin{proof}
For $\epsilon>0$ and $t\in(0,\infty)$, consider $f_k: \R \to \R$ given by $f_k(x) = (x+\epsilon)^k \exp(-t\modulus{x + \epsilon})$.
Clearly, $f_k \in \Ck{0} \cap \Lp{\infty}(\R)$ and therefore, $\norm{f_k(T)v} \lesssim \norm{v}$ for $v \in \Hil$.
For $u \in \dom(\modulus{T}^l)$,
$$\norm{f_k(T)u}_{\dom(\modulus{T}^m)}
= \norm{\modulus{T}_{\epsilon}^{m+k+l} \exp(-t \modulus{T}_{\epsilon}) \modulus{T}_{\epsilon}^{-l}u }
\lesssim \norm{\modulus{T}_{\epsilon}^{-l} u}_{\Hil}
\simeq \norm{u}_{\dom(\modulus{T}^l)^\ast}.$$
By Lemma~\ref{Lem:Frac0}, since $\dom(\modulus{T}^l)$ is dense in $\dom(\modulus{T}^l)^\ast$, we have that $\ext$ extends to a mapping on $\dom(T)^\ast$.
Moreover, we see that  $\exp(-t\modulus{T}_{\epsilon}): \dom(\modulus{T}^l)^\ast \to \dom(\modulus{T}^{m})$ boundedly.
Since $m$ is arbitrary,
$$ \ran(\exp(-t\modulus{T}_{\epsilon})) \bigcap_{k=1}^\infty \dom(\modulus{T}^{k}) = \Dk{\infty}(T).$$
The map $(0,\infty) \ni t \mapsto \exp(-t \modulus{T}_{\epsilon})u$ is, in fact, analytic.

Now, for $u \in \Dk{\infty}(T)$, for each $k \geq 1$,  we have that
$$\partial_t^k (\ext u)(t) = (-1)^k \modulus{T}^k (\ext u)(t) = (-1)^k (\ext \modulus{T}^k u)(t).$$
Since $\modulus{T}^k u \in \Hil$, we obtain that the trace $\lim_{t\to 0}\partial_t^k (\ext u)(t)$ exists.
This shows that $\ext u \in \Ck{\infty}([0,\infty); \Dk{\infty}(T))$.
\end{proof}

\begin{corollary}
\label{Cor:DinfDensity}
For every $l \geq 0$, the subspace $\Dk{\infty}(T)$ is dense in $\dom(\modulus{T}^l)$.
\end{corollary}
\begin{proof}
Note that restricting $\ext$ to $\dom(\modulus{T}^l)$, we obtain that $ \norm{(\ext u)(t) - u}_{\dom(\modulus{T}^l)} \to 0$.
Since by Lemma~\ref{Lem:ExtReg} we have that $(\ext u)(t) \in \Dk{\infty}(T)$, the conclusion follows.
\end{proof}

\begin{corollary}
\label{Cor:Dinf}
The subspace $\Dk{\infty}(T) \subset \checkH(T)$ is dense.
Similarly,  $\Dk{\infty}(T) \subset \hatH(T)$ is dense.
\end{corollary}
\begin{proof}
We note that for each $\dom(\modulus{T}^k)$, we have that it is dense in $\dom(T)$ and also in $\Hil$.
Moreover, $\chi^{\pm}(T) \dom(\modulus{T}^k) \subset \dom(\modulus{T}^k)$.
By application of Lemma~\ref{Lem:Density}, we obtain the conclusion for $\checkH(T)$.
Since $\hatH(T) = \checkH(-T)$, by application of this to $-T$ in place of $T$, we obtain the corresponding conclusion also for $\hatH(T)$.
\end{proof}

\begin{lemma}
\label{Lem:Duality}
The $\Hil$ inner product extends to a perfect pairing $\inprod{\cdot,\cdot}: \checkH(T) \times \hatH(T) \to \C$.
Therefore, this induces an isomorphism $\checkH(T)^\ast \cong \hatH(T)$.
\end{lemma}
\begin{proof}
An argument similar to Proposition~5.1 in \cite{BBan}, using that $\inprod{\cdot,\cdot}$ extends to a linear pairing between $\dom(\modulus{T}^\alpha)$ and $\dom(\modulus{T}^\alpha)^\ast$ by Lemma~\ref{Lem:Frac}, and using the fact that projectors preserve reflexivity,   shows that the $\Hil$ inner product extends to a pairing $\checkH(T) \times \hatH(T) \to \C$ and that these spaces are dual to each other.
\end{proof}

For $r \in \R$, let $T_r := T - r$ and note that by construction, $\OpSobH{\alpha}(T_r) = \OpSobH{\alpha}(T)$  for all $\alpha \in \R$.
In the following, we show that for any two $q, r \in \R$, the spaces $\checkH(T_r) \simeq \checkH(T_q)$ with equivalence of norms.
The argument here proceeds differently to that of \cite{BBan}, since we may have continuous spectrum than just pure point spectrum.
Nevertheless, as the proof illustrates, we can compensate by resorting to the Borel functional calculus.

\begin{proposition}
\label{Prop:CheckIso}
For  $ r \in \R$, we have that $\checkH(T_r) = \checkH(T)$ as sets with equivalence of norms.
\end{proposition}
\begin{proof}
Without loss of generality, we assume that $0 < r$.
Note first that, via functional calculus,
\begin{equation}
\label{E0:ProjForm}
\chi^+(T - r) = \chi_{[r,\infty)]}(T)\quad\text{and}\quad\chi^{-}(T - r) = \chi_{(-\infty, r)}(T).
\end{equation}
Fix $u \in \Dk{\infty}(T)$, which is a dense subset of both $\checkH(T_r)$ and $\checkH(T)$ by Lemma~\ref{Lem:Frac}.
Then, $u = \chi^-(T)u + \chi^+(T)u$ and arguing as in the proof of Proposition~5.2 in \cite{BBan} and using \eqref{E0:ProjForm}  we have that $\chi^-(T_r)u = \chi_{[0,\infty)}(T) u + \chi_{[0,r)}(T) u.$
Also,
\begin{align*}
\norm{u}_{\checkH(T_r)}^2 & = \norm{\chi^{-}(T_r)u}^2_{\OpSobH{\frac{1}{2}}} + \norm{\chi^{+}(T_r)u}^2_{\OpSobH{-\frac{1}{2}}}                                                    \\
                          & \simeq \norm{\chi^-(T) u}_{\OpSobH{\frac{1}{2}}}^2 + \norm{\chi_{[0,r)}(T)u}_{\OpSobH{\frac{1}{2}}}^2 + \norm{\chi^+(T_r)u}^2_{\OpSobH{-\frac{1}{2}}}.
\end{align*}
Moreover, on writing $u = \chi^-(T_r) u + \chi^+(T_r)u$ and via a similar calculation, we find that
$$\norm{u}_{\checkH(T)} \simeq \norm{\chi^-(T)u}^2_{\OpSobH{\frac{1}{2}}} + \norm{\chi_{[0,r)}(T)u}_{\OpSobH{-\frac{1}{2}}}^2 + \norm{\chi^+(T_r)u}_{\OpSobH{-\frac{1}{2}}}^2.$$
Therefore, we are reduced to showing that $\norm{\chi_{[0,r)}(T)u}_{\OpSobH{-\frac{1}{2}}} \simeq \norm{\chi_{[0,r)}(T)u}_{\OpSobH{\frac{1}{2}}}.$

Fixing an $\epsilon > 0$ and using Lemma~\ref{Lem:Frac}, we obtain that
$$\norm{\chi_{[0,r)}(T)u}_{\OpSobH{\frac{1}{2}}} \simeq \norm{\modulus{T}_\epsilon^\frac{1}{2}\chi_{[0,r)}(T)u}  = \norm{\modulus{T}_\epsilon \modulus{T}_\epsilon^{-\frac{1}{2}}\chi_{[0,r)}(T)u}.$$
Now, note that $v = \modulus{T}_\epsilon^{-\frac{1}{2}}\chi_{[0,r)}(T)u \in \ran(\chi_{[0,r)}(T))$, and that $f(t) =( \modulus{t} + \epsilon)\chi_{[0,r)}(t) > \epsilon$ is a bounded Borel function.
Therefore, $f(T) = \modulus{T}_\epsilon\chi_{[0,r)}(T)$ is a bounded self adjoint operator on $\Hil$.
Also, for $w \in \Hil$,
\begin{multline*}
\inprod{f(T)w,w} = \inprod{\modulus{T}_\epsilon \chi_{[0,r)}(T)w,w} = \inprod{\chi_{[0,r)}(T) \modulus{T}_\epsilon \chi_{[0,r)}(T)w,w} \\
= \inprod{\modulus{T}_\epsilon \chi_{[0,r)}(T)w, \chi_{[0,r)}(T)w} \geq \epsilon  \norm{ \chi_{[0,r)}(T)w}^2.
\end{multline*}
Using the Cauchy-Schwarz inequality on $\inprod{f(T)w,w}$ and on noting that $f(T) \chi_{[0,r)}(T)w = f(T) w$,  we obtain
$\epsilon \norm{\chi_{[0,r)}(T) w} \leq \norm{f(T)w} \leq \norm{f(T)}_{\Hil\to\Hil} \norm{\chi_{[0,r)}(T)w}$.
Therefore, on noting that $w = v = \chi_{[0,r)}(T)v$,
$$\norm{\modulus{T}_\epsilon \modulus{T}_\epsilon^{-\frac{1}{2}}\chi_{[0,r)}(T)u} \simeq \norm{ \modulus{T}_\epsilon^{-\frac{1}{2}}\chi_{[0,r)}(T)u} \simeq \norm{\chi_{[0,r)}(T)u}_{\OpSobH{-\frac{1}{2}}},$$
which is the required estimate to establish the claim.
\end{proof}

\begin{proposition}
\label{Prop:IntSmoothAbstract}
For any $\alpha \ge 0$ and any Borel set $S \subset \R$ that is bounded,
$$\chi_{S}(T) \dom(\modulus{T}^{\alpha})^\ast  \subset \bigcap_{k=0}^\infty\dom(T^{2k}) .$$
\end{proposition}
\begin{proof}
Drawing inspiration from the proof of Proposition~\ref{Prop:CheckIso}, for $k \in \N$ define $f_k: \R\to \R$ by
$$ f_k(x) := x^{2k} \chi_{S}(x).$$
Since $S$ is bounded, clearly there exists a constant $C_{k,S} < \infty$ dependent  on $k$ and $I$ such that $\modulus{f_k(x)}  \leq C_{k,S}$.
Therefore, by the Borel functional calculus, we have that
$$ \norm{f_k(T)u} \lesssim \norm{u}$$
for all $u \in \Hil$.
However, $f_k(T)u = T^{2k} \chi_{S}(T)u$, which means that $\chi_{S}(T)u \in \dom(T^{2k})$.

Now, from Remark~\ref{Rem:FracPowerBd}, we have that $\modulus{T}^{-\alpha}_\epsilon: \dom(\modulus{T}^\alpha)^\ast \to \Hil$ is bounded.
Moreover, by functional calculus, $\chi_{S}(T) \modulus{T}^{-\alpha}_{\epsilon}u =  \modulus{T}^{-\alpha}_{\epsilon} \chi_{S}(T) u$ and hence, for $u \in \dom(\modulus{T}^\alpha)^\ast$,
\begin{align*}
\norm{\modulus{T}^{2k-\alpha}_\epsilon \chi_{S}(T) u}
 & = \norm{\modulus{T}^{2k}_\epsilon \chi_{S}(T) \modulus{T}^{-\alpha}_\epsilon u }                                                                \\
 & \simeq \norm{\modulus{T}^{2k} \chi_{S}(T) \chi_{S}(T) \modulus{T}^{-\alpha}_\epsilon u} + \norm{  \chi_{S}(T) \modulus{T}^{-\alpha}_\epsilon u} \\
 & = \norm{f_k(T)\modulus{T}^{-\alpha}_\epsilon \chi_{S}(T) u} + \norm{ \modulus{T}^{-\alpha}_\epsilon \chi_{S}(T) u}                              \\
 & \lesssim \norm{ \modulus{T}^{-\alpha}_\epsilon \chi_{S}(T) u}                                                                                   \\
 & \simeq \norm{\chi_{S}(T)u}_{\dom(\modulus{T}^\alpha)^\ast},
\end{align*}
where in the second norm equivalence, we have used $\chi_{S}(T)^2 = \chi_{S}(T)$.
Therefore,
$$\chi_{S}(T)u \in \dom(\modulus{T}^{2k-\alpha}) \subset   \dom(\modulus{T}^{2k-2}) = \dom(T^{2(k-1)})$$
when $k > 1$.
\end{proof}

\subsection{Abstract Rellich theory}

\begin{theorem}
\label{Thm:AbsCpctEmbed}
Let $T$ be a selfadjoint operator on a separable Hilbert space $\Hil$.
Then the following are equivalent.
\begin{enumerate}[label=(\roman*), labelwidth=0pt, labelindent=2pt, leftmargin=21pt]
\item \label{It:AbsCpctEmbed:1}
$T$ has discrete spectrum.
\item \label{It:AbsCpctEmbed:2} 
The embedding $\dom((1+T^2)^s) \embed \dom((1+T^2)^t)$ is compact for some $s, t \in \R$ with $s > t\ge0$.
\item  \label{It:AbsCpctEmbed:3} 
The embedding $\dom((1+T^2)^s) \embed \dom((1+T^2)^t)$ is compact for all $s, t \in \R$ with $s > t\ge0$.
\end{enumerate}
\end{theorem}

\begin{proof}
Without loss of generality, we assume that $T$ unbounded.
Indeed, if $T$ is bounded, then each of the assertions \ref{It:AbsCpctEmbed:1}--\ref{It:AbsCpctEmbed:3} is equivalent to $\Hil$ being finite dimensional.

It is immediate that \ref{It:AbsCpctEmbed:3} implies \ref{It:AbsCpctEmbed:2}.
We show that \ref{It:AbsCpctEmbed:2} implies \ref{It:AbsCpctEmbed:1}.
So, assume \ref{It:AbsCpctEmbed:2} and note that 
$$ 
\dom((1+T^2)^s) \xhookrightarrow{\mathrm{compact}} \dom((1+T^2)^t) \xhookrightarrow{\mathrm{continuous}} \Hil,
$$
and therefore we have that $(1+T^2)^{-s}:\Hil \to \Hil$ is a compact map.
Hence $(1+T^2)^{-s}$ has discrete spectrum except for the accumulation point $0$.
Since $s>0$ it follows that $1+T^2$ and hence $T$ has discrete spectrum.

It remains to show that \ref{It:AbsCpctEmbed:1} implies \ref{It:AbsCpctEmbed:3}.
Let $T$ have discrete spectrum.
To show that $\dom((1+T^2)^s) \embed \dom((1+T^2)^t)$ is compact, let $u_k\in \dom((1+T^2)^s)$ be a bounded sequence, 
$$
\|u_k\|_{\dom((1 +T^2)^s)} \leq C.
$$
We need to find a subsequence which converges in $\dom((1 +T^2)^t)$.
Let $\lambda_j$ be the eigenvalues of $T$ and $\phi_j$ its orthonormalised eigenvectors.
We write $u_k = \sum_j u_{k,j}\phi_j$.
For each fixed $j$ we have
$$
(1+\lambda_j^2)^{2s} |u_{k,j}|^2 \le \|u_k\|_{\dom((1 +T^2)^s)}^2 \le C^2
$$
and hence 
$$
|u_{k,j}| \le C .
$$
By a diagonal argument, we can pass to a subsequence, again denoted $(u_k)$, such that for each $j$
$$
u_{k,j} \xrightarrow{k\to\infty} w_j .
$$

We show that $(u_k)$ is a Cauchy sequence in $\dom((1 +T^2)^t)$.
Let $\epsilon>0$.
Since the spectrum is discrete and $T$ is unbounded there exists $j_0$ such that 
$$
|\lambda_j| \ge \frac1\epsilon 
$$
for all $j> j_0$.
For these $j$ we then have
$$
(1+\lambda_j^2)^{t-s}
\le 
|\lambda_j|^{2(t-s)}
\le
\epsilon^{2(s-t)}.
$$
Therefore,
\begin{align*}
\sum_{j> j_0} (1+\lambda_j^2)^{2t} |u_{k,j}-u_{\ell,j}|^2
&=
\sum_{j> j_0} (1+\lambda_j^2)^{2(t-s)}(1+\lambda_j^2)^{2s} |u_{k,j}-u_{\ell,j}|^2 \\
&\le
\epsilon^{4(s-t)}\sum_{j\le j_0}(1+\lambda_j^2)^{2s} |u_{k,j}-u_{\ell,j}|^2 \\
&\le
\epsilon^{4(s-t)} \|u_k-u_\ell\|_{\dom((1 +T^2)^s)}^2 \\
&\le
\epsilon^{4(s-t)} \cdot 4C^2.
\end{align*}
Choose $m$ so large that such that we have for all (finitely many) $j\le j_0$
$$
(1+\lambda_j^2)^{2t} |u_{k,j}-u_{\ell,j}|^2 \le \frac{\epsilon}{j_0}
$$
whenever $k,\ell\ge m$.
Then we find for such $k$ and $\ell$ that
\begin{align*}
\|u_k-u_\ell\|_{\dom((1 +T^2)^t)}^2
&=
\sum_{j\le j_0}(1+\lambda_j^2)^{2t} |u_{k,j}-u_{\ell,j}|^2
+ \sum_{j> j_0}(1+\lambda_j^2)^{2t} |u_{k,j}-u_{\ell,j}|^2 \\
&\le
\epsilon + \epsilon^{4(s-t)} \cdot 4C^2.
\end{align*}
This shows that $(u_k)$ is a Cauchy sequence in $\dom((1 +T^2)^{\frac t 2})$ and hence converges.
\end{proof} 

\begin{remark}
A simple calculation shows that $\dom(\sqrt{1 + T^2}) = \dom(\modulus{T})$.
Thus $\dom( (1+ T^2)^{\frac {\alpha}{2}}) = \dom(\modulus{T}^\alpha)$ for all $\alpha \ge0$.
Negative powers of $|T|$ may not be defined as $T$ may have nontrivial kernel.
\end{remark}

\subsection{Involutions} 
Let $\Xi:\Hil \to \Hil$ be a bounded \emph{involution}, i.e., a bounded operator with $\Xi^2 = I$.
We recall some basic facts, in particular when such a $\Xi$ interacts with a selfadjoint operator $T$.

\begin{lemma}
\label{Lem:Idem}
Let $\Xi: \Hil \to \Hil$ be a bounded involution. Then:
\begin{enumerate}[label=(\roman*), labelwidth=0pt, labelindent=2pt, leftmargin=21pt]
\item \label{Itm:Lem:Idem:1}
      $\Xi^\ast:\Hil \to \Hil$ is also a bounded involution.
\item  \label{Itm:Lem:Idem:2}
      $\spec(\Xi) = \set{\pm 1}$.
      \item\label{Itm:Lem:Idem:3}
      $P_{\pm}= \frac12 (I \pm \Xi)$ define bounded projectors to eigenspaces corresponding to $\pm 1$ so that inducing a splitting $\Hil = \Hil_+ \oplus \Hil_-$ where $\Hil_{\pm} = P_{\pm} \Hil$.
\end{enumerate}
\end{lemma}
\begin{proof}
The first assertion~\ref{Itm:Lem:Idem:1} is immediate.
For \ref{Itm:Lem:Idem:2}, note that $\Xi^2 = 1 \iff (\Xi - 1)(\Xi +1) =0$ and therefore, we get $\set{\pm 1} \subset \spec(\Xi)$.
The reverse inclusion follows from noting that $(\Xi - \lambda)^{-1} = (1 - \lambda^2)^{-1}(\Xi + \lambda)$.
The last assertion~\ref{Itm:Lem:Idem:3} is readily verified.
\end{proof}

Due to \ref{Itm:Lem:Idem:1}, the assertions~\ref{Itm:Lem:Idem:1} and \ref{Itm:Lem:Idem:2} are valid on replacing $\Xi$ by  $\Xi^*$.

\begin{proposition}
\label{Prop:IdemProjs}
Let $T:\Hil \to \Hil$ be a selfadjoint operator and $\Xi: \Hil \to \Hil$  a bounded involution satisfying $\Xi T = - T\Xi$.
Then, $\Xi, \Xi^\ast: \dom(\modulus{T}^\alpha)^\ast \to  \dom(\modulus{T}^\alpha)^\ast$ boundedly and restrict to bounded operators $\Xi, \Xi^\ast: \dom(\modulus{T}^\alpha) \to  \dom(\modulus{T}^\alpha)$ for all $\alpha \in [0, 1]$.
\end{proposition}
\begin{proof}
We first prove that $\Xi:\dom(\modulus{T}) \to \dom(\modulus{T})$ boundedly.
For this, note that
\begin{align*}
\norm{\Xi u}_{\dom(\modulus{T})}
 & \simeq \norm{\modulus{T}\Xi u} + \norm{u}
= \norm{T \Xi u} + \norm{u}                  \\
 & = \norm{- \Xi T u} + \norm{u}
\lesssim \norm{Tu} + \norm{u}
\simeq \norm{u}_{\dom(\modulus{T})}.
\end{align*}

By taking adjoints in $\Xi T = - T \Xi$, we obtain that $\Xi^\ast T = - T \Xi^\ast$.
Therefore, replacing the previous argument with $\Xi^\ast$ in place of $\Xi$, we obtain that $\Xi^\ast: \dom(\modulus{T}) \to \dom(\modulus{T})$.
Through run of the mill interpolation for selfadjoint operators, we obtain that $\Xi, \Xi^\ast: \dom(\modulus{T}^\alpha) \to \dom{\modulus{T}^\alpha}$ for $\alpha \in [0,1]$.
On dualising this, we obtain the assertions.
\end{proof}

As a direct consequence, we obtain the following corollary.
\begin{corollary}
\label{Cor:IdemProjs}
Under the hypothesis of Proposition~\ref{Prop:IdemProjs}, the projectors $P_{\pm}, P_{\pm}^\ast: \Hil \to \Hil$ extend to bounded projectors $P_{\pm}, P_{\pm}^\ast: \dom(\modulus{T}^\alpha)^\ast \to \dom(\modulus{T}^\alpha)^\ast$ and restrict to bounded projectors $P_{\pm}, P_{\pm}^\ast: \dom(\modulus{T}^\alpha) \to \dom(\modulus{T}^\alpha)$ for all $\alpha \in [0,1]$.
\end{corollary}


\section*{Legal notes}
\subsection*{Data availability}
Data sharing is not applicable to this article as no datasets were generated or analysed during the current study.

\subsection*{Conflict of interest}
There is no conflict of interest.

\setlength{\parskip}{0pt}
\end{document}